\documentclass{cedram-afst}

\usepackage{bm}
\usepackage[T1]{fontenc}
\usepackage[english]{babel}
\usepackage[utf8]{inputenc}
\usepackage[mathscr]{euscript}
\usepackage{textcomp,multicol,enumitem}
\usepackage[papersize={180mm,240mm},top=2cm, bottom=2cm, left=2cm , right=2cm]{geometry}

% ...

% ...

\title
[Algebraicity modulo p of generalized hypergeometric series $_nF_{n-1}$]
{Algebraicity modulo p of generalized hypergeometric series $_nF_{n-1}$}

\author{\firstname{Daniel} \middlename{} \lastname{vargas-Montoya}}

\urladdr{Daniel Vargas-Montoya, Institute of Mathematics of the Polish Academy of Sciences, \'Sniadeckich 8, 00-656 Warsaw, Poland}

\thanks{This work was supported by the National Science Centre of Poland (NCN), grant UMO-2020/39/B/ST1/00940.}

\email{\textit{Email address:} devargasmontoya@impan.pl}

  
\altkeywords{}


\begin{document}
%% Abstracts must be placed before \maketitle
\begin{abstract}
Let $f(z)={}_nF_{n-1}(\bm\alpha,\bm\beta)$ be the hypergeometric series with parameters $\bm\alpha=(\alpha_1,\ldots,\alpha_n)$ and $\bm\beta=(\beta_1,\ldots,\beta_{n-1},1)$ in $(\mathbb{Q}\cap(0,1])^n$, let $d_{\bm\alpha,\bm\beta}$ be the least common multiple of the denominators of $\alpha_1,\ldots,\alpha_n$, $\beta_1,\ldots,\beta_{n-1}$ written in lowest form and let $p$ be a prime number such that $p$ does not divide $d_{\bm\alpha,\bm\beta}$ and $f(z)\in\mathbb{Z}_{(p)}[[z]]$. Recently in \cite{vmsff}, it was shown that if for all $i,j\in\{1,\ldots,n\}$, $\alpha_i-\beta_j\notin\mathbb{Z}$ then the reduction of $f(z)$ modulo $p$ is algebraic over $\mathbb{F}_p(z)$. A standard way to measure the complexity of an algebraic power series is to estimate its degree and its height. In this work, we prove that if $p>2d_{\bm\alpha,\bm\beta}$ then there is a nonzero polynomial $P_p(Y)\in\mathbb{F}_p(z)[Y]$ having  degree at most  $p^{2^n\varphi(d_{\bm\alpha,\bm\beta})}$ and height at most $5^n(n+1)!p^{2^{n}\varphi({d_{\bm\alpha,\bm\beta})}}$ such that $P_p(f(z)\bmod p)=0$, where $\varphi$ is the Euler's totient function. Furthermore, our method of proof provides us a way to make an explicit construction of the polynomial $P_p(Y)$. We illustrate this construction by applying it to some explicit hypergeometric series.
\end{abstract}

%% French abstract
%\begin{altabstract}
%Ceci est le r\'esum\'e fran\c cais.
%\end{altabstract}

\setcounter{tocdepth}{1}

\maketitle

\tableofcontents

\section{Introduction}

Let $\bm\alpha=(\alpha_1,\ldots,\alpha_n)$ and $\bm\beta=(\beta_1,\ldots,\beta_{n-1},1)$ be in $(\mathbb{Q}\setminus\mathbb{Z}_{\leq0})^n$. The generalized hypergeometric series with parameters $\bm\alpha$, $\bm\beta$ is the power series given by $$_nF_{n-1}(\bm\alpha,\bm\beta;z)=\sum_{i\geq0}\mathcal{Q}_{\bm\alpha,\bm\beta}(i)z^{i}\in\mathbb{Q}[[z]]\text{ with } \mathcal{Q}_{\bm\alpha,\bm\beta}(i)=\frac{(\alpha_1)_i\cdots(\alpha_n)_i}{(\beta_1)_i\cdots(\beta_{n-1})_ii!},$$
where for a real number $x$ and a nonnegative integer $i$, $(x)_i$ is the Pochhammer symbol, that is, $(x)_0=1$ and $(x)_i=x(x+1)\cdots(x+i-1)$ for $i>0$. We denote by $d_{\bm\alpha,\bm\beta}$ the least common multiple of the denominators of $\alpha_1,\ldots,\alpha_n$ and $\beta_1,\ldots, \beta_{n-1}$ written in lowest form. It is well-known that $_nF_{n-1}(\bm\alpha,\bm\beta;z)$ is a solution of the hypergeometric operator $$\mathcal{H}(\bm\alpha,\bm\beta)=\prod_{i=1}^{n}(\delta+\beta_i-1)-z\prod_{i=1}^n(\delta+\alpha_i),\text{ with }\delta=z\frac{d}{dz}.$$

We recall that for any field $K$, the power series  $h(z)\in K[[z]]$ is an algebraic power series over $K(z)$ if there exists a nonzero polynomial $P(Y)\in K(z)[Y]$ such that $P(h(z))=0$. Given a prime number $p$, we denote by $\mathbb{Z}_{(p)}$ the localization of $\mathbb{Z}$ at ideal $(p)$. That is, $\mathbb{Z}_{(p)}$ is the set of rational numbers $a/b$ written in lowest form such that $p$ does not divide $b$. This ring is a local ring whose maximal ideal is $(p)\mathbb{Z}_{(p)}$ and its residue field is  the field with $p$ elements, which is denoted by $\mathbb{F}_p$. Given a power series $f(z)=\sum_{i\geq0}a(i)\in\mathbb{Z}_{(p)}[[z]]$, the reduction of $f$ modulo $p$ is $f(z)\bmod p:= \sum_{i\geq0}(a(i)\bmod p)z^{i}\in\mathbb{F}_p[[z]]$. The power series $f(z)$ is said to be \emph{algebraic modulo $p$} if $f(z)\bmod p$ is an algebraic power series over $\mathbb{F}_p(z)$. A usual way to measure the complexity of an algebraic power series is to estimate its \emph{degree} and its \emph{height}.
\begin{defi}
Let $K$ be a field and let $a(z)=s(z)/t(z)$ be in $K(z)$ written in lowest form. The height of $a(z)$ is equal to $\max\{deg(s(z)), deg(t(z))\}$. Let $P(Y)=\sum_{i=0}^{m}a_i(z)Y^i$ be in $K(z)[Y]$ such that $a_m(z)$ is not zero. The degree of $P$ is $m$ and the height of $P$ is the maximum of the heights of $a_0(z),\ldots, a_{m}(z)$.
\end{defi}

We have shown in \cite[Theorem 1.2]{vmsff} the following result. Let $\mathcal{S}$ be an infinite set of prime numbers $p$ such that $p$ does not divide $d_{\bm\alpha,\bm\beta}$ and $_nF_{n-1}(\bm\alpha,\bm\beta;z)\in\mathbb{Z}_{(p)}[[z]]$. We proved that if, for all $i,j\in\{1,\ldots, n\}$, $\alpha_i-\beta_j\notin\mathbb{Z}$ then, for all $p\in\mathcal{S}$, $_nF_{n-1}(\bm\alpha,\bm\beta;z)$ is algebraic modulo $p$. We also established that $_nF_{n-1}(\bm\alpha,\bm\beta;z)\bmod p$  has degree at most $p^{n^2\varphi(d_{\bm\alpha,\bm\beta})}$, where $\varphi$ is the Euler's totient function. However, the result obtained in \cite{vmsff} does not offer any information about the height. The main result of this work shows that if  $\bm\alpha$ and $\bm\beta$ belong to  $(\mathbb{Q}\cap(0,1])^n$ then,  for all $p\in\mathcal{S}$ satisfying $p>2d_{\bm\alpha,\bm\beta}$, there is a nonzero polynomial $P_p(Y)\in\mathbb{F}_p(z)[Y]$ having  degree at most  $p^{2^t l}$ and height at most $5^t(t+1)!p^{2^{t}l}$ such that $P_p(_nF_{n-1}(\bm\alpha,\bm\beta;z)\bmod p)=0$, where $t\leq n$ and $l$ is the order of $p$ in $(\mathbb{Z}/d_{\bm\alpha,\bm\beta}\mathbb{Z})^*$.  Further, the advantage of the present method is that it gives an explicit way to construct the polynomial $P_p(Y)$.

\subsection{Main result} In order to state our main result, Theorem~\ref{theo_main_infinite_p}, we have to introduce some notations. Let $p$ be a prime number such that $p$ does not divide $d_{\bm\alpha,\bm\beta}$. Then $\mathcal{H}(\bm\alpha,\bm\beta)\in\mathbb{Z}_{(p)}[z][\delta]$. In particular, we can reduce $\mathcal{H}(\bm\alpha,\bm\beta)$ modulo $p$ and we denote by  $\mathcal{H}(\bm\alpha,\bm\beta,p)$ its reduction modulo $p$. That is, $$\mathcal{H}(\bm\alpha,\bm\beta,p):=\prod_{i=1}^{n}(\delta+(\beta_i-1)\bmod p)-z\prod_{i=1}^n(\delta+\alpha_i\bmod p)\in\mathbb{F}_p[z][\delta].$$  An element of the set  $\{0,1-\beta_1\bmod p,\ldots, 1-\beta_{n-1}\bmod p\}$ will be called an exponent at zero of $\mathcal{H}(\bm\alpha,\bm\beta,p)$. Consider the following set: $$E_{\bm\alpha,\bm\beta,p}=\{r\in\{0,1,\ldots,p-1\}: r\bmod p \text{ is an exponent at zero of }\mathcal{H}(\bm\alpha,\bm\beta,p)\}.$$
Given a finite set $E$, by $\#E$ we mean the number of elements of $E$.  It is clear that $0\in E_{\bm\alpha,\bm\beta,p}$ and that $\#E_{\bm\alpha,\bm\beta,p}\leq n$. As usual, $v_p:\mathbb{Q}\rightarrow\mathbb{Z}$ denotes the $p$-adic valuation map. We define the following set:
$$S_{\bm\alpha,\bm\beta,p}=\{r\in\{0,1,\ldots,p-1\}: r\in E_{\bm\alpha,\bm\beta,p} \text{ and }v_p\left(\mathcal{Q}_{\bm\alpha,\bm\beta}(r)\right)=0\}.$$
 The set $S_{\bm\alpha,\bm\beta,p}$ is not empty because $0\in S_{\bm\alpha,\bm\beta,p}$, and $\#S_{\bm\alpha,\bm\beta,p}\leq n$ since $S_{\bm\alpha,\bm\beta,p}\subset E_{\bm\alpha,\bm\beta,p}$.

Let us recall the definition of the map $\mathfrak{D}_p:\mathbb{Z}_{(p)}\rightarrow\mathbb{Z}_{(p)}$ introduced by Dwork in \cite[Chap.~8]{Dworklectures}. The map $\mathfrak{D}_p:\mathbb{Z}_{(p)}\rightarrow\mathbb{Z}_{(p)}$ is such that, for  every $\gamma$ in $\mathbb{Z}_{(p)}$, $\mathfrak{D}_p(\gamma)$ is the unique element in $\mathbb{Z}_{(p)}$ such that $p\mathfrak{D}_p(\gamma)-\gamma$ belongs to $\{0,\ldots, p-1\}$. In \cite[Chap. 8]{Dworklectures} this map is denoted by $\gamma\mapsto\gamma'$.
 For $\bm\gamma=(\gamma_1,\ldots,\gamma_n)\in\mathbb{Z}_{(p)}^n$ we write $\mathfrak{D}_p(\bm\gamma)$ for $(\mathfrak{D}_p(\gamma_1),\ldots,\mathfrak{D}_p(\gamma_n))$. For all integers $m\geq1$, $\mathfrak{D}_p^m$ is the $m$-th composition of $\mathfrak{D}_p$ with itself and  $\mathfrak{D}_p^0$ is identity map on $\mathbb{Z}_{(p)}$.
 
 \begin{rema}
Let $\bm\alpha=(\alpha_1,\ldots,\alpha_n)$ and $\bm\beta=(\beta_1,\ldots,\beta_{n-1},1)$ be in $\mathbb{Q}^n$ and let $p$ a prime number such that $p$ does not divide $d_{\bm\alpha,\bm\beta}$. Then, $\bm\alpha$ and $\bm\beta$ belong to $\mathbb{Z}_{(p)}^n$ and for this reason, for all integers $m\geq0$, the differential operator  $\mathcal{H}(\mathfrak{D}_p^m(\bm\alpha),\mathfrak{D}_p^m(\bm\beta))$ belongs to $\mathbb{Z}_{(p)}[z][\delta]$. Thus, for all integers $m\geq0$, the sets $E_{\mathfrak{D}_p^m(\bm\alpha),\mathfrak{D}_p^m(\bm\beta),p}$, $S_{\mathfrak{D}_p^m(\bm\alpha),\mathfrak{D}_p^m(\bm\beta),p}$ are well-defined.
\end{rema}

We are now ready to state our main result:

\begin{theo}\label{theo_main_infinite_p}
Let $\bm\alpha=(\alpha_1,\ldots,\alpha_n)$ and $\bm\beta=(\beta_1,\ldots,\beta_{n-1},1)$ be in $(\mathbb{Q}\cap(0,1])^n$,  let $f(z)$ be the hypergeometric series $_nF_{n-1}(\bm\alpha,\bm\beta;z)$, let $p$ be  prime number such that $p>2d_{\bm\alpha,\bm\beta}$ and $f(z)\in\mathbb{Z}_{(p)}[[z]]$, and let $(\mathbb{Z}/d_{\bm\alpha,\bm\beta}\mathbb{Z})^*$ be the unit group of $\mathbb{Z}/d_{\bm\alpha,\bm\beta}\mathbb{Z}$.  Suppose that, for all $i,j\in\{1,\ldots, n\}$, $\alpha_i-\beta_j\notin\mathbb{Z}$. Then there is a nonzero polynomial $P_p(Y)\in\mathbb{F}_p(z)[Y]$ having  degree at most  $p^{2^n\varphi(d_{\bm\alpha,\bm\beta})}$ and height at most $5^n(n+1)!p^{2^{n}\varphi({d_{\bm\alpha,\bm\beta})}}$ such that $P_p(f(z)\bmod p)=0.$
Moreover, if $l$ is the order of $p$ in $(\mathbb{Z}/d_{\bm\alpha,\bm\beta}\mathbb{Z})^*$ then the following assertions hold:
 \begin{enumerate}[label=(\arabic*)]
\item if $1=\#S_{\mathfrak{D}_p^{l-1}(\bm\alpha),\mathfrak{D}_p^{l-1}(\bm\beta),p}$ then $$P_p(Y)=Y-Q_1(z)Y^{p^l},$$
where $Q_1(z)$ belongs to $\mathbb{F}_{p}[z]$ and has degree less than $p^l$;
 
\smallskip

\item if $2=\#S_{\mathfrak{D}_p^{l-1}(\bm\alpha),\mathfrak{D}_p^{l-1}(\bm\beta),p}$ then $$P_p(Y)=Y-Q_1(z)Y^{p^l}-Q_{2}(z)Y^{p^{2l}},$$
and the height of $P_p(Y)$ is less than $2p^{2l}$;

\smallskip 

\item if $2<\#S_{\mathfrak{D}_p^{l-1}(\bm\alpha),\mathfrak{D}_p^{l-1}(\bm\beta),p}=t+1$ then $$P_p(Y)=Y-Q_1(z)Y^{p^l}-Q_2(z)Y^{p^{2l}}-Q_{3}(z)Y^{p^{3l}}-\cdots-Q_{2^t}(z)Y^{p^{2^{t}l}},$$
and the height of $P_p(Y)$ is less than  $5^t(t+1)!p^{2^{t}l}$.

\end{enumerate}
\end{theo}
Let us make a few comments. In these comments we keep the notations used in the statement of Theorem~\ref{theo_main_infinite_p}.

\textbullet\quad To prove Theorem~\ref{theo_main_infinite_p} it is sufficient to show that the assertions (1), (2) and (3) hold because $l\leq\varphi(d_{\bm\alpha,\bm\beta})=\#(\mathbb{Z}/d_{\bm\alpha,\bm\beta}\mathbb{Z})^*$ and $\#S_{\mathfrak{D}_p^{l-1}(\bm\alpha),\mathfrak{D}_p^{l-1}(\bm\beta),p}\leq n$. 

\textbullet\quad The method of proof of Theorem~\ref{theo_main_infinite_p} provides us a way to make an explicit construction of the polynomial $P_p(Y)\in\mathbb{F}_p(z)[Y]$. In Section \ref{sec_construction}, we show how to construct the polynomial $P_p(Y)$ and we illustrate this construction by applying it to some hypergeometric series.

\textbullet\quad The conclusion of the assertion (1) of Theorem~\ref{theo_main_infinite_p} is to equivalent to saying that the hypergeometric series  $_{n}F_{n-1}(\bm\alpha,\bm\beta,z)$ satisfies the \emph{$p^l$-Lucas property}. We say that a power series $f(z)=\sum_{i\geq0}a(i)z^i\in\mathbb{Q}[[z]]$ satisfies the $p^l$-Lucas property if $f(z)\in\mathbb{Z}_{(p)}[[z]]$, $a(0)=1$ and, for all integers $m\geq0$ and for all $r\in\{0,\ldots,p^l-1\}$, $a(mp^l+r)\equiv a(m)a(r)\bmod p.$ From \cite[Proposition~4.8]{ABD}, it follows that $f(z)\in 1+z\mathbb{Z}_{(p)}[[z]]$ satisfies the $p^l$-Lucas property if and only if $f\equiv A_p(z)f^{p^l}\bmod p$, where $A_p$ is a polynomial with coefficients in $\mathbb{Z}_{(p)}$ having degree less than $p^l$. 

\textbullet\quad As we have already said, from Theorem 1.2 of \cite{vmsff} it follows that ${}_nF_{n-1}(\bm\alpha,\bm\beta;z)$ is algebraic modulo $p$ and the degree of its reduction modulo $p$ is at most $p^{n^2\varphi(d_{\bm\alpha,\bm\beta})}$. The proof of this result relies on the fact that  $\mathcal{H}(\bm\alpha,\bm\beta)$ has a strong Frobenius structure for all $p\in\mathcal{S}$ with period $\varphi(d_{\bm\alpha,\bm\beta})$. Nevertheless, the approach used in this work to prove Theorem~\ref{theo_main_infinite_p} does not use the existence of strong Frobenius structure.\footnote{The existence of strong Frobenius structure of  $\mathcal{H}(\bm\alpha,\bm\beta)$ is a directly consequence of a result due to Crew  \cite{RC17}. The approach used by Crew is via $p$-adic cohomology. Nevertheless, we also obtain this result in \cite[Theorem 6.2]{vmsff} by using an elementary approach based on ideas of Christol \cite{Gillesfacteurs} and Saliner \cite{S87}.}  

\textbullet\quad  In Section~\ref{sec_exam} we will compare  through some hypergeometric series the estimate $p^{2^tl}$ given by Theorem~\ref{theo_main_infinite_p} and the estimate $p^{n^2\varphi(d_{\bm\alpha,\bm\beta})}$ given by Theorem 1.2 of \cite{vmsff}. As we will see, for these particular examples, the estimate $p^{2^tl}$ is much finer than the estimate $p^{n^2\varphi(d_{\bm\alpha,\bm\beta})}$.

\subsection{Structure of proof}

The proof of Theorem~\ref{theo_main_infinite_p} is based on Theorem~\ref{theo_main}. The latter one is derived from Propositions~\ref{prop_third_step} and \ref{lemm_system}. In section~\ref{sec_proof_lemm_system}, Proposition~\ref{lemm_system} is proved. Proposition~\ref{prop_third_step} will be proved in Section~\ref{sec_proof_prop_third_step} and its proof relies on Lemmas~\ref{lemm_aux} and \ref{lemm_second_step}. The proof of Lemma~\ref{lemm_aux} is given in Section~\ref{sec_proof_lemm_aux}. Finally, in Section~\ref{sec_proof_lemm_second_step} we prove Lemma~\ref{lemm_second_step}. Nevertheless, the proof of this lemma depends on Lemma~\ref{lemm_first_step}, which is also proved in Section~\ref{sec_proof_lemm_second_step}. Lemma~\ref{lemm_first_step} is, in fact, the main ingredient of this work and its proof is based essentially on two facts. The first one deals with some $p$-adic properties of the sequence $\{\mathcal{Q}_{\bm\alpha,\bm\beta}(j)\}_{j\geq0}$. Sections~\ref{sec_aux_1} and \ref{sec_aux_2} are devoted to studying these $p$-adic properties. The second fact is the equality 
\begin{equation}\label{eq_fundamental}
I(j)\mathcal{Q}_{\bm\alpha,\bm\beta}(j)=\mathcal{Q}_{\bm\alpha,\bm\beta}(j-1)T(j-1)
\end{equation}
 for all integers $j\geq1$, where $I(j)=\prod_{i=1}^n(j+\beta_i-1)$ and $T(j)=\prod_{i=1}^n(j+\alpha_i)$. The Equality~\eqref{eq_fundamental} is equivalent to the fact that ${}_nF_{n-1}(\bm\alpha,\bm\beta;z)$ is solution of $\mathcal{H}(\bm\alpha,\bm\beta)$.

\subsection{ Reduction modulo $p$ of generalized hypergeometric series }\label{subsec_diag_lucas} 

In oder to apply Theorem~\ref{theo_main_infinite_p}, a natural question is to determine when it is possible to reduce a hypergeometric series modulo $p$. In this direction,  an interesting class of hypergeometric series is the class of \emph{globally bounded} hypergeometric series. We say that the  hypergeometric series $_nF_{n-1}(\bm\alpha,\bm\beta;z)$ is globally bounded if there is $c\in\mathbb{Q}\setminus\{0\}$ such that $_nF_{n-1}(\bm\alpha,\bm\beta;cz)$ belongs to $\mathbb{Z}[[z]]$. Consequently, a globally bounded hypergeometric series can be reduced modulo $p$ for almost every prime number $p$. As an example, the hypergeometric series $\mathfrak{g}(z):={}_3F_2(\bm\alpha,\bm\beta;z)$ with parameters $\bm\alpha=(\frac{1}{9},\frac{4}{9},\frac{5}{9})$ and $\bm\beta=(\frac{1}{3},1,1)$ is globally bounded because $\mathfrak{g}(27^2z)\in\mathbb{Z}[[z]]$. In \cite{C86}, Christol has given a characterization of the hypergeometric series that are globally bounded. For more exemples of globally bounded hypergeometric series we refer the reader to \cite{AKM20, BY22}.

In addition, there are also many generalized hypergeometric series that are not globally bounded  but,  for infinitely many prime numbers $p$, they can be reduced modulo $p$. For example, $\mathfrak{f}(z)={}_2F_1(\bm\alpha,\bm\beta;z)$ with $\bm\alpha=(\frac{1}{3},\frac{1}{2})$ and $\bm\beta=(\frac{5}{12},1)$ is not globally bounded but thanks to Proposition 24 of \cite{DRR}, for all primes $p\equiv1\bmod 12$, $\mathfrak{f}(z)\in\mathbb{Z}_{(p)}[[z]]$. 

\section{Examples}\label{sec_exam}
The aim of this section is to illustrate Theorem~\ref{theo_main_infinite_p} by applying it to the hypergeometric series $\mathfrak{f}(z)$ and $\mathfrak{g}(z)$. In order to proceed, we need some results which  will also be useful in the rest of the paper. 
\begin{lemm}\label{cyclic}
Let $\gamma=\frac{a}{b}$ be in $\mathbb{Q}\cap(0,1]$ written in lowest form and let $p$ be a prime number such $v_p(\gamma)=0$. If $p^l\equiv1\bmod b$ then $\mathfrak{D}_p^l(\gamma)=\gamma$.
\end{lemm}

\begin{proof}
Since $\gamma\in\mathbb{Z}_{(p)}$, we have $$\gamma=\sum_{s\geq0}j_sp^s,$$
where, for all $s$, $j_s\in\{0,\ldots, p-1\}$.  Note that $j_0\neq0$ because $v_p(\gamma)=0$. First, we are going to show by induction on $n\in\mathbb{N}_{>0}$ that $$\mathfrak{D}_p^n(\gamma)=1+\sum_{s\geq n}j_sp^{s-n}.$$
It is clear that, $p\left(1+\sum_{s\geq1}j_sp^{s-1}\right)-\gamma=p-j_0.$ Then, $\mathfrak{D}_p(\gamma)=1+\sum_{s\geq1}j_sp^{s-1}$ because $p-j_0\in\{1,\ldots, p-1\}$. Now, suppose that $\mathfrak{D}_p^n(\gamma)=1+\sum_{s\geq n}j_sp^{s-n}.$ It is clear that $$p\left(1+\sum_{s\geq n+1}j_sp^{s-n-1}\right)-\left(1+\sum_{s\geq n}j_sp^{s-n}\right)=p-j_n-1.$$
As $p-j_n-1\in\{0,\ldots, p-1\}$ and, by induction hypothesis, $\mathfrak{D}_p^n(\gamma)=1+\sum_{s\geq n}j_sp^{s-n}$ then $\mathfrak{D}_p^{n+1}(\gamma)=1+\sum_{s\geq n+1}j_sp^{s-n-1}$.

Thus, for all integers $n\geq1$ ,$$\mathfrak{D}_p^n(\gamma)=1+\sum_{s\geq n}j_sp^{s-n}.$$
In particular for the integer $l$, we have
 \begin{equation*}
 \begin{split}
 p^l\mathfrak{D}_p^l(\gamma)-\gamma&=p^l\left(1+\sum_{s\geq l}j_sp^{s-l}\right)-\sum_{s\geq0}j_sp^s\\
 &=p^l-\sum_{s=0}^{l-1}j_sp^s.
 \end{split}
 \end{equation*}
As $j_0\in\{1,\ldots, p-1\}$ and for all $s\geq1$, $j_s\in\{0,\ldots,p-1\}$, then $p^l-\sum_{s=0}^{l-1}j_sp^s\in\{1,\ldots, p^l-1\}$. Hence, $p^l\mathfrak{D}_p^l(\gamma)-\gamma\in\{1,\ldots, p^l-1\}$. 

We now prove that $p^l\gamma-\gamma\in\{1,\ldots, p^l-1\}$. Indeed, as $p^l\equiv1\bmod b$ then $p^l=1+bk$. So, $p^l\gamma=\gamma+ak$. We also have $a\leq b$ because by assumption $\gamma\in(0,1]$. Thus, $ak\leq bk$ and as $bk=p^l-1$ then $ak\leq p^l-1$. Therefore, $p^l\gamma-\gamma\in\{1,\ldots, p^l-1\}$. 

So that $p^l\mathfrak{D}_p^l(\gamma)-\gamma$ and $p^l\gamma-\gamma$ belong to $\{1,\ldots, p^l-1\}$. Without losing any generality we can assume that $p^l\mathfrak{D}_p^l(\gamma)-\gamma\geq p^l\gamma-\gamma.$ Then, $p^l\mathfrak{D}_p^l(\gamma)-p^l\gamma=p^l(\mathfrak{D}_p^l(\gamma)-\gamma)$ belongs to $\{0,1\ldots, p^l-1\}$. We write $\mathfrak{D}_p^l(\gamma)-\gamma=\frac{c}{d}\in\mathbb{Z}_{(p)}$ where $c$ and $d$ are positive co-prime integers. Hence, $p^lc=dt$ with $t\in\{0,1\ldots, p^l-1\}$. We assume for contradiction that $\frac{c}{d}\neq0$. As $p$ does not divide $d$, then $p^l$ divides $t$. This is a clear contradiction of the fact that $t$ belongs to $\{0,1\ldots, p^l-1\}$. Consequently, $\frac{c}{d}=0$, that is, $\mathfrak{D}_p^l(\gamma)-\gamma=0$.
 
\end{proof}

\begin{lemm}\label{lemm_val_p}
Let $p$ be a prime number and $\gamma$ be in $\mathbb{Z}_{(p)}$. We put $s:=p\mathfrak{D}_p(\gamma)-\gamma$. If  $v_p(\mathfrak{D}_{p}(\gamma))=0$ then, for every $r\in\{0,1,\ldots, p-1\}$, we have $$v_p((\gamma)_r)= \left\{ \begin{array}{lcc}
             0 &   if  & r\leq s \\
             \\ 1 &  if & r>s.  \\
             \end{array}
   \right.$$
\end{lemm}

\begin{proof}
By definition of the map $\mathfrak{D}_p$, $s$ is the unique integer in $\{0,\ldots, p-1\}$ such that $\gamma+s\in p\mathbb{Z}_{(p)}$. For this reason, $v_p((\gamma)\cdots(\gamma+s-1)(\gamma+s+1)\cdots(\gamma+p-1))=0$. So, if $r\leq s$, $v_p((\gamma)_r)=0$ and if $r>s$, $v_p((\gamma)_r)=v_p(p\mathfrak{D}_p(\gamma))=1+v_p(\mathfrak{D}_p(\gamma))=1$ because by assumption, $v_p(\mathfrak{D}_{p}(\gamma))=0$.
\end{proof}

\begin{exem}\label{exam_1}
Consider the hypergeometric series $\mathfrak{f}(z):={}_2F_1(\bm\alpha,\bm\beta;z)$, with $\bm\alpha=(\frac{1}{3},\frac{1}{2})$ and $\bm\beta=(\frac{5}{12},1)$. In this case $d_{\bm\alpha,\bm\beta}=12$. Let $\mathcal{S}$ be the set of prime numbers $p$ such that $p>24$ and $p\equiv1\bmod 12$. So, by applying Proposition 24 of \cite{DRR}, we conclude that, for every $p\in\mathcal{S}$, $\mathfrak{f}(z)\in\mathbb{Z}_{(p)}[[z]].$  From Theorem~1.2 of \cite{vmsff}, we get that, for every $p\in\mathcal{S}$, $\mathfrak{f}(z)\bmod p$ has degree at most $p^{16}$. Actually, we will see that, by applying Theorem~\ref{theo_main_infinite_p}, $\mathfrak{f}(z)\bmod p$ has degree at most $p^{2}$ for all $p\in\mathcal{S}$. Let $p$ be in $\mathcal{S}$. Then, $p=1+12k$ with $k>1$. We first prove that $S_{\bm\alpha,\bm\beta,p}=\{0,1+5k\}$. It is nor hard to see that $E_{\bm\alpha,\bm\beta,p}=\{0,1+5k\}$ and it is clear that $0\in S_{\bm\alpha,\bm\beta,p}$. As $p\equiv1\bmod 12$ and $v_p(\bm\alpha)=(0,0)=v_p(\bm\beta)$ then, from Lemma~\ref{cyclic}, we obtain $\mathfrak{D}_p(\bm\alpha)=\bm\alpha$ and $\mathfrak{D}_p(\bm\beta)=\bm\beta$. Thus, we obtain the following equalities: $$4k=p\mathfrak{D}_p(1/3)-1/3,\quad 6k=p\mathfrak{D}_p(1/2)-1/2,\text{ and } 5k=p\mathfrak{D}_p(5/12)-5/12.$$ So, from Lemma~\ref{lemm_val_p}, we obtain $$v_p((1/3)_{1+5k})=1,\quad v_p((1/2)_{1+5k})=0,\text{ and }v_p((5/12)_{1+5k})=1.$$ It is clear that $v_p((1)_{1+5k})=0$. Therefore, $$v_p\left(\frac{(1/3)_{1+5k}(1/2)_{1+5k}}{(5/12)_{1+5k}(1)_{1+5k}}\right)=0.$$
Whence, $1+5k\in S_{\bm\alpha,\bm\beta,p}$. Consequently, $\#S_{\bm\alpha,\bm\beta,p}=2$. 
Then, it follows from (2) of Theorem~\ref{theo_main_infinite_p} that there are $Q_{1,p}(z), Q_{2,p}(z)\in\mathbb{Q}(z)\cap\mathbb{Z}_{(p)}[[z]][z^{-1}]$ such that 
\begin{equation}\label{eq_ex_1}
\mathfrak{f}\equiv Q_{1,p}(z)\mathfrak{f}^{p}+Q_{2,p}(z)\mathfrak{f}^{p^{2}}\bmod p
\end{equation}
 and the heights of $Q_{1,p}(z)\bmod p$ and $Q_{2,p}(z)\bmod p$ are less than $2p^{2}.$
 \end{exem}

\begin{exem}\label{exam_2}
Consider the hypergeometric series $\mathfrak{g}(z):={}_3F_2(\bm\alpha,\bm\beta;z)$, with $\bm\alpha=(\frac{1}{9},\frac{4}{9},\frac{5}{9})$ and $\bm\beta=(\frac{1}{3},1,1)$. In this case $d_{\bm\alpha,\bm\beta}=9.$ It turns out that $\mathfrak{g}(27^2z)\in\mathbb{Z}[[z]]$. So that, for every prime number $p\neq3$, $\mathfrak{g}(z)$ belongs to $\mathbb{Z}_{(p)}[[z]]$. From Theorem~1.2 of \cite{vmsff}, we get that, for all primes $p\neq3$, $\mathfrak{g}(z)\bmod p$ has degree at most $p^{54}$. Nevertheless, by applying Theorem~\ref{theo_main_infinite_p}, we obtain for some prime numbers $p$ a finer estimate than $p^{54}$. Let $\mathcal{S}$ be the set of prime numbers $p$ such that $p>18$ and $p\equiv8\bmod 9$. Then, for every $p\in\mathcal{S}$, $p^2\equiv1\bmod 9$. We are going to see that, for every $p\in\mathcal{S}$, $\#S_{\mathfrak{D}_p(\bm\alpha),\mathfrak{D}_p(\bm\beta),p}=2$. Let $p$ be in $\mathcal{S}$. We put $\bm\alpha'=(8/9,5/9,4/9)$ and $\bm\beta'=(2/3,1,1)$. As $p\equiv8\bmod 9$ then $p=8+9k$ with $k>1$ and we also have the following equalities: 
\begin{align*}
7+8k&=p(8/9)-1/9 & 4+5k&=p(5/9)-4/9\\
3+4k&=p(4/9)-5/9 & 5+6k&=p(2/3)-1/3.\\
\end{align*}
 So that, $\mathfrak{D}_p(\bm\alpha)=\bm\alpha'$ and $\mathfrak{D}_p(\bm\beta)=\bm\beta'$. Thus, $E_{\mathfrak{D}_p(\bm\alpha),\mathfrak{D}_p(\bm\beta),p}=\{0, 3+3k_p\}$. Furthermore, since $p^2\equiv1\bmod 9$ and $v_p(\bm\alpha)=(0,0,0)=v_p(\bm\beta)$, by Lemma~\ref{cyclic}, we obtain $\mathfrak{D}_p^2(\bm\alpha)=\bm\alpha$ and $\mathfrak{D}_p^2(\bm\beta)=\bm\beta$. Therefore,  $\mathfrak{D}_p(\bm\alpha')=\bm\alpha$ and $\mathfrak{D}_p(\bm\beta')=\bm\beta$ and consequently, we obtain the following equalities: 
 \begin{align*}
 k&=p(1/9)-(8/9) & 3+4k&=p(4/9)-(5/9)\\
 4+5k&=p(5/9)-(4/9) & 2+3k&=p(1/3)-2/3.
 \end{align*}
So, it follows from Lemma~\ref{lemm_val_p} that $$v_p((8/9)_{3+3k})=1,\quad v_p((5/9)_{3+3k})=0,\quad v_p((4/9)_{3+3k})=0,\text{ and }v_p((2/3)_{3+3k})=1.$$ It is clear that $v_p((1)_{3+3k})=0$. Therefore, $$v_p\left(\frac{(8/9)_{3+3k}(5/9)_{3+3k}(4/9)_{3+3k}}{(2/3)_{3+3k}(1)_{3+3k}^2}\right)=0.$$
Whence, $3+3k_p\in S_{\mathfrak{D}_p(\bm\alpha),\mathfrak{D}_p(\bm\beta),p}$. And, it is clear that $0\in S_{\mathfrak{D}_p(\bm\alpha),\mathfrak{D}_p(\bm\beta),p}.$

 Consequently, $\#S_{\mathfrak{D}_p(\bm\alpha),\mathfrak{D}_p(\bm\beta),p}=2$. Then, it follows from (2) of Theorem~\ref{theo_main_infinite_p} that, for every prime $p\in\mathcal{S}$, there are $A_{1,p}(z), A_{2,p}(z)\in\mathbb{Q}(z)\cap\mathbb{Z}_{(p)}[[z]][z^{-1}]$ such that 
\begin{equation}\label{eq_ex_3}
\mathfrak{g}\equiv A_{1,p}(z)\mathfrak{g}^{p^2}+A_{2,p}(z)\mathfrak{g}^{p^{4}}\bmod p
\end{equation}
 and the heights of $A_{1,p}(z)\bmod p$ and $A_{2,p}(z)\bmod p$ are less than $2p^{4}.$
\end{exem}
 
%\textbullet\quad Let $\mathcal{S}_1$ be the set of prime numbers $p$ such that $p>18$ and $p\equiv1\bmod 9$.  We are going to see that, for every $p\in\mathcal{S}_1$, $\#S_{\bm\alpha,\bm\beta,p}=2$. Let $p$ be in $\mathcal{S}_1$. Then, $p=1+9k$ with $k>1$. In this case we have $E_{\bm\alpha,\bm\beta,p}=\{0,1+3k\}.$ From Lemma~\ref{cyclic}, we conclude that $\mathfrak{D}_p(\bm\alpha)=\bm\alpha$ and $\mathfrak{D}_p(\bm\beta)=\bm\beta$. Thus, $$k=p\mathfrak{D}_p(1/9)-1/9,\text{ } 4k=p\mathfrak{D}_p(4/9)-4/9,\text{ } 5k=p\mathfrak{D}_p(5/9)-5/9,\text{ and }3k=p\mathfrak{D}_p(1/3)-1/3.$$ So, from Lemma~\ref{lemm_val_p}, we have $$v_p((1/9)_{1+3k})=1,\quad v_p((4/9)_{1+3k})=0,\quad v_p((5/9)_{1+3k})=0,\text{ and }v_p((1/3)_{1+3k})=1.$$ It is clear that $v_p((1)_{1+3k})=0$. Therefore, $$v_p\left(\frac{(1/9)_{1+3k}(4/9)_{1+3k}(5/9)_{1+3k}}{(1/3)_{1+3k}(1)_{1+3k}^2}\right)=0.$$
%Whence, $1+3k\in S_{\bm\alpha,\bm\beta,p}$. And, it is clear that $0\in S_{\bm\alpha,\bm\beta,p}$.

%Consequently, $\#S_{\bm\alpha,\bm\beta,p}=2$. Then, it follows from (2) of Theorem~\ref{theo_main_infinite_p} that, for every prime $p\in\mathcal{S}_1$, there are $B_{1,p}(z), B_{2,p}(z)\in\mathbb{Q}(z)\cap\mathbb{Z}_{(p)}[[z]][z^{-1}]$ such that 
%\begin{equation}\label{eq_ex_2}
%\mathfrak{g}\equiv B_{1,p}(z)\mathfrak{g}^{p}+B_{2,p}(z)\mathfrak{g}^{p^{2}}\bmod p
%\end{equation}
 %and the heights of $B_{1,p}(z)\bmod p$ and $B_{2,p}(z)\bmod p$ are less than $2p^{2}.$
 
\begin{rema}
 An explicit formula for each rational function appearing in Equation \eqref{eq_ex_1} can be obtained from Theorem~\ref{theo_cons}. Further, Theorem~\ref{theo_explicit} gives an explicit formula for each rational function appearing in Equation \eqref{eq_ex_3}.
 \end{rema} 
  
\section{Proof of Theorem~\ref{theo_main_infinite_p}}

The proof of Theorem~\ref{theo_main_infinite_p} is based on Theorem~\ref{theo_main} and Proposition~\ref{prop_d_p}, which are stated below. In order to formulate Theorem~\ref{theo_main}, we have to define the $\textbf{P}_{p,l}$ property. We denote by $\mathbb{Z}_{(p)}^*$ the set of units of $\mathbb{Z}_{(p)}$. As we have already said, the ring $\mathbb{Z}_{(p)}$ is a local ring and its maximal ideal is $(p)\mathbb{Z}_{(p)}$. So, $\gamma\in\mathbb{Z}_{(p)}^*$ if and only if $\gamma\notin(p)\mathbb{Z}_{(p)}$ if and only if $v_p(\gamma)=0$.
\begin{defi}
Let $p$ be a prime number and let $\bm\alpha=(\alpha_1,\ldots,\alpha_n)$, $\bm\beta=(\beta_1,\ldots,\beta_{n-1},1)$ be in $(\mathbb{Z}_{(p)})^n$ and let  $l\geq1$ be an integer. We say that $(\bm\alpha$, $\bm\beta)$ satisfies the $\textbf{P}_{p,l}$ property,  if,  for every $k\in\{1,\ldots,l\}$, we have:

\begin{enumerate}[label=(\textbf{P}\arabic*)]

\item  $\mathfrak{D}_p^{k}(\bm\alpha)$ and $\mathfrak{D}_p^{k}(\bm\beta)$ belong to $(\mathbb{Z}^*_{(p)}\cap(0,1])^n$,

\item  $\mathfrak{D}_p^{k}(\alpha_i)-\mathfrak{D}_p^{k}(\beta_j)$ belongs to $\mathbb{Z}^*_{(p)}$ for $1\leq i,j\leq n$,

\item   $\mathfrak{D}_p^{k}(\beta_j)-\mathfrak{D}_p^{k}(\beta_s)$ belongs to $\mathbb{Z}^*_{(p)}$ if and only if $\beta_j\neq\beta_s$,

\item $p-1\notin I_{\bm\beta}^{(k+1)}$, where $I_{\bm\beta}^{(k+1)}=\{p\mathfrak{D}_p^{k+1}(\beta_j)-\mathfrak{D}_p^{k}(\beta_j): 1\leq j\leq n\text{ and } \beta_j\neq1\}.$

\item For every, $i,j\in\{1,\ldots,n\}$, $1-\mathfrak{D}_p^k(\beta_j)+\mathfrak{D}_p^k(\alpha_i)\in\mathbb{Z}^{*}_{(p)}$ and $1-\mathfrak{D}_p^k(\beta_j)+\mathfrak{D}_p^k(\beta_i)\in\mathbb{Z}^*_{(p)}$.

\end{enumerate}
\end{defi}

%\begin{defi}
%Let $\bm\alpha=(\alpha_1,\ldots,\alpha_n)$, $\bm\beta=(\beta_1,\ldots,\beta_{n-1},1)$ be in $(\mathbb{Q}\setminus\mathbb{Z}_{\leq0})^n$ and let $p$ be a prime number. We say $(\bm\alpha,\bm\beta)$ satisfies the $\textbf{H}_p$-property if:
%\begin{enumerate}[label=(\textbf{H}\arabic*)]
%\item $\bm\alpha$, $\bm\beta$ be in $(\mathbb{Z}_{(p)})^n$,
%\item ${}_nF_{n-1}(\bm\alpha,\bm\beta;z)$ belongs to $\mathbb{Z}_{(p)}[[z]]$,
%\item  $(\bm\alpha$, $\bm\beta)$ satisfies the $\textbf{P}_{p,l}$, where $l$ is the order of $p$ in $(\mathbb{Z}/d_{\bm\alpha,\bm\beta}\mathbb{Z})^*$.
%\end{enumerate}
%\end{defi}

We are now ready to state Theorem~\ref{theo_main}.
\begin{theo}\label{theo_main}
Let $\bm\alpha=(\alpha_1,\ldots,\alpha_n)$, $\bm\beta=(\beta_1,\ldots,\beta_{n-1},1)$ be in $(\mathbb{Q}\cap(0,1])^n$ and let $p$ be a prime number such that $f(z):={}_nF_{n-1}(\bm\alpha,\bm\beta;z)$ belongs to $\mathbb{Z}_{(p)}[[z]]$. Suppose that $(\bm\alpha, \bm\beta)$ satisfies the $\textbf{P}_{p,l}$ property, where $l$ is the order of $p$ in $(\mathbb{Z}/d_{\bm\alpha,\bm\beta}\mathbb{Z})^*$. If, for all $i,j\in\{1,\ldots,n\}$, $\alpha_i$ and $\beta_j$ belong to $\mathbb{Z}^*_{(p)}$ then the assertions (1), (2), and (3) of Theorem~\ref{theo_main_infinite_p} hold.
\end{theo}

The next proposition deals with some properties of the map $\mathfrak{D}_p$.  

\begin{prop}\label{prop_d_p}
Let $\gamma=a/b$ and $\tau=c/d$ be in $\mathbb{Q}\cap(0,1]$ written in lowest form and let $p$ be a prime number such that $\gamma,\tau\in\mathbb{Z}_{(p)}$. Then:
\begin{enumerate}
\item $\mathfrak{D}_p(\gamma)=y/b\in\mathbb{Q}\cap(0,1]$, where $y\in\{1,\ldots, b\}$ and $py\equiv a\bmod b$. Moreover, $y$ and $b$ are co-prime,

\item $\mathfrak{D}_p(\gamma)=\mathfrak{D}_p(\tau)$ if and only if $\gamma=\tau$,

\item if $p>b$, $\mathfrak{D}_p(\gamma)\in\mathbb{Z}^{*}_{(p)}$.
\end{enumerate}
\end{prop}

\begin{proof}
(1).  By definition of $\mathfrak{D}_p$, it follows that $p\mathfrak{D}_p(\gamma)=\gamma+s$ with $s\in\{0,\ldots,p-1\}$. So, $\gamma+s=(a+sb)/b\in (p)\mathbb{Z}_{(p)}$. Thus, $y:=(a+sb)/p\in\mathbb{N}$. So, $\mathfrak{D}_p(\gamma)=y/b$ and $py\equiv a\bmod b$. Assume for contradiction that $y=0$. Then, $\mathfrak{D}_p(\gamma)=0$ and $a\in(b)\mathbb{Z}$. But, by hypotheses, $0<a/b\leq1$. Thus, $a=b$ and so, $\gamma=1$. But, $\mathfrak{D}_p(1)=1$. Whence, $0=1$, which is a contradiction. Thus, $y>0$. We now prove that $y\in\{1,\ldots,b-1,b\}$. Since $\gamma\in(0,1]$ and $s\leq p-1$, it follows that $\gamma+s\leq p$. As $\gamma+s=(a+bs)/b$ then $(a+bs)/b\leq p$. Thus, $y=(a+bs)/p\leq b$. Consequently, $y\in\{1,\ldots,b-1,b\}.$ Therefore, $\mathfrak{D}_p(\gamma)\in\mathbb{Q}\cap(0,1]$. Finally, we show that $y$ and $b$ are co-prime. As $py\equiv a\bmod b$ and, by assumption, $a$ and $b$ are co-prime then $y$ and $b$ are co-prime.

(2). It is clear that if $\gamma=\tau$ then $\mathfrak{D}_p(\gamma)=\mathfrak{D}_p(\tau)$. We now prove that if $\mathfrak{D}_p(\gamma)=\mathfrak{D}_p(\tau)$ then $\gamma=\tau$. From (1), we know that $\mathfrak{D}_p(\gamma)=y/b\in\mathbb{Q}\cap(0,1]$, where $y\in\{1,\ldots, b\}$ and $py\equiv a\bmod b$ and $\mathfrak{D}_p(\tau)=x/d\in\mathbb{Q}\cap(0,1]$, where $x\in\{1,\ldots, d\}$ and $px\equiv c\bmod d$. First, we suppose that $\gamma=1$. So, $1=\mathfrak{D}_p(\gamma)=y/b$. In particular, $y=b$. By assumption, $\mathfrak{D}_p(\gamma)=\mathfrak{D}_p(\tau)$. Then, $1=\mathfrak{D}_p(\tau)$ and $x=d$. Thus, $pd\equiv c\bmod d$. Whence,  $c\in(d)\mathbb{Z}$. But, by hypotheses, $0<c/d\leq 1$. For this reason, $c=d$. Therefore, $\tau=1$. Now, we suppose that $\gamma<1$. Assume for contradiction that $\mathfrak{D}_p(\gamma)=1$. Then, $y=b$ and $pb\equiv a\bmod b$. Whence, $a\in(b)\mathbb{Z}$. Since $0<a/b\leq1$, we have $a=b$. Therefore, $\gamma=1$, which is a contradiction. Consequently, $\mathfrak{D}_p(\gamma)<1$. By assumption, $\mathfrak{D}_p(\gamma)=\mathfrak{D}_p(\tau)$. Then, $\mathfrak{D}_p(\tau)<1$. From (1), we know that $\mathfrak{D}_p(\gamma)=y/b$ where $y\in\{1,\ldots, b\}$ and $\mathfrak{D}_p(\tau)=x/d$, where $x\in\{1,\ldots, d\}$. Actually, we have $y\in\{1,\ldots, b-1\}$ and $x\in\{1,\ldots, d-1\}$ because $0<y/b, x/d<1$. But, $y/b=x/d$ because $\mathfrak{D}_p(\gamma)=\mathfrak{D}_p(\tau)$. As $y$, $b$ are co-prime and $x$, $d$ are co-prime then the equality $y/b=x/d$ implies $y=x$ and $b=d$. In particular, we have $py\equiv a\bmod b$ and $py\equiv c\bmod b$. So, $a-c\in(b)\mathbb{Z}$. As $\gamma=a/b<1$, $\tau=c/d\leq1$, and $d=b$ then $|a-c|<b$. But, $a-c\in(b)\mathbb{Z}$. Thus, $a=c$. So, $\gamma=\tau$.

(3). According to (1), $\mathfrak{D}_p(\gamma)=y/b$, where $y\in\{1,\ldots, b\}$ and $y$ and $b$ are co-prime. Since $p>b$, we have $p>y$. In particular, $p$ does not divide $y$ and thus, $y/b\in\mathbb{Z}^{*}_{(p)}$.

\end{proof}
By assuming Theorem~\ref{theo_main}, we can now prove Theorem~\ref{theo_main_infinite_p}.

\begin{proof}[Proof of Theorem~\ref{theo_main_infinite_p}]  Let $p>2d_{\bm\alpha,\bm\beta}$ be a prime number. Then $\alpha_i, \beta_j\in\mathbb{Z}_{(p)}$ for all $1\leq i,j\leq n$ given that $p>d_{\bm\alpha,\bm\beta}$ and $\alpha_i, \beta_j\in(0,1]$ for all $1\leq i,j\leq n$. We first prove that the following two conditions are satisfied. 

\begin{enumerate}[label=\alph*)]
\item For every integer $m\geq0$, $\mathfrak{D}_p^m(\alpha_i)-\mathfrak{D}_p^m(\beta_j)\notin\mathbb{Z}$ for all $1\leq i,j\leq n$.

\item For every integer $m\geq0$, $\mathfrak{D}_p^m(\beta_j)-\mathfrak{D}_p^m(\beta_s)\notin\mathbb{Z}$ if and only if $\beta_j\neq\beta_s$.
\end{enumerate}
 By hypotheses, we know that, for all $i,j\in\{1,\ldots,n\}$, $\alpha_i-\beta_j\notin\mathbb{Z}$. That is equivalent to saying that, $\alpha_i\neq\beta_j$ because, for all
$i,j\in\{1,\ldots,n\}$, $\alpha_i$, $\beta_j$ belong to $(0,1]$. Thus, it follows from (2) of Proposition~\ref{prop_d_p} that, for all integers $m\geq0$, $\mathfrak{D}^m_p(\alpha_i)\neq\mathfrak{D}^m_p(\beta_j)$. Thus, $\mathfrak{D}^m_p(\alpha_i)-\mathfrak{D}^m_p(\beta_j)\notin\mathbb{Z}$ because,  for all integers $m\geq0$, $\mathfrak{D}^m_p(\alpha_i)$, $\mathfrak{D}^m_p(\beta_j)$ belong to $(0,1]$. Therefore, the condition (a) is satisfied. Following the same argument, one shows that the condition (b) is also satisfied. 
 
We now prove that  $(\bm\alpha,\bm\beta)$ satisfies the $\textbf{P}_{p,m}$ property for all integers $m\geq1$. To this end, we set
\begin{itemize}
\item$\mathcal{U}_1=\{\mathfrak{D}_p^m(\alpha_i),\mathfrak{D}_p^m(\beta_j)\}_ {m\geq1, 1\leq i,j \leq n}$.
\end{itemize}
As $p>d_{\bm\alpha,\bm\beta}$ and $\alpha_i,\beta_j\in(0,1]$ for all $1\leq i,j\leq n$ then $\alpha_i,\beta_j\in\mathbb{Z}_{(p)}^*$ for all $1\leq i,j\leq n$. Then, it follows from (1) and (3) of Proposition~\ref{prop_d_p} that $\mathcal{U}_1\subset\mathbb{Z}_{(p)}^{*}\cap(0,1]$.

Now, we consider the following set,

\begin{itemize}
\item $\mathcal{U}_2=\{\mathfrak{D}_p^m(\alpha_i)-\mathfrak{D}_p^m(\beta_j)\}_{m\geq1, 1\leq i,j\leq n}.$
\end{itemize}
We have $0\notin\mathcal{U}_2$ because, by condition a), we know that, for every $m\geq1$, $\mathfrak{D}_p^m(\alpha_i)-\mathfrak{D}_p^m(\beta_j)\notin\mathbb{Z}$ for all $1\leq i,j\leq n$. Now, we prove that for any $1\leq i,j\leq n$, $\mathfrak{D}_p^m(\alpha_i)-\mathfrak{D}_p^m(\beta_j)$ belongs to $\mathbb{Z}_{(p)}^{*}$. Indeed, let $i,j$ be in $\{1,\ldots,n\}$ and let us write $\alpha_i=a/b$ and $\beta_j=c/d$ in lowest form. Then, from (1) of Proposition~\ref{prop_d_p}, we get $\mathfrak{D}_p^m(\alpha_i)=y/b$ and $\mathfrak{D}_p^m(\beta_j)=y'/d$, where $0<y\leq b$, $0<y'\leq d$ and $y,b$ are co-prime and $y',d$ are co-prime. So, $\mathfrak{D}_p^m(\alpha_i)-\mathfrak{D}_p^m(\beta_j)=(yd-y'b)/bd$ and $|yd-y'b|<bd$. As $p>d_{\bm\alpha,\bm\beta}$ then $p>bd$ and thus, $p>|yd-y'b|$. So, $\mathfrak{D}_p^m(\alpha_i)-\mathfrak{D}_p^m(\beta_j)$ belongs to $\mathbb{Z}_{(p)}^{*}$. Thus, $\mathcal{U}_2\subset\mathbb{Z}_{(p)}^{*}$.

We also consider the following set,

\begin{itemize}
\item$\mathcal{U}_3=\{\mathfrak{D}_p^m(\beta_j)-\mathfrak{D}_p^m(\beta_s): 1\leq j,s\leq n,\text{ }\beta_j\neq\beta_s\}_{m\geq1}.$
\end{itemize}
We have $0\notin\mathcal{U}_3$ because, by condition b), we know that, for every $m\geq1$, $\mathfrak{D}_p^m(\beta_j)-\mathfrak{D}_p^m(\beta_s)\notin\mathbb{Z}$ if and only if $\beta_j\neq\beta_s$. Following the same argument as in $\mathcal{U}_2$, one gets $\mathcal{U}_3\subset\mathbb{Z}_{(p)}^{*}$.

We have the following set,

\begin{itemize}
\item$\mathcal{U}_4=\{1-\mathfrak{D}_p^m(\beta_j): 1\leq j\leq n,\text{ }\beta_j\neq 1\}_{m\geq1}.$
\end{itemize}
Assume for contradiction that $0\in\mathcal{U}_4$. Then $1=\mathfrak{D}_p^m(\beta_j)$ for somme $j\in\{1,\ldots,n\}$ with $\beta_j\neq1$.  As $\mathfrak{D}_p^m(1)=1=\mathfrak{D}_p^m(\beta_j)$ then, according to (2) of Proposition~\ref{prop_d_p}, $\beta_j=1$, which is a contradiction. Therefore, $0\notin\mathcal{U}_4$. Following the same argument as in $\mathcal{U}_2$, one gets $\mathcal{U}_4\subset\mathbb{Z}_{(p)}^{*}$.
\begin{itemize}
\item$\mathcal{U}_5=\{1-\mathfrak{D}_p^m(\beta_j)+\mathfrak{D}_p^m(\alpha_i), 1-\mathfrak{D}_p^m(\beta_t)+\mathfrak{D}_p^m(\beta_s)\}_{m\geq1, 1\leq i,j,t,s \leq n}$.
\end{itemize}

We have $0\notin\mathcal{U}_5$ because, from (1) of Proposition~\ref{prop_d_p}, for every $m\geq1$, $\mathfrak{D}_p^m(\bm\alpha)$ and $\mathfrak{D}_p^m(\bm\beta)$ belong to $((0,1]\cap\mathbb{Q})^n$. We now prove that $\mathcal{U}_5\subset\mathbb{Z}_{(p)}^{*}$.  Indeed, let $i,j$ be in $\{1,\ldots,n\}$ and let us write $\alpha_i=a/b$ and $\beta_j=c/d$ in lowest form. Then, from (1) of Proposition~\ref{prop_d_p}, we get $\mathfrak{D}_p^m(\alpha_i)=y/a$ and $\mathfrak{D}_p^m(\beta_j)=y'/d$, where $0<y\leq b$, $0<y'\leq d$ and $y,b$ are co-prime and $y',d$ are co-prime. So, $1-\mathfrak{D}_p^m(\beta_j)+\mathfrak{D}_p^m(\alpha_i)=(bd-y'b+yd)/bd$ and $|bd-y'b+yd|<2bd$. As $p>2d_{\bm\alpha,\bm\beta}$ then $p>2bd$ and thus, $p>|bd-y'b+yd|$. So, $1-\mathfrak{D}_p^m(\beta_j)+\mathfrak{D}_p^m(\alpha_i)\in\mathbb{Z}_{(p)}^{*}$. In a similar way, one shows that, for any $1\leq m,s\leq n$,  $1-\mathfrak{D}_p^m(\beta_t)+\mathfrak{D}_p^m(\beta_s)\in\mathbb{Z}_{(p)}^{*}$.

We now see that $(\bm\alpha,\bm\beta)$ satisfies the $\textbf{P}_{p,m}$ property. The condition $(\textbf{P}1)$ is satisfied because $\mathcal{U}_1\subset\mathbb{Z}_{(p)}^{*}\cap(0,1]$. Now, since $\mathcal{U}_2\subset\mathbb{Z}_{(p)}^{*}$, the condition $(\textbf{P}2)$ is satisfied. The condition $(\textbf{P}3)$ is  also satisfied because $\mathcal{U}_3\subset\mathbb{Z}_{(p)}^{*}$. Assume now for contradiction that $p-1\in I_{\bm\beta}^{(k+1)}$ for some $k\in\{1,\ldots,l\}$. Then, $p-1=p\mathfrak{D}_p^{k+1}(\beta_j)-\mathfrak{D}_p^{k}(\beta_j)$ for $\beta_j\neq1$. So that, $1-\mathfrak{D}_p^{k}(\beta_j)\in\ p\mathbb{Z}_{(p)}$. But,  $1-\mathfrak{D}_p^k(\beta_j)\in\mathbb{Z}_{(p)}^*$ because $1-\mathfrak{D}_p^k(\beta_j)\in\mathcal{U}_4$. So, we obtain a contradiction. Thus, for all $k\in\{1,\ldots,l\}$, $p-1\notin I_{\bm\beta}^{(k+1)}$. Whence, the condition $(\textbf{P}4)$ is satisfied. Finally, the condition $(\textbf{P}5)$ is satisfied since $\mathcal{U}_5\subset\mathbb{Z}_{(p)}^{*}$. Hence, $(\bm\alpha,\bm\beta)$ satisfies the $\textbf{P}_{p,m}$ property for all integers $m\geq1$. In particular, $(\bm\alpha,\bm\beta)$ satisfies the $\textbf{P}_{p,l}$ property, where $l$ is the order of $p$ in $(\mathbb{Z}/d_{\bm\alpha,\bm\beta}\mathbb{Z})^*$.

We have already seen that, for all $i,j\in\{1,\ldots,n\}$, $\alpha_i$ and $\beta_j$ belong to $\mathbb{Z}^*_{(p)}$. Consequently, by applying Theorem~\ref{theo_main}, the assertions (1), (2), and (3) hold, which completes the proof.

\end{proof}

\begin{rema}\label{rema_p_l} Let $\bm\alpha=(\alpha_1,\ldots,\alpha_n)$, $\bm\beta=(\beta_1,\ldots,\beta_{n-1},1)$ be in $(\mathbb{Q}\cap(0,1])^n$, $p>2d_{\bm\alpha,\bm\beta}$ be a prime number. Then, it follows from the proof of Theorem~\ref{theo_main_infinite_p} that $(\bm\alpha,\bm\beta)$ satisfies the $\textbf{P}_{p,m}$ property for all integers $m\geq1$.

%\textbullet\quad If $(\bm\alpha,\bm\beta)$ satisfies the $\textbf{P}_{p,l}$ property then $\bm\alpha,\bm\beta\in(\mathbb{Z}_{(p)}^*)^n$. Indeed, we have $\mathfrak{D}_p(\bm\alpha)$, $\mathfrak{D}_p(\bm\beta)\in(\mathbb{Z}_{(p)}^*\cap(0,1])^n$ because of $(\textbf{P}1)$ so, by Lemma~\ref{cyclic}, we obtain $\mathfrak{D}^l_p(\bm\alpha)=\bm\alpha$, $\mathfrak{D}^l_p(\bm\beta)=\bm\beta$ but $\mathfrak{D}^l_p(\bm\alpha)$, $\mathfrak{D}^l_p(\bm\beta)\in(\mathbb{Z}_{(p)}^*\cap(0,1])^n$  because of $(\textbf{P}1)$.
\end{rema}

\section{Proof of Theorem~\ref{theo_main}}\label{sec_proof_theo_main}

Theorem~\ref{theo_main} is derived from Propositions \ref{prop_third_step} and \ref{lemm_system}. In order to state Proposition~\ref{prop_third_step}, we need to introduce two more sets and some notations. Let $p$ be a prime number. For $\bm\gamma=(\gamma_1,\ldots,\gamma_n)\in\mathbb{Z}_{(p)}^n$  and $r\in\mathbb{Z}_{\geq0}$, we consider the following two sets: $$\mathcal{P}_{\bm\gamma,r}=\{s\in\{1,\ldots,n\}: (\gamma_s)_r\in p\mathbb{Z}_{(p)}\}\text{ and }\mathcal{C}_{\bm\gamma,r}=\{s\in\{1,\ldots,n\}: (\gamma_s)_r\notin p\mathbb{Z}_{(p)}\}.$$
Note that $\mathcal{C}_{\bm\gamma,r}$ is the complement of $\mathcal{P}_{\bm\gamma,r}$ in $\{1,\ldots,n\}$ and that $\mathcal{C}_{\bm\gamma,0}=\{1,\ldots, n\}$.

Let  $\bm\alpha=(\alpha_1,\ldots,\alpha_n)$, $\bm\beta=(\beta_1,\ldots,\beta_{n-1},1)$ be in $(\mathbb{Q}\setminus\mathbb{Z}_{\leq0})^n$ and let $p$ be a prime number such that $\bm\alpha$ and $\bm\beta$ belong to $\mathbb{Z}_{(p)}^n$. Then, for every integer $a\geq1$ and for every $r\in\{0,\ldots, p-1\}$, we set $\bm\alpha_{a,r}=(\alpha_{1,a,r},\ldots,\ldots,\alpha_{n,a,r})$ and $\bm\beta_{a,r}=(\beta_{1,a,r},\ldots,\ldots,\beta_{n,a,r}),$
where, for every $s\in\{1,\ldots,n\}$, $$\alpha_{s,a,r}= \left\{ \begin{array}{lcc}
             \mathfrak{D}_p^a(\alpha_s) &   if  &  s\in\mathcal{C}_{\mathfrak{D}_p^{a-1}(\bm\alpha),r} \\
             \\  \mathfrak{D}_p^a(\alpha_s)+1&  if & s\in\mathcal{P}_{\mathfrak{D}_p^{a-1}(\bm\alpha),r},  \\
             \end{array}
   \right.$$
 $$\beta_{s,a,r}= \left\{ \begin{array}{lcc}
             \mathfrak{D}_p^a(\beta_s) &   if  &  s\in\mathcal{C}_{\mathfrak{D}_p^{a-1}(\bm\beta),r} \\
             \\  \mathfrak{D}_p^a(\beta_s)+1&  if & s\in\mathcal{P}_{\mathfrak{D}_p^{a-1}(\bm\beta),r}.  \\
             \end{array}
   \right.$$
   Note that, for every $a\geq1$ and $r\in\{0,\ldots,p-1\}$, $\beta_{n,a,r}=1$ because $n\in\mathcal{C}_{\mathfrak{D}_p^{a-1}(\bm\beta),r}$ and $\mathfrak{D}_p^a(1)=1$. So, it makes sense to consider the hypergeometric series ${}_nF_{n-1}(\bm\alpha_{a,r},\bm\beta_{a,r};z)$. We let $f_{a,r}$ denote  the hypergeometric series ${}_nF_{n-1}(\bm\alpha_{a,r},\bm\beta_{a,r};z)$. Thus, $$f_{a,r}=\sum_{m\geq0}\left(\frac{\underset{s\in\mathcal{C}_{\mathfrak{D}_p^{a-1}(\bm\alpha),r}}\prod(\mathfrak{D}_p^{a}(\alpha_s))_m\underset{s\in\mathcal{P}_{\mathfrak{D}_p^{a-1}(\bm\alpha),r}}{\prod}(\mathfrak{D}_p^a(\alpha_s)+1)_m}{\underset{s\in\mathcal{C}_{\mathfrak{D}_p^{a-1}(\bm\beta),r}}\prod(\mathfrak{D}_p^{a}(\beta_s))_m\underset{s\in\mathcal{P}_{\mathfrak{D}_p^{a-1}(\bm\beta),r}}{\prod}(\mathfrak{D}_p^a(\beta_s)+1)_m}\right)z^m.$$
\begin{prop}\label{prop_third_step}
Let the assumptions be as in Theorem~\ref{theo_main}. Then, for every $r\in S_{\mathfrak{D}_p^{l-1}(\bm\alpha),\mathfrak{D}_p^{l-1}(\bm\beta),p}$, $f_{l,r}\in1+z\mathbb{Z}_{(p)}[[z]]$ and $$f_{l,r}\equiv\sum_{j\in S_{\mathfrak{D}_p^{l-1}(\bm\alpha),\mathfrak{D}_p^{l-1}(\bm\beta),p}}Q_{r,j}(z)f_{l,j}^{p^l}\bmod p,$$
where, for every $j\in S_{\mathfrak{D}_p^{l-1}(\bm\alpha),\mathfrak{D}_p^{l-1}(\bm\beta),p}$, $Q_{r,j}(z)$ belongs to $\mathbb{Z}_{(p)}[z]$ and has degree less than $p^l$.  
\end{prop}

\begin{prop}\label{lemm_system}
Let $g_0,g_1,\ldots, g_{t-r}$ be in $\mathbb{F}_p[[z]]$ different from zero and let $l$ be a positive integer. Suppose that, for every $i\in\{0,\ldots, t-r\}$, $$g_i=\sum_{k=1}^{s}P_{i,k}g_0^{p^{kl}}+\sum_{k=1}^{t-r}A_{i,k}g_k^{p^{sl}},$$
where, for all $i\in\{0,\ldots,t-r\}$, $P_{i,1},\ldots, P_{i,s}, A_{i,1},\ldots A_{i,t-r}$ belong to $\mathbb{F}_p(z)$ and their heights are less than $cp^{sl}$.
If $A_{0,t-r}$ is not zero then, for every $i\in\{0,\ldots,t-r-1\}$, $$g_i=\sum_{k=1}^{2s}T_{i,k}g_0^{p^{kl}}+\sum_{k=1}^{t-r-1}D_{i,k}g_k^{p^{2sl}},$$
where, for all $i\in\{0,\ldots,t-r-1\}$, $T_{i,1},\ldots, T_{i,2s}$, $D_{i,1},\ldots D_{i,t-r-1}$ belong to $\mathbb{F}_p(z)$ and their heights are less than $5c(t-r+1)p^{2sl}$.
\end{prop}

By assuming Propositions \ref{prop_third_step} and \ref{lemm_system}, we are now in a position to prove Theorem~\ref{theo_main}.

\begin{proof}[Proof of Theorem~\ref{theo_main}]
Note that $f_{l,0}$ is the hypergeometric series $_nF_{n-1}(\bm\alpha,\bm\beta;z)$ because, by assumption $\bm\alpha,\bm\beta$ belong to $(\mathbb{Z}_{(p)}^*\cap(0,1])^n$ and thus, Lemma~\ref{cyclic} implies $\mathfrak{D}_p^{l}(\bm\alpha)=\bm\alpha$ and $\mathfrak{D}_p^{l}(\bm\beta)=\bm\beta$. 

1). If $\#S_{\mathfrak{D}_p^{l-1}(\bm\alpha),\mathfrak{D}_p^{l-1}(\bm\beta),p}=1$ then, by Proposition~\ref{prop_third_step}, we have $$f_{l,0}(z)\equiv Q_{0,0}f_{l,0}^{p^l}\bmod p,$$
where $Q_{0,0}$ is a polynomial with coefficients in $\mathbb{Z}_{(p)}$ whose degree is less than $p^l$.  

2). Suppose that $\#S_{\mathfrak{D}_p^{l-1}(\bm\alpha),\mathfrak{D}_p^{l-1}(\bm\beta),p}=2$. Let us write $S_{\mathfrak{D}_p^{l-1}(\bm\alpha),\mathfrak{D}_p^{l-1}(\bm\beta),p}=\{r_0,r_1\}$ with $r_0=0$. Then, by Proposition  \ref{prop_third_step}, we have 
\begin{equation}\label{eq_tre}
f_{l,0}\equiv Q_{0,0}f_{l,0}^{p^l}+Q_{0,1}f_{l,r_1}^{p^{l}}\bmod p,
\end{equation}
\begin{equation}\label{eq_tt}
f_{l,r_1}\equiv Q_{1,0}f_{l,0}^{p^l}+Q_{1,1}f_{l,r_1}^{p^{l}}\bmod p,
\end{equation}
where $Q_{0,0}(z)$, $Q_{0,1}(z)$, $Q_{1,0}(z)$ and $Q_{1,1}(z)$ belong to $\mathbb{Z}_{(p)}[z]$ and their degrees are less than $p^l$.

If $Q_{0,1}(z)\bmod p$ is the zero polynomial then, from \eqref{eq_tre}, we have $f_{l,0}\equiv Q_{0,0}f_{l,0}^{p^l}\bmod p$. Now, suppose that $Q_{0,1}(z)\bmod p$ is not the zero polynomial. From Equations \eqref{eq_tre} and \eqref{eq_tt}, we have $$Q_{0,1}f_{l,r_1}-Q_{1,1}f_{l,0}\equiv (Q_{0,1}Q_{1,0}-Q_{1,1}Q_{0,0})f_{l,0}^{p^l}\bmod p.$$
Since $Q_{0,1}(z)\bmod p$ is not the zero polynomial, it follows from the previous equality that $$f_{l,r_1}\equiv\frac{Q_{0,1}Q_{1,0}-Q_{1,1}Q_{0,0}}{Q_{0,1}}f_{l,0}^{p^l}+\frac{Q_{1,1}}{Q_{0,1}}f_{l,0}\bmod p.$$
As the characteristic of $\mathbb{F}_p$ is $p$, then $$f_{l,r_1}^{p^l}\equiv\left(\frac{Q_{0,1}Q_{1,0}-Q_{1,1}Q_{0,0}}{Q_{0,1}}\right)^{p^l}f_{l,0}^{p^{2l}}+\left(\frac{Q_{1,1}}{Q_{0,1}}\right)^{p^l}f_{l,0}^{p^l}\bmod p.$$ By replacing the previous equality into \eqref{eq_tre}, we obtain
\begin{align*}
f_{l,0}&\equiv Q_{0,0}f_{l,0}^{p^l}+Q_{0,1}\left(\left(\frac{Q_{0,1}Q_{1,0}-Q_{1,1}Q_{0,0}}{Q_{0,1}}\right)^{p^l}f_{l,0}^{p^{2l}}+\left(\frac{Q_{1,1}}{Q_{0,1}}\right)^{p^l}f_{l,0}^{p^l}\right)\bmod p\\
&\equiv\left(Q_{0,0}+Q_{0,1}\left(\frac{Q_{1,1}}{Q_{0,1}}\right)^{p^l}\right)f_{l,0}^{p^l}+Q_{0,1}\left(\frac{Q_{0,1}Q_{1,0}-Q_{1,1}Q_{0,0}}{Q_{0,1}}\right)^{p^l}f_{l,0}^{p^{2l}}\bmod p.
\end{align*}
Since the degrees of $Q_{0,0}$, $Q_{0,1}$, $Q_{1,0}$ and $Q_{1,1}$ are less than or equal to $p^l-1$, we conclude that the height of $Q_{0,0}+Q_{0,1}\left(\frac{Q_{1,1}}{Q_{0,1}}\right)^{p^l}$ is less than $p^{2l}$ and that the height of the rational function $Q_{0,1}\left(\frac{Q_{0,1}Q_{1,0}-Q_{1,1}Q_{0,0}}{Q_{0,1}}\right)^{p^l}$ is less than $2p^{2l}$. 

3).  Suppose thta $t+1=\#S_{\mathfrak{D}_p^{l-1}(\bm\alpha),\mathfrak{D}_p^{l-1}(\bm\beta),p}$ with $t>1$. Let us write $S_{\mathfrak{D}_p^{l-1}(\bm\alpha),\mathfrak{D}_p^{l-1}(\bm\beta),p}=\{r_0,r_1,\ldots, r_t\}$ with $r_0=0$. Then, by Proposition~\ref{prop_third_step}, we have for all $i\in\{0,\ldots,t\}$, 
\begin{equation}\label{eq_system}
f_{l,r_i}\equiv Q_{i,0}f_{l,0}^{p^l}+\sum_{j=1}^tQ_{i,j}(z)f_{l,r_j}^{p^l}\bmod p,
\end{equation}
where, for every $j\in\{0,\ldots, t\}$, $Q_{i,j}(z)$ is a polynomial with coefficients in $\mathbb{Z}_{(p)}$ whose degree is less than $p^l$. 

If, for every $j\in\{1,\ldots, t\}$, $Q_{0,j}(z)\bmod p$ is the zero polynomial then, it follows from \eqref{eq_system} that $f_{l,0}\equiv Q_{0,0}(z)f_{l,0}^{p^l}\bmod p$.  Now, suppose that there is $j\in\{1,\ldots, t\}$ such that $Q_{0,j}\bmod p$ is not the zero polynomial. Without losing any generality we can assume that $Q_{0,t}\bmod p$ is not the zero polynomial. Then, by applying Proposition~\ref{lemm_system} to \eqref{eq_system}, it follows that, for all  $i\in\{0,\ldots, t-1\}$,  
\begin{equation}\label{eq_system_2}
f_{l,r_i}\equiv\sum_{k=1}^{2}T_{i,k,2}f_{l,0}^{p^{kl}}+\sum_{k=1}^{t-1}D_{i,k,2}f_{l,r_k}^{p^{2l}}\bmod p,
\end{equation}
where, $T_{i,1,2}, T_{i,2,2}$, $D_{i,1,2},\ldots D_{i,t-1,2}$ belong to $\mathbb{Q}(z)\cap\mathbb{Z}_{(p)}[[z]][z^{-1}]$ and their heights are less than $5(t+1)p^{2l}$.

Now, if for all $k\in\{1,\ldots,t-1\}$, $D_{0,k,2}\bmod p=0$ then, $f_{l,0}\equiv P_{1}f_{l,0}^{p^l}+P_2f_{l,0}^{p^{2l}}\bmod p$, where $P_1=T_{0,1,2}$ and $P_2=T_{0,2,2}$.

Now, suppose that there is $k\in\{1,\ldots,t-1\}$ such that $D_{0,k,2}\bmod p$ is not the zero polynomial. Without losing any generality we can assume that $D_{0,t-1,2}\bmod p$ is not the zero polynomial. Then, by applying Proposition~\ref{lemm_system} to \eqref{eq_system_2}, we infer that, for all  $i\in\{0,\ldots, t-2\}$,  
\begin{equation}\label{eq_system_3}
f_{l,r_i}\equiv\sum_{k=1}^{4}T_{i,k,3}f_{l,0}^{p^{kl}}+\sum_{k=1}^{t-2}D_{i,k,3}f_{l,r_k}^{p^{4l}}\bmod p,
\end{equation}
where, $T_{i,1,3},\ldots, T_{i,4,3}$, $D_{i,1,3},\ldots D_{i,t-2,3}$ belong to $\mathbb{Q}(z)\cap\mathbb{Z}_{(p)}[[z]][z^{-1}]$ and their heights are less than $5^2(t+1)tp^{4l}$.

 After making the previous process $t$-times we deduce that, $$f_{l,0}\equiv Q_{1}f_{l,0}^{p^l}+Q_{2}f_{l,0}^{p^{2l}}+\cdots+Q_{2^t}f_{l,0}^{p^{2^{t}l}}\bmod p,$$
where, for every $i\in\{1,\ldots, t\}$, $Q_{i}$ belongs to $\mathbb{Q}(z)\cap\mathbb{Z}_{(p)}[[z]][z^{-1}]$ and the height of $Q_i\bmod p$ is less than  $5^{t}(t+1)!p^{2^{t}l}$.

\end{proof}

\section{Proof of Proposition~\ref{lemm_system}}\label{sec_proof_lemm_system}
\begin{proof}
By hypotheses, for every $i\in\{0,\ldots, t-r\}$, we have
\begin{equation}\label{eq_system_0}
g_i=\sum_{k=1}^{s}P_{i,k}g_0^{p^{kl}}+\sum_{k=1}^{t-r}A_{i,k}g_k^{p^{sl}}.
\end{equation}
Then, for every $i\in\{1,\ldots, t-r\}$,  $$A_{0,t-r}g_i-A_{i,t-r}g_0=\sum_{k=1}^{s}(A_{0,t-r}P_{i,k}-A_{i,t-r}P_{0,k})g_{0}^{p^{kl}}+\sum_{k=1}^{t-r-1}(A_{0,t-r}A_{i,k}-A_{i,t-r}A_{0,k})g_{k}^{p^{sl}}.$$
By assumption, $A_{0,t-r}$ is not zero. Then, it follows from the last equality that, for every $i\in\{1,\ldots, t-r\}$, $$g_{i}=\sum_{k=1}^{s}Q_{i,k}g_{0}^{p^{kl}}+\sum_{k=1}^{t-r-1}B_{i,k}g_{k}^{p^{sl}}+C_ig_0,$$
where, $$Q_{i,k}=\frac{A_{0,t-r}P_{i,k}-A_{i,t-r}P_{0,k}}{A_{0,t-r}},\text{  }B_{i,k}=\frac{A_{0,t-r}A_{i,k}-A_{i,t-r}A_{0,k}}{A_{0,t-r}},\text{ and }C_i=\frac{A_{i,t-r}}{A_{0,t-r}}.$$
As the characteristic of $\mathbb{F}_p$ is $p$ then, for every $i\in\{1,\ldots, t-r\}$, $$g_{i}^{p^{sl}}=\sum_{k=1}^{s}Q_{i,k}^{p^{sl}}g_{0}^{p^{(s+k)l}}+\sum_{k=1}^{t-r-1}B_{i,k}^{p^{sl}}g_{k}^{p^{2sl}}+C_i^{p^{sl}}g_0^{p^{sl}}.$$
By substituting this last equality into \eqref{eq_system_0}, for every $i\in\{0,\ldots,t-r-1\}$,  we get

\begin{align*}
g_i=&\sum_{k=1}^{s-1}P_{i,k}g_0^{p^{kl}}+\left(P_{i,s}+\sum_{k=1}^{t-r}A_{i,k}C_k^{p^{sl}}\right)g_0^{p^{sl}}+\sum_{k=1}^{s}\left(\sum_{j=1}^{t-r}A_{i,j}Q_{j,k}^{p^{sl}}\right)g_0^{p^{(s+k)l}}\\
&+\sum_{k=1}^{t-r-1}\left(\sum_{j=1}^{t-r}A_{i,j}B_{j,k}^{p^{sl}}\right)g_k^{p^{2sl}}.
\end{align*}

For every $k\in\{1,\ldots, s-1\}$, we set $T_{i,k}=P_{i,k}$, for $k=s$, we set $T_{i,s}=P_{i,s}+\sum_{k=1}^{t-r}A_{i,k}C_k^{p^{sl}}$, for every $k\in\{s+1,\ldots, 2s\}$, we set $T_{i,k}=\sum_{j=1}^{t-r}A_{i,j}Q_{j,k-s}^{p^{sl}}$, and finally for every $k\in\{1,\ldots,t-r-1\}$, we set $D_{i,k}=\sum_{j=1}^{t-r}A_{i,j}B_{j,k}^{p^{sl}}$. So that, for every $i\in\{0,\ldots,t-r-1\}$, we have $$g_i=\sum_{k=1}^{2s}T_{i,k}g_0^{p^{kl}}+\sum_{k=1}^{t-r-1}D_{i,k}g_k^{p^{2sl}}.$$
Finally, we are going to see that, for all $i\in\{0,\ldots,t-r-1\}$, the heights of   $T_{i,1},\ldots, T_{i,2s}$, $D_{i,1},\ldots D_{i,t-r-1}$ are less than $5c(t-r+1)p^{2sl}$. In fact, if $k\in\{1,\ldots, s-1\}$ then, $T_{i,k}=P_{i,k}$.  By hypotheses, the height of $P_{i,k}$ is less than $cp^{sl}$. So, if $k\in\{1,\ldots, s-1\}$ then the height of $T_{i,k}$ is less than $cp^{sl}$. By definition, $T_{i,s}=P_{i,s}+\sum_{k=1}^{t-r}A_{i,k}C_k^{p^{sl}}.$ Recall that, for every $k\in\{1,\ldots,t-r-1\}$, $C_k=\frac{A_{k,t-r}}{A_{0,t-r}}$. Thus, the height of $C_k$ is less than $2cp^{sl}$ because, by hypotheses, the heights of $A_{k,t-r}$ and $A_{0,t-r}$ are less than $cp^{sl}$. So, for  every $k\in\{1,\ldots,t-r-1\}$, the height of $C_k^{p^{sl}}$ is less than $2cp^{2sl}$. Again, by hypotheses, the height of $A_{i,k}$ is $cp^{sl}$. Thus, the height of $A_{i,k}C_k^{p^{sl}}$ is less than $3cp^{2sl}$. Thus, the height of $T_{i,s}$ is less than $3c(t-r+1)p^{2sl}$. Now, we prove that,  for every $k\in\{s+1,\ldots, 2s\}$,  the height of $T_{i,k}$ is less than $5c(t-r)p^{2sl}$. By definition, $T_{i,k}=\sum_{j=1}^{t-r}A_{i,j}Q_{j,k-s}^{p^{sl}}$. The height of $Q_{j,k-s}$ is less than $4cp^{sl}$ because, $Q_{j,k-s}=P_{j,k-s}-(A_{j,t-r}P_{0,k-s})\big/(A_{0,t-r})$ and by hypotheses, the heights of $A_{0,t-r}, A_{j,t-r}, P_{j,k-s}, P_{0,k-s}$ are less that $cp^{sl}$. Thus, the height of $Q_{j,k-s}^{p^{sl}}$ is less than $4cp^{2sl}$. So, the height of $ A_{i,j}Q_{j,k-s}^{p^{sl}}$ is less than $5cp^{2sl}$. Whence, the height of $T_{i,k}$ is less than $5c(t-r)p^{2sl}$. Similarly, it follows that, for every $k\in\{1,\ldots,t-r-1\}$, the height of $D_{i,k}$ is $5c(t-r)p^{2sl}$. This completes the proof of our proposition.

\end{proof}

\section{Proof of Proposition~\ref{prop_third_step}}\label{sec_proof_prop_third_step}

The proof of Proposition~\ref{prop_third_step} relies on Lemmas \ref{lemm_aux} and \ref{lemm_second_step}.  
\begin{lemm}\label{lemm_aux}
Let $\bm\alpha=(\alpha_1,\ldots,\alpha_n)$, $\bm\beta=(\beta_1,\ldots,\beta_{n-1},1)$ be in $(\mathbb{Q}\setminus\mathbb{Z}_{\leq0})^n$ and let $p$ be a prime number such that $p$ does not divide $d_{\bm\alpha,\bm\beta}$. Suppose that $(\bm\alpha, \bm\beta)$ satisfies the $\textbf{P}_{p,l}$ property. Then:
 \begin{enumerate}[label=\Alph*)]
 \item for every $(a,r)\in\{1,\ldots, l\}\times\{0,\ldots, p-1\}$, the map $\sigma:S_{\bm\alpha_{a,r},\bm\beta_{a,r},p}\rightarrow S_{\mathfrak{D}_p^a(\bm\alpha),\mathfrak{D}_p^a(\bm\beta),p}$ given by $$\sigma(t)= \left\{ \begin{array}{lcc}
             t &   if  & t\equiv 1-\mathfrak{D}_p^a(\beta_{s})\bmod p  \text{ with }  s\in\mathcal{C}_{\mathfrak{D}_p^{a-1}(\bm\beta),r} \\
             \\ t+1 &  if & t\equiv-\mathfrak{D}_p^a(\beta_{s})\bmod p \text{ with } s\in\mathcal{P}_{\mathfrak{D}_p^{a-1}(\bm\beta),r}  \\
             \end{array}
   \right.$$
 is well-defined and is bijective. Moreover, its inverse $\tau:S_{\mathfrak{D}_p^a(\bm\alpha),\mathfrak{D}_p^a(\bm\beta),p}\rightarrow S_{\bm\alpha_{a,r},\bm\beta_{a,r},p}$ is given by $$\tau(t)= \left\{ \begin{array}{lcc}
             t &   if  & t\equiv 1-\mathfrak{D}_p^a(\beta_{s})\bmod p  \text{ with }  s\in\mathcal{C}_{\mathfrak{D}_p^{a-1}(\bm\beta),r} \\
             \\ t-1 &  if & t\equiv1-\mathfrak{D}_p^a(\beta_{s})\bmod p \text{ with } s\in\mathcal{P}_{\mathfrak{D}_p^{a-1}(\bm\beta),r};  \\
             \end{array}
   \right.$$
      
 \item for every $(a,r)\in\{1,\ldots, l\}\times\{0,\ldots, p-1\}$, the following equalities hold for every $t\in S_{\bm\alpha_{a,r},\bm\beta_{a,r},p}$,  $\mathcal{P}_{\bm\alpha_{a,r},t}=\mathcal{P}_{\mathfrak{D}_p^a(\bm\alpha),\sigma(t)}$, $\mathcal{C}_{\bm\alpha_{a,r},t}=\mathcal{C}_{\mathfrak{D}_p^a(\bm\alpha),\sigma(t)}$,  $\mathcal{P}_{\bm\beta_{a,r},t}=\mathcal{P}_{\mathfrak{D}_p^a(\bm\beta),\sigma(t)}$, and  $\mathcal{C}_{\bm\beta_{a,r},t}=\mathcal{C}_{\mathfrak{D}_p^a(\bm\beta),\sigma(t)}$.
 
 \end{enumerate}
\end{lemm}

\begin{lemm}\label{lemm_second_step}
Let $\bm\alpha=(\alpha_1,\ldots,\alpha_n)$, $\bm\beta=(\beta_1,\ldots,\beta_{n-1},1)$ be in $(\mathbb{Q}\setminus\mathbb{Z}_{\leq0})^n$, let $p$ be a prime number such that $p>d_{\bm\alpha,\bm\beta}$ and $f(z):={}_nF_{n-1}(\bm\alpha,\bm\beta;z)$ belongs to $\mathbb{Z}_{(p)}[[z]]$. Suppose that $(\bm\alpha,\bm\beta)$ satisfies the $\textbf{P}_{p,l}$ property, where $l$ is the order of $p$ in $(\mathbb{Z}/d_{\bm\alpha,\bm\beta}\mathbb{Z})^*$. Then, for every $a\in\{1,\ldots,l\}$ and for every $r\in S_{\mathfrak{D}_p^{a-1}(\bm\alpha),\mathfrak{D}_p^{a-1}(\bm\beta),p}$, $f_{a,r}\in1+z\mathbb{Z}_{(p)}[[z]]$ and $$f\equiv\sum_{r\in S_{\mathfrak{D}_p^{a-1}(\bm\alpha),\mathfrak{D}_p^{a-1}(\bm\beta),p}} Q_{a,r}(z)f_{a,r}^{p^a}\bmod p,$$
where, for every $r\in S_{\mathfrak{D}_p^{a-1}(\bm\alpha),\mathfrak{D}_p^{a-1}(\bm\beta),p}$, $Q_{a,r}(z)$ belongs to $\mathbb{Z}_{(p)}[z]$ and has degree less than $p^a$.
 \end{lemm}
 
Section \ref{sec_proof_lemm_aux} is devoted to proving Lemma  \ref{lemm_aux} and  Lemma~\ref{lemm_second_step} will be proved in Section \ref{sec_proof_lemm_second_step}. The following remarks are useful in the proofs of Proposition~\ref{prop_third_step} and Lemmas~\ref{lemm_second_step} and \ref{lemm_explicit_poly}.
\begin{rema}\label{rema_d_p_1}\hfill
\begin{enumerate}
\item If $\gamma\in\mathbb{Z}_{(p)}^*$ then $\mathfrak{D}_p(\gamma+1)=\mathfrak{D}_p(\gamma)$. Indeed, $p\mathfrak{D}_p(\gamma)-\gamma$ belongs to $\{1,\ldots,p-1\}$ because $\gamma\in\mathbb{Z}_{(p)}^*$. Whence, $p\mathfrak{D}_p(\gamma)-\gamma-1$ belongs to $\{0,\ldots,p-1\}$. So, $\mathfrak{D}_p(\gamma+1)=\mathfrak{D}_p(\gamma)$.

\item Let $p$ be a prime number and let $\bm\alpha=(\alpha_1,\ldots,\alpha_n)$, $\bm\beta=(\beta_1,\ldots,\beta_{n-1},1)$ be in $(\mathbb{Z}_{(p)})^n$. Suppose that $(\bm\alpha$, $\bm\beta)$ satisfies the $\textbf{P}_{p,l}$ property. In this remark we show that, for all $(a,r)\in\{1,\ldots, l-1\}\times\{0,\ldots,p-1\}$,  $\mathfrak{D}_p(\bm\alpha_{a,r})=\mathfrak{D}_p^{a+1}(\bm\alpha)$ and $\mathfrak{D}_p(\bm\beta_{a,r})=\mathfrak{D}_p^{a+1}(\bm\beta)$.  As $(\bm\alpha,\bm\beta)$ satisfies the $\textbf{P}_{p,l}$ property and $a\in\{1,\ldots,l-1\}$ then, from $(\textbf{P}1)$, we know that $\mathfrak{D}_p^a(\bm\alpha)$, $\mathfrak{D}_p^a(\bm\beta)$, belong to $(\mathbb{Z}_{(p)}^*)^n$. Let $w$ be in $\{1,\ldots, n\}$. If $w\in\mathcal{C}_{\mathfrak{D}_p^{a-1}(\bm\alpha),r}$ then $\alpha_{w,a,r}=\mathfrak{D}_p^a(\alpha_w)$. Thus, $\mathfrak{D}_p(\alpha_{w,a,r})=\mathfrak{D}_p^{a+1}(\alpha_w)$. Now, if $w\in\mathcal{P}_{\mathfrak{D}_p^{a-1}(\bm\alpha),r}$ then $\alpha_{w,a,r}=\mathfrak{D}_p^a(\alpha_w)+1$. Hence, by (1), $\mathfrak{D}_p(\mathfrak{D}_p^a(\alpha_{w})+1)=\mathfrak{D}_p^{a+1}(\alpha_{w})$ because $\mathfrak{D}_p^a(\alpha_{w})$ belongs to $\mathbb{Z}_{(p)}^*$. Therefore,  $\mathfrak{D}_p(\bm\alpha_{a,r})=\mathfrak{D}_p^{a+1}(\bm\alpha)$. In a similar fashion, it follows that $\mathfrak{D}_p(\bm\beta_{a,r})=\mathfrak{D}_p^{a+1}(\bm\beta)$.
\end{enumerate}
\end{rema}

\begin{rema}\label{rema_d_p_h}\hfill

Let $p$ be a prime number and let $\bm\alpha=(\alpha_1,\ldots,\alpha_n)$, $\bm\beta=(\beta_1,\ldots,\beta_{n-1},1)$ be in $(\mathbb{Z}_{(p)})^n$. Suppose that $(\bm\alpha$, $\bm\beta)$ satisfies the $\textbf{P}_{p,l}$ property. The goal of this remark is to show that, for all $(a,r)\in\{1,\ldots, l-1\}\times\{0,\ldots,p-1\}$, $(\bm\alpha_{a,r},\bm\beta_{a,r})$ satisfies the $\textbf{P}_{p,l'}$ property, where $l'$ is the order of $p$ in $(\mathbb{Z}/d_{\bm\alpha_{a,r},\bm\beta_{a,r}}\mathbb{Z})^*$. For this purpose, we will first show that, for any $1\leq k\leq l$, there is $1\leq s\leq l$ such that $\mathfrak{D}^k_p(\bm\alpha_{a,r})=\mathfrak{D}_p^s(\bm\alpha)$ and $\mathfrak{D}^k_p(\bm\beta_{a,r})=\mathfrak{D}_p^s(\bm\beta)$. From (2) of Remark~\ref{rema_d_p_1}, we have $\mathfrak{D}_p(\bm\alpha_{a,r})=\mathfrak{D}_p^{a+1}(\bm\alpha)$ and $\mathfrak{D}_p(\bm\beta_{a,r})=\mathfrak{D}_p^{a+1}(\bm\beta)$. Consequently, for all $k\in\{1,\ldots, l\}$, we get $\mathfrak{D}^k_p(\bm\alpha_{a,r})=\mathfrak{D}_p^{a+k}(\bm\alpha)$ and $\mathfrak{D}^k_p(\bm\beta_{a,r})=\mathfrak{D}_p^{a+k}(\bm\beta)$. Let $k$ be in $\{1,\ldots, l\}$ and let us write $a+k=s+tl$ with $0\leq s<l$. We have $p^l=1\bmod d_{\mathfrak{D}_p^k(\bm\alpha),\mathfrak{D}_p^k(\bm\beta)}$ because, from the definition of $\mathfrak{D}_p$, it follows that $d_{\mathfrak{D}_p^k(\bm\alpha),\mathfrak{D}_p^k(\bm\beta)}$ divides $d_{\bm\alpha,\bm\beta}$.\footnote{Let $\alpha$ be in $\mathbb{Z}_{(p)}$. Then, from the definition of $\mathfrak{D}_p$, it follows that the denominator of $\mathfrak{D}_p(\alpha)$ is a factor of the denominator of $\alpha$.} Further, from $(\textbf{P}1)$ we have $\mathfrak{D}_p^k(\bm\alpha)$, $\mathfrak{D}_p^k(\bm\beta)\in\mathbb{Z}_{(p)}^*\cap(0,1]$. Then, by Lemma~\ref{cyclic}, we get that, for all $m\geq1$, $\mathfrak{D}_{p}^{ml}(\mathfrak{D}_p^k(\bm\alpha))=\mathfrak{D}_p^k(\bm\alpha)$ and  $\mathfrak{D}_{p}^{ml}(\mathfrak{D}_p^k(\bm\beta))=\mathfrak{D}_p^k(\bm\beta)$. So, $\mathfrak{D}^k_p(\bm\alpha_{a,r})=\mathfrak{D}_p^s(\bm\alpha)$ if $s\neq0$ and if $s=0$, we have $\mathfrak{D}^k_p(\bm\alpha_{a,r})=\mathfrak{D}_p^{tl}(\bm\alpha)=\mathfrak{D}_p^{(t-1)l}(\mathfrak{D}_p^l(\bm\alpha))=\mathfrak{D}_p^l(\bm\alpha)$. Similarly, we have $\mathfrak{D}^k_p(\bm\beta_{a,r})=\mathfrak{D}_p^s(\bm\beta)$ if $s\neq0$ and $\mathfrak{D}^k_p(\bm\beta_{a,r})=\mathfrak{D}_p^l(\bm\beta)$ if $s=0$. Consequently, $(\bm\alpha_{a,r},\bm\beta_{a,r})$ satisfies the $\textbf{P}_{p,l}$ property. Finally, from the definition of $\mathfrak{D}_p$ again, it immediately follows that $d_{\bm\alpha_{a,r},\bm\beta_{a,r}}$ divides $d_{\bm\alpha,\bm\beta}$. Hence, if $l'$ is the order of $p$ in $(\mathbb{Z}/d_{\bm\alpha_{a,r},\bm\beta_{a,r}}\mathbb{Z})^*$ then $l'$ divides $l$ and therefore, $l'\leq l$. So, $(\bm\alpha_{a,r},\bm\beta_{a,r})$ satisfies the $\textbf{P}_{p,l'}$ property. 
\end{rema}

\begin{rema}\label{rema_p_l_i}
Let $\bm\alpha=(\alpha_1,\ldots,\alpha_n)$, $\bm\beta=(\beta_1,\ldots,\beta_{n-1},1)$ be in $(\mathbb{Q}\cap(0,1])^n$ and let $p$ be a prime number such that $\bm\alpha$, $\bm\beta$ belong to $(\mathbb{Z}_{(p)}^*)^n$ and  let $l$ be the order of $p$ in $(\mathbb{Z}/d_{\bm\alpha,\bm\beta}\mathbb{Z})^*$. 
\begin{enumerate}
\item We show that, for all integers $m\geq1$ and $r\in\{0,\ldots, p-1\}$, $\mathfrak{D}_p^m(\bm\alpha_{l,r})=\mathfrak{D}_p^m(\bm\alpha)$ and $\mathfrak{D}_p^m(\bm\beta_{l,r})=\mathfrak{D}_p^m(\bm\beta)$. Indeed, by assumption, $\bm\alpha$, $\bm\beta$ belong to $(\mathbb{Z}_{(p)}^*)^n$ and thus, according to Lemma~\ref{cyclic}, we get that $\mathfrak{D}_p^l(\bm\alpha)=\bm\alpha$ and $\mathfrak{D}_p^l(\bm\beta)=\bm\beta$. Therefore, for $1\leq s\leq n$, $\alpha_{s,l,r}=\mathfrak{D}_p^l(\alpha_s)=\alpha_s$ if $s\in\mathcal{C}_{\mathfrak{D}_p^{l-1}(\bm\alpha),r}$ or $\alpha_{s,l,r}=\mathfrak{D}_p^l(\alpha_s)+1=\alpha_s+1$ if $s\in\mathcal{P}_{\mathfrak{D}_p^{l-1}(\bm\alpha),r}$. From (1) of Remark~\ref{rema_d_p_1}, it follows that, for all $\gamma\in\{\alpha_1,\ldots, \alpha_n,\beta_1,\ldots,\beta_n\}$, $\mathfrak{D}_p(\gamma+1)=\mathfrak{D}_p(\gamma)$. Consequently, for all integers $m\geq1$, $\mathfrak{D}_p^m(\alpha_{s,l,r})=\mathfrak{D}_p^m(\alpha_s)$ for all $1\leq s\leq n$. In a similar fashion, one gets that, for all integers $m\geq1$, $\mathfrak{D}_p^m(\beta_{s,l,r})=\mathfrak{D}_p^m(\beta_s)$ for all $1\leq s\leq n$. So that, for all $m\geq1$, $\mathfrak{D}_p^m(\bm\alpha_{l,r})=\mathfrak{D}_p^m(\bm\alpha)$ and $\mathfrak{D}_p^m(\bm\beta_{l,r})=\mathfrak{D}_p^m(\bm\beta)$. 

\item As an immediately consequence of (1), we get that if $(\bm\alpha, \bm\beta)$ satisfies the  $\textbf{P}_{p,l}$ property then, for all $r\in\{0,\ldots,p-1\}$, $(\bm\alpha_{l,r},\bm\beta_{l,r})$ satisfies also the $\textbf{P}_{p,l}$ property. Furthermore, it is easily seen that $d_{\bm\alpha_{l,r},\bm\beta_{l,r}}=d_{\bm\alpha,\bm\beta}$. Thus, $l$ is also the order of $p$ in $(\mathbb{Z}/d_{\bm\alpha_{l,r},\bm\beta_{l,r}}\mathbb{Z})^*$.
\end{enumerate}
\end{rema}

We can now prove Proposition~\ref{prop_third_step}.

\begin{proof}[Proof of Proposition~\ref{prop_third_step}]
% For each $j\in S_{\mathfrak{D}_p^{l-1}(\bm\alpha),\mathfrak{D}_p^{l-1}(\bm\beta),p}$, we set  $$f_{j}=\sum_{m\geq0}\left(\frac{\underset{s\in\mathcal{C}_{\mathfrak{D}_p^{l-1}(\bm\alpha),j}}{\prod}(\alpha_s)_m {\underset{s\in\mathcal{P}_{\mathfrak{D}_p^{l-1}(\bm\alpha),j}}{\prod}(\alpha_s+1)_m}}{\underset{s\in\mathcal{C}_{\mathfrak{D}_p^{l-1}(\bm\beta),j}}{\prod}(\beta_s)_m \underset{s\in\mathcal{P}_{\mathfrak{D}_p^{l-1}(\bm\beta),j}}{\prod}(\beta_s+1)_m}\right)z^m.$$
Note that $f_{l,0}$ is the hypergeometric series $_nF_{n-1}(\bm\alpha,\bm\beta;z)$ because, by assumption $\bm\alpha,\bm\beta$ belong to $(\mathbb{Z}_{(p)}^*\cap(0,1])^n$ and thus, Lemma~\ref{cyclic} implies $\mathfrak{D}_p^{l}(\bm\alpha)=\bm\alpha$ and $\mathfrak{D}_p^{l}(\bm\beta)=\bm\beta$. We first prove that, for all $j\in S_{\mathfrak{D}_p^{l-1}(\bm\alpha),\mathfrak{D}_p^{l-1}(\bm\beta),p}$, $f_{l,j}\in 1+z\mathbb{Z}_{(p)}[[z]]$. By assumption,  $(\bm\alpha,\bm\beta)$ satisfies the $\textbf{P}_{p,l}$ property and $f_{l,0}$ belongs to $\mathbb{Z}_{(p)}[[z]]$. By applying Lemma~\ref{lemm_second_step} to $f_{l,0}(z)$ we get that, for every $j\in S_{\mathfrak{D}_p^{l-1}(\bm\alpha),\mathfrak{D}_p^{l-1}(\bm\beta),p}$, $f_{l,j}\in1+z\mathbb{Z}_{(p)}[[z]]$.

%$$f_{l,j}=\sum_{m\geq0}\left(\frac{\underset{s\in\mathcal{C}_{\mathfrak{D}_p^{l-1}(\bm\alpha),j}}\prod(\mathfrak{D}_p^{l}(\alpha_s))_m\underset{s\in\mathcal{P}_{\mathfrak{D}_p^{l-1}(\bm\alpha),j}}{\prod}(\mathfrak{D}_p^l(\alpha_s)+1)_m}{\underset{s\in\mathcal{C}_{\mathfrak{D}_p^{l-1}(\bm\beta),j}}\prod(\mathfrak{D}_p^{l}(\beta_s))_m\underset{s\in\mathcal{P}_{\mathfrak{D}_p^{l-1}(\bm\beta),j}}{\prod}(\mathfrak{D}_p^l(\beta_s)+1)_m}\right)z^m\in1+z\mathbb{Z}_{(p)}[[z]].$$

Let $i$ be an arbitrary element in $S_{\mathfrak{D}_p^{l-1}(\bm\alpha),\mathfrak{D}_p^{l-1}(\bm\beta),p}$. We are going to prove that $$f_{l,i}\equiv\sum_{j\in S_{\mathfrak{D}_p^{l-1}(\bm\alpha),\mathfrak{D}_p^{l-1}(\bm\beta),p}} Q_{i,j}f_{l,j}^{p^l}\bmod p,$$
where, each $Q_{i,j}$ belongs to $\mathbb{Z}_{(p)}[z]$ with degree less than $p^l$. For this purpose, we are going to see that we can apply Lemma~\ref{lemm_second_step} to $f_{l,i}$. By definition $f_{l,i}$ is the hypergeometric series $_nF_{n-1}(\bm\alpha_{l,i},\bm\beta_{l,i};z)$.  By (2) of Remark~\ref{rema_p_l_i}, we know that $l$ is also the order of $p$ in $(\mathbb{Z}/d_{\bm\alpha_{l,i},\bm\beta_{l,i}}\mathbb{Z})^*$ and that $(\bm\alpha_{l,i},\bm\beta_{l,i})$ satisfies the $\textbf{P}_{p,l}$ property. Further, we also have $f_{l,i}\in\mathbb{Z}_{(p)}[[z]]$. So we are in a position to apply Lemma~\ref{lemm_second_step} to $f_{l,i}$ and therefore,
  \begin{equation}\label{eq_f_i}
 f_{l,i}\equiv\sum_{j\in S_{\mathfrak{D}_p^{l-1}(\bm\alpha_{l,i}),\mathfrak{D}_p^{l-1}(\bm\beta_{l,i}),p}} Q_{i,j}g_{i,j}^{p^l}\bmod p,
  \end{equation}
where, $Q_{i,j}(z)\in\mathbb{Z}_{(p)}[z]$ has degree less than $p^l$ and $$g_{i,j}=\sum_{m\geq0}\left(\frac{\underset{s\in\mathcal{C}_{\mathfrak{D}_p^{l-1}(\bm\alpha_{l,i}),j}}\prod(\mathfrak{D}_p^{l}(\alpha_{s,l,i}))_m\underset{s\in\mathcal{P}_{\mathfrak{D}_p^{l-1}(\bm\alpha_{l,i}),j}}{\prod}(\mathfrak{D}_p^l(\alpha_{s,l,i})+1)_m}{\underset{s\in\mathcal{C}_{\mathfrak{D}_p^{l-1}(\bm\beta_{l,i}),j}}\prod(\mathfrak{D}_p^{l}(\beta_{s,l,i}))_m\underset{s\in\mathcal{P}_{\mathfrak{D}_p^{l-1}(\bm\beta_{l,i}),j}}{\prod}(\mathfrak{D}_p^l(\beta_{s,l,i})+1)_m}\right)z^m\in\mathbb{Z}_{(p)}[[z]].$$
  We have already seen that $\mathfrak{D}_p^{l}(\bm\alpha)=\bm\alpha$. Thus, $\alpha_{s,l,i}=\alpha_s$ if $s\in\mathcal{C}_{\mathfrak{D}_p^{l-1}(\bm\alpha_{l,i}),j}$ and $\alpha_{s,l,i}=\alpha_s$+1 if $s\in\mathcal{P}_{\mathfrak{D}_p^{l-1}(\bm\alpha_{l,i}),j}$. Since, by assumption $\bm\alpha$ belongs to $(\mathbb{Z}_{(p)}^*\cap(0,1])^n$, by (1) of Remark~\ref{rema_p_l_i}, we deduce that $\mathfrak{D}_p^l(\alpha_{s,l,i})=\mathfrak{D}_p^l(\alpha_s)=\alpha_s$. In a similar way, one obtains  $\mathfrak{D}_p^l(\beta_{s,l,i})=\mathfrak{D}_p^l(\beta_s)=\beta_s$. Hence,
$$
g_{i,j}=\sum_{m\geq0}\left(\frac{\underset{s\in\mathcal{C}_{\mathfrak{D}_p^{l-1}(\bm\alpha_{l,i}),j}}\prod(\alpha_{s})_m\underset{s\in\mathcal{P}_{\mathfrak{D}_p^{l-1}(\bm\alpha_{l,i}),j}}{\prod}(\alpha_{s}+1)_m}{\underset{s\in\mathcal{C}_{\mathfrak{D}_p^{l-1}(\bm\beta_{l,i}),j}}\prod(\beta_{s})_m\underset{s\in\mathcal{P}_{\mathfrak{D}_p^{l-1}(\bm\beta_{l,i}),j}}{\prod}(\beta_{s}+1)_m}\right)z^m.\\
$$
Suppose that $l\geq2$. We want to see that $g_{i,j}=f_{l,j}$. Since $l\geq2$, by (1) of Remark~\ref{rema_p_l_i}, we know that $\mathfrak{D}_p^{l-1}(\bm\alpha_{l,i})=\mathfrak{D}_p^{l-1}(\bm\alpha)$ and that $\mathfrak{D}_p^{l-1}(\bm\beta_{l,i})=\mathfrak{D}_p^{l-1}(\bm\beta)$. Therefore, we have the equality, $S_{\mathfrak{D}_p^{l-1}(\bm\alpha_{l,i}),\mathfrak{D}_p^{l-1}(\bm\beta_{l,i}),p}=S_{\mathfrak{D}_p^{l-1}(\bm\alpha),\mathfrak{D}_p^{l-1}(\bm\beta),p}$. Whence, $g_{i,j}=f_{l,j}$ for all $j\in S_{\mathfrak{D}_p^{l-1}(\bm\alpha_{l,i}),\mathfrak{D}_p^{l-1}(\bm\beta_{l,i}),p}$. So, from Equation~\eqref{eq_f_i}, we get $$f_{l,i}\equiv\sum_{j\in S_{\mathfrak{D}_p^{l-1}(\bm\alpha),\mathfrak{D}_p^{l-1}(\bm\beta),p}} Q_{i,j}f_{l,j}^{p^l}\bmod p.$$

Suppose now that  $l=1$. We want to see that $g_{i,j}=f_{1,\sigma(j)}$, where $\sigma:S_{\bm\alpha_{l,i},\bm\beta_{l,i},p}\rightarrow S_{\bm\alpha,\bm\beta,p}$ is the map given by Lemma~\ref{lemm_aux}. Since $l=1$, it is clear that $S_{\mathfrak{D}_p^{l-1}(\bm\alpha_{l,i}),\mathfrak{D}_p^{l-1}(\bm\beta_{l,i}),p}=S_{\bm\alpha_{l,i},\bm\beta_{l,i},p}$. Since $l$ is the order of $p$ in $(\mathbb{Z}/d_{\bm\alpha,\bm\beta}\mathbb{Z})^*$ and by hypotheses, $\bm\alpha$ and $\bm\beta$ belong to $(\mathbb{Z}^*_{(p)})^n$, by using Lemma~\ref{cyclic}, we obtain $\mathfrak{D}_p(\bm\alpha)=\bm\alpha$ and $\mathfrak{D}_p(\bm\beta)=\bm\beta$. Further, $p$ does not divide $d_{\bm\alpha,\bm\beta}$ because, by assumption, for all $i,j\in\{1,\ldots,n\}$, $\alpha_i$, $\beta_j$ belong to $\mathbb{Z}_{(p)}^{*}$. We also have, by assumption again, $(\bm\alpha, \bm\beta)$ satisfies the $\textbf{P}_{p,1}$ property. So, by B) of Lemma~\ref{lemm_aux}, we infer that, for all  $j\in S_{\bm\alpha_{l,i},\bm\beta_{l,i},p}$, we have  $\mathcal{P}_{\bm\alpha_{l,i},j}=\mathcal{P}_{\bm\alpha,\sigma(j)}$, $\mathcal{C}_{\bm\alpha_{l,i},j}=\mathcal{C}_{\bm\alpha,\sigma(j)}$, $\mathcal{P}_{\bm\alpha_{l,i},j}=\mathcal{P}_{\bm\alpha,\sigma(j)}$, and $\mathcal{C}_{\bm\alpha_{l,i},j}=\mathcal{C}_{\bm\alpha,\sigma(j)}$. Thus, for all $j\in S_{\bm\alpha_{l,i},\bm\beta_{l,i},p}$, $g_{i,j}=f_{1,\sigma(j)}$. Since, by A) of Lemma~\ref{lemm_aux}, $\sigma:S_{\bm\alpha_{l,i},\bm\beta_{l,i},p}\rightarrow S_{\bm\alpha,\bm\beta,p}$ is a bijective map, it follows from Equation~\eqref{eq_f_i} that  $$f_{1,i}\equiv\sum_{j\in S_{\bm\alpha,\bm\beta,p}} Q_{i,\tau(j)}f_{1,j}^{p}\bmod p.$$
 This completes the proof because $i$ is an arbitrary element in $S_{\mathfrak{D}_p^{l-1}(\bm\alpha),\mathfrak{D}_p^{l-1}(\bm\beta),p}$.
\end{proof}

\section{Proof of Lemma~\ref{lemm_aux}}\label{sec_proof_lemm_aux}
     
A)  Let $t$ be in $S_{\bm\alpha_{a,r},\bm\beta_{a,r},p}$.  Then $t\bmod p\equiv1-\beta_{s,a,r}\bmod p$ for some $s\in\{1,\ldots,n\}$ and $v_p(\mathcal{Q}_{\bm\alpha_{a,r},\bm\beta_{a,r}}(t))=0$. We are going to see that $\sigma$ is well-defined. For this purpose, we first show that if there exists $s'\in\{1,\ldots,n\}$ such that $t\bmod p\equiv1-\beta_{s',a,r}\bmod p$ then $s\in\mathcal{C}_{\mathfrak{D}_p^{a-1}(\bm\beta),r}$ if and only if $s'\in\mathcal{C}_{\mathfrak{D}_p^{a-1}(\bm\beta),r}$ and second, we prove that $\sigma(t)\in S_{\mathfrak{D}^a_p(\bm\alpha),\mathfrak{D}^a_p(\bm\beta),p}$. Suppose that $s\in\mathcal{C}_{\mathfrak{D}_p^{a-1}(\bm\beta),r}$. Assume for contradiction that $s'\in\mathcal{P}_{\mathfrak{D}_p^{a-1}(\bm\beta),r}$. Thus, $\beta_{s,a,r}=\mathfrak{D}^a_p(\beta_s)$ and $\beta_{s',a,r}=\mathfrak{D}^a_p(\beta_{s'})+1$. Then $\mathfrak{D}^{a}_p(\beta_s)\equiv\mathfrak{D}^{a}_p(\beta_{s'})+1\bmod p$ because $1-\beta_{s,a,r}\bmod p\equiv t\bmod p\equiv1-\beta_{s',a,r}\bmod p$. Hence, $1+\mathfrak{D}^{a}_p(\beta_{s'})-\mathfrak{D}^{a}_p(\beta_s)\in p\mathbb{Z}_{p}$. That is a contradiction to $(\textbf{P}5)$. Therefore, we have $s'\in\mathcal{C}_{\mathfrak{D}_p^{a-1}(\bm\beta),r}$. In a similar way, one shows that if $s'\in\mathcal{C}_{\mathfrak{D}_p^{a-1}(\bm\beta),r} $ then $s\in\mathcal{C}_{\mathfrak{D}_p^{a-1}(\bm\beta),r}$. 

We now prove that $\sigma(t)\in S_{\mathfrak{D}^a_p(\bm\alpha),\mathfrak{D}^a_p(\bm\beta),p}$. By definition, $t\bmod p\equiv1-\beta_{s,a,r}\bmod p$ for some $s\in\{1,\ldots,n\}$. Thus $$t\bmod p= \left\{ \begin{array}{lcc}
             1-\mathfrak{D}_p^a(\beta_s)\bmod p &   if  &  s\in\mathcal{C}_{\mathfrak{D}_p^{a-1}(\bm\beta),r} \\
             \\  -\mathfrak{D}_p^a(\beta_s)\bmod p&  if & s\in\mathcal{P}_{\mathfrak{D}_p^{a-1}(\bm\beta),r}.  \\
             \end{array}
   \right.$$

\textbullet\quad Suppose that $s\in\mathcal{C}_{\mathfrak{D}_p^{a-1}(\bm\beta),r}$. Then, $t\bmod p=1-\mathfrak{D}_p^a(\beta_{s})\bmod p$ and therefore, $t\in E_{\mathfrak{D}_p^a(\bm\alpha),\mathfrak{D}_p^a(\bm\beta),p}$. We now show that $t\in S_{\mathfrak{D}_p^a(\bm\alpha),\mathfrak{D}_p^a(\bm\beta),p}$. It is clear that we have the following equality
 \begin{equation}\label{eq_5}
 \mathcal{Q}_{\bm\alpha_{a,r},\bm\beta_{a,r}}(t)=\mathcal{Q}_{\mathfrak{D}_p^a(\bm\alpha),\mathfrak{D}_p^a(\bm\beta)}(t)\cdot\frac{\underset{w\in\mathcal{P}_{\mathfrak{D}_p^{a-1}(\bm\alpha),r}}\prod(\mathfrak{D}_p^{a}(\alpha_w)+t)}{\underset{w\in\mathcal{P}_{\mathfrak{D}_p^{a-1}(\bm\beta),r}}\prod(\mathfrak{D}_p^{a}(\beta_w)+t)}\cdot\frac{\underset{w\in\mathcal{P}_{\mathfrak{D}_p^{a-1}(\bm\beta),r}}\prod\mathfrak{D}_p^{a}(\beta_w)}{\underset{w\in\mathcal{P}_{\mathfrak{D}_p^{a-1}(\bm\alpha),r}}\prod\mathfrak{D}_p^{a}(\alpha_w)}.
 \end{equation}
 By (\textbf{P}1), we know that $\mathfrak{D}_p^a(\beta_{1}),\ldots,\mathfrak{D}_p^a(\beta_{n}),\mathfrak{D}_p^a(\alpha_{1}),\ldots,\mathfrak{D}_p^a(\alpha_{n})$ belong to $\mathbb{Z}_{(p)}^*$. Then,
 $$\left(\underset{w\in\mathcal{P}_{\mathfrak{D}_p^{a-1}(\bm\beta),r}}\prod\mathfrak{D}_p^{a}(\beta_w)\right)\bigg/\left(\underset{w\in\mathcal{P}_{\mathfrak{D}_p^{a-1}(\bm\alpha),r}}\prod\mathfrak{D}_p^{a}(\alpha_w)\right)\in\mathbb{Z}_{(p)}^*.$$
 Now, assume for contradiction that there is $\gamma\in\{\alpha_{1},\ldots,\alpha_{n},\beta_{1},\ldots,\beta_{n}\}$ such that $\mathfrak{D}_p^a(\gamma)+t$ belongs to $p\mathbb{Z}_{(p)}$. Then $t=p\mathfrak{D}_p^{a+1}(\gamma)-\mathfrak{D}_p^a(\gamma)$ because $0\leq t<p$. As $t\bmod p=1-\mathfrak{D}_p^a(\beta_{s})\bmod p$ then $1-\mathfrak{D}_p^a(\beta_{s})+\mathfrak{D}_p^a(\gamma)\in p\mathbb{Z}_{(p)}$, which is a contradiction to $(\textbf{P}5)$. Consequently,  $\mathfrak{D}_p^a(\alpha_{1})+t,\ldots,\mathfrak{D}_p^a(\alpha_{n})+t,\mathfrak{D}_p^a(\beta_{1})+t,\ldots,\mathfrak{D}_p^a(\beta_{n})+t$ belong to $\mathbb{Z}_{(p)}^*$. Therefore, 
 $$\left(\underset{w\in\mathcal{P}_{\mathfrak{D}_p^{a-1}(\bm\alpha),r}}\prod(\mathfrak{D}_p^{a}(\alpha_w)+t)\right)\bigg/\left(\underset{w\in\mathcal{P}_{\mathfrak{D}_p^{a-1}(\bm\beta),r}}\prod(\mathfrak{D}_p^{a}(\beta_w)+t)\right)\in\mathbb{Z}_{(p)}^*.$$
 Thus, from \eqref{eq_5}, we get $v_p(\mathcal{Q}_{\mathfrak{D}_p^a(\bm\alpha),\mathfrak{D}_p^a(\bm\beta)}(t))=0$ because $v_p(\mathcal{Q}_{\bm\alpha_{a,r},\bm\beta_{a,r}}(t))=0$. So that $t\in S_{\mathfrak{D}_p^a(\bm\alpha),\mathfrak{D}_p^a(\bm\beta),p}$.
 
\textbullet\quad Suppose that $s\in\mathcal{P}_{\mathfrak{D}_p^{a-1}(\bm\beta),r}$. Then, $t\bmod p\equiv-\mathfrak{D}_p^a(\beta_{s})\bmod p$.

We first prove that $t+1\in E_{\mathfrak{D}_p^a(\bm\alpha),\mathfrak{D}_p^a(\bm\beta),p}$. As $0\leq r<p$ then, $(1)_{r}\notin p\mathbb{Z}_{(p)}$. By assumption, $s\in\mathcal{P}_{\mathfrak{D}_p^{a-1}(\bm\beta),r}$ and thus, $(\mathfrak{D}_p^{a-1}(\beta_s))_{r}\in p\mathbb{Z}_{(p)}$. Hence, $\beta_s\neq 1$. Now, assume for contradiction that $t=p-1$. Since $t\bmod p=-\mathfrak{D}_p^a(\beta_{j})\bmod p$, we have $p-1+\mathfrak{D}_p^a(\beta_{s})\in p\mathbb{Z}_{(p)}$. Therefore, we have $p\mathfrak{D}_p^{a+1}(\beta_{s})-\mathfrak{D}_p^a(\beta_{s})=p-1$. Since $\beta_{s}\neq1$, it follows that $p-1\in I_{\bm\beta}^{(a+1)}$. But, according to $(\textbf{P}4)$, $p-1$ does not belong to $I_{\bm\beta}^{(a+1)}$. For this reason $t\neq p-1$. As $S_{\bm\alpha_{a,r},\bm\beta_{a,r},p}\subset\{0,1,\ldots,p-1\}$ then $t<p-1$. 
Hence, $t+1\in E_{\mathfrak{D}_p^a(\bm\alpha),\mathfrak{D}_p^a(\bm\beta),p}$ because $t+1\leq p-1$ and $t+1\bmod p\equiv1-\mathfrak{D}_p^a(\beta_{s})\bmod p$. 

We now proceed to see that $t+1\in S_{\mathfrak{D}_p^a(\bm\alpha),\mathfrak{D}_p^a(\bm\beta),p}$. It is clear that we have the following equality
\begin{equation}\label{eq_6}
\mathcal{Q}_{\mathfrak{D}_p^a(\bm\alpha),\mathfrak{D}_p^a(\bm\beta)}(t+1)=\mathcal{Q}_{\bm\alpha_{a,r},\bm\beta_{a,r}}(t)\cdot \frac{\underset{w\in\mathcal{C}_{\mathfrak{D}_p^{a-1}(\bm\alpha),r}}\prod(\mathfrak{D}_p^{a}(\alpha_w)+t)}{\underset{w\in\mathcal{C}_{\mathfrak{D}_p^{a-1}(\bm\beta),r}}\prod(\mathfrak{D}_p^{a}(\beta_w)+t)}\cdot\frac{\underset{w\in\mathcal{P}_{\mathfrak{D}_p^{a-1}(\bm\alpha),r}}\prod\mathfrak{D}_p^{a}(\alpha_w)}{\underset{w\in\mathcal{P}_{\mathfrak{D}_p^{a-1}(\bm\beta),r}}\prod\mathfrak{D}_p^{a}(\beta_w)}.
\end{equation}

It follows from $(\textbf{P}1)$ that $\mathfrak{D}_p^a(\alpha_{1}),\ldots$, $\mathfrak{D}_p^a(\alpha_{n}),\mathfrak{D}_p^a(\beta_{1}),\ldots,\mathfrak{D}_p^a(\beta_{n})$ belong to $\mathbb{Z}_{(p)}^*$. Then,
$$\left(\underset{w\in\mathcal{P}_{\mathfrak{D}_p^{a-1}(\bm\alpha),r}}\prod\mathfrak{D}_p^{a}(\alpha_w)\right)\bigg/\left(\underset{w\in\mathcal{P}_{\mathfrak{D}_p^{a-1}(\bm\beta),r}}\prod\mathfrak{D}_p^{a}(\beta_w)\right) \in\mathbb{Z}_{(p)}^*.$$ 
Now, assume for contradiction that there is $\gamma\in\{\alpha_{1},\ldots,\alpha_{n}\}$ such that $\mathfrak{D}_p^a(\gamma)+t$ belongs to $p\mathbb{Z}_{(p)}$. Since $t\bmod p=-\mathfrak{D}_p^a(\beta_{s})\bmod p$, it follows that $\mathfrak{D}_p^a(\gamma)-\mathfrak{D}_p^a(\beta_{s})\in p\mathbb{Z}_{(p)}$. That is a contradiction because, according to $(\textbf{P}2)$, $\mathfrak{D}_p^a(\gamma)-\mathfrak{D}_p^a(\beta_{s})\notin p\mathbb{Z}_{(p)}$. For this reason, the elements $\mathfrak{D}_p^a(\alpha_{1})+t,\ldots,\mathfrak{D}_p^a(\alpha_{n})+t$ belong to $\mathbb{Z}_{(p)}^*$.

Again, suppose, to derive a contradiction, that there is $w\in\mathcal{C}_{\mathfrak{D}_p^{a-1}(\bm\beta),r}$ such that $\mathfrak{D}_p^a(\beta_w)+t$ belongs to $p\mathbb{Z}_{(p)}$. Since $t\equiv-\mathfrak{D}_p^a(\beta_{s})\bmod p$, we obtain $\mathfrak{D}_p^a(\beta_w)-\mathfrak{D}_p^a(\beta_{s})\in p\mathbb{Z}_{(p)}$. Then, according to $(\textbf{P}3)$, $\beta_w=\beta_s$. On the one hand, we have $(\mathfrak{D}_p^{a-1}(\beta_s))_{r}\in p\mathbb{Z}_{(p)}$ because $s\in\mathcal{P}_{\mathfrak{D}_p^{a-1}(\bm\beta),r}$. On the other hand, $(\mathfrak{D}_p^{a-1}(\beta_s))_{r}\notin p\mathbb{Z}_{(p)}$  because $\beta_w=\beta_s$ and $w\in\mathcal{C}_{\mathfrak{D}_p^{a-1}(\bm\beta),r}$. So that, we have a contradiction. For this reason, for every $w\in\mathcal{C}_{\mathfrak{D}_p^{a-1}(\bm\beta),r}$, $\mathfrak{D}_p^a(\beta_{w})+t$ belongs to $\mathbb{Z}_{(p)}^*$.  Consequently, the element
$$\left(\underset{w\in\mathcal{C}_{\mathfrak{D}_p^{a-1}(\bm\alpha),r}}\prod(\mathfrak{D}_p^{a}(\alpha_w)+t)\right)\bigg/\left(\underset{w\in\mathcal{C}_{\mathfrak{D}_p^{a-1}(\bm\beta),r}}\prod(\mathfrak{D}_p^{a}(\beta_w)+t)\right) \in\mathbb{Z}_{(p)}^*.$$ 
Then, it follows from Equation~\eqref{eq_6} that $v_p(\mathcal{Q}_{\mathfrak{D}_p^a(\bm\alpha),\mathfrak{D}_p^a(\bm\beta)}(t+1))=0$ because $v_p(\mathcal{Q}_{\bm\alpha_{a,r},\bm\beta_{a,r}}(t))=0$. So that, we have $t+1\in S_{\mathfrak{D}_p^a(\bm\alpha),\mathfrak{D}_p^a(\bm\beta),p}.$

Therefore, we have  $\sigma$ is well-defined. In order to prove that $\sigma$ is a bijective map we are going to show that its inverse is $\tau:S_{\mathfrak{D}_p^a(\bm\alpha),\mathfrak{D}_p^a(\bm\beta),p}\rightarrow S_{\bm\alpha_{a,r},\bm\beta_{a,r},p}$. 

Let $t$ be in $S_{\mathfrak{D}_p^a(\bm\alpha),\mathfrak{D}_p^a(\bm\beta),p}$. Then $v_p(\mathcal{Q}_{\mathfrak{D}_p^a(\bm\alpha),\mathfrak{D}_p^a(\bm\beta)}(t))=0$ and $t\bmod p\equiv 1-\mathfrak{D}_p^{a}(\beta_s)\bmod p$ for some $s\in\{1,\ldots,n\}$. We are going to see that $\tau$ is well-defined. For this purpose, we first show that if there exists $s'\in\{1,\ldots,n\}$ such that $t\bmod p\equiv1-\mathfrak{D}_p^{a}(\beta_{s'})\bmod p$ then, $s\in\mathcal{C}_{\mathfrak{D}_p^{a-1}(\bm\beta),r}$ if and only if $s'\in\mathcal{C}_{\mathfrak{D}_p^{a-1}(\bm\beta),r}$ and second, we prove that $\tau(t)\in S_{\bm\alpha_{a,r},\bm\beta_{a,r},p}$. Suppose that $s\in\mathcal{C}_{\mathfrak{D}_p^{a-1}(\bm\beta),r}$. Note that $\mathfrak{D}_p^{a}(\beta_s)-\mathfrak{D}_p^{a}(\beta_{s'})\in p\mathbb{Z}_{(p)}$ because $1-\mathfrak{D}_p^{a}(\beta_s)\bmod p=t\bmod p=1-\mathfrak{D}_p^{a}(\beta_s')$. So according to $(\textbf{P}3)$, $\beta_s=\beta_{s'}$. So, $(\mathfrak{D}_p^{a-1}(\beta_{s'}))_{r}\notin p\mathbb{Z}_{(p)}$ because $s\in\mathcal{C}_{\mathfrak{D}_p^{a-1}(\bm\beta),r}$. Hence, $s'\in\mathcal{C}_{\mathfrak{D}_p^{a-1}(\bm\beta),r}$. In a similar way, one shows that if $s'\in\mathcal{C}_{\mathfrak{D}_p^{a-1}(\bm\beta),r} $ then $s\in\mathcal{C}_{\mathfrak{D}_p^{a-1}(\bm\beta),r}$. We now proceed to show that  $\tau(t)\in S_{\bm\alpha_{a,r},\bm\beta_{a,r},p}$.

\textbullet\quad Suppose that $t\bmod p\equiv1-\mathfrak{D}_p^a(\beta_{s})\bmod p$ with $s\in\mathcal{C}_{\mathfrak{D}_p^{a-1}(\bm\beta),r}$. Then $\beta_{s,a,r}=\mathfrak{D}^a_p(\beta_s)$ and therefore, $t\in E_{\bm\alpha_{a,r},\bm\beta_{a,r},p}$. We want to see that $t\in S_{\bm\alpha_{a,r},\bm\beta_{a,r},p}$. We have the following equality
\begin{equation}\label{eq_7}
\mathcal{Q}_{\mathfrak{D}_p^a(\bm\alpha),\mathfrak{D}_p^a(\bm\beta)}(t)=\mathcal{Q}_{\bm\alpha_{a,r},\bm\beta_{a,r}}(t)\cdot\frac{\underset{w\in\mathcal{P}_{\mathfrak{D}_p^{a-1}(\bm\beta),r}}\prod(\mathfrak{D}_p^{a}(\beta_w)+t)}{\underset{w\in\mathcal{P}_{\mathfrak{D}_p^{a-1}(\bm\alpha),r}}\prod(\mathfrak{D}_p^{a}(\alpha_w)+t)}\cdot\frac{\underset{w\in\mathcal{P}_{\mathfrak{D}_p^{a-1}(\bm\alpha),r}}\prod\mathfrak{D}_p^{a}(\alpha_w)}{\underset{w\in\mathcal{P}_{\mathfrak{D}_p^{a-1}(\bm\beta),r}}\prod\mathfrak{D}_p^{a}(\beta_w)}.
\end{equation}
By $(\textbf{P}1)$, we know that $\mathfrak{D}_p^a(\alpha_{1}),\ldots,\mathfrak{D}_p^a(\alpha_{n}),$ $\mathfrak{D}_p^a(\beta_{1}),\ldots,\mathfrak{D}_p^a(\beta_{n})$ belong to $\mathbb{Z}_{(p)}^*$. Therefore,
$$\left(\underset{w\in\mathcal{P}_{\mathfrak{D}_p^{a-1}(\bm\alpha),r}}\prod\mathfrak{D}_p^{a}(\alpha_w)\right)\bigg/\left(\underset{w\in\mathcal{P}_{\mathfrak{D}_p^{a-1}(\bm\beta),r}}\prod\mathfrak{D}_p^{a}(\beta_w)\right)\in\mathbb{Z}_{(p)}^*.$$
Now, assume for contradiction that there is $\gamma\in\{\alpha_{1},\ldots,\alpha_{n},\beta_{1},\ldots,\beta_{n}\}$ such that $\mathfrak{D}_p^a(\gamma)+t$ belongs to $p\mathbb{Z}_{(p)}$. Since $0\leq t<p$, it follows that $t=p\mathfrak{D}_p^{a+1}(\gamma)-\mathfrak{D}_p^a(\gamma)$. As $t\bmod p\equiv1-\mathfrak{D}_p^a(\beta_{s})\bmod p$ then $1-\mathfrak{D}_p^a(\beta_{s})+\mathfrak{D}_p^a(\gamma)\in p\mathbb{Z}_{(p)}$. This a contradiction to $(\textbf{P}5)$.  Consequently, the elements $\mathfrak{D}_p^a(\beta_{1})+t,\ldots,\mathfrak{D}_p^a(\beta_{n})+t,$ $\mathfrak{D}_p^a(\alpha_{1})+t,\ldots,\mathfrak{D}_p^a(\alpha_{n})+t$ belong to $\mathbb{Z}_{(p)}^*$. Thus,
$$\left(\underset{w\in\mathcal{P}_{\mathfrak{D}_p^{a-1}(\bm\beta),r}}\prod(\mathfrak{D}_p^{a}(\beta_w)+t)\right)\bigg/\left(\underset{w\in\mathcal{P}_{\mathfrak{D}_p^{a-1}(\bm\alpha),r}}\prod(\mathfrak{D}_p^{a}(\alpha_w)+t)\right)\in\mathbb{Z}_{(p)}^*.$$
Then, from Equation~\eqref{eq_7}, we have $v_p(\mathcal{Q}_{\bm\alpha_{a,r},\bm\beta_{a,r}}(t))=0$ because $v_p(\mathcal{Q}_{\mathfrak{D}_p^a(\bm\alpha),\mathfrak{D}_p^a(\bm\beta)}(t))=0$. So that, $t\in S_{\bm\alpha_{a,r},\bm\beta_{a,r},p}$.

\textbullet\quad Suppose that $t\bmod p\equiv1-\mathfrak{D}_p^a(\beta_{s})\bmod p$ with $s\in\mathcal{P}_{\mathfrak{D}_p^{a-1}(\bm\beta),r}.$ Then $\beta_{s,a,r}=\mathfrak{D}^a_p(\beta_s)+1$. We want to see that $t-1\in S_{\bm\alpha_{a,r},\bm\beta_{a,r},p}$. To this end, we first prove that $t-1\in E_{\bm\alpha_{a,r},\bm\beta_{a,r},p}$.

Assume for contradiction that $t=0$. Then, $1-\mathfrak{D}_p^a(\beta_s)\in p\mathbb{Z}_{(p)}$. Since, $\beta_n=1$ and, for all integers $m\geq1$, $\mathfrak{D}_p^m(1)=1$, it follows that $\mathfrak{D}_p^a(\beta_n)-\mathfrak{D}_p^a(\beta_s)\in p\mathbb{Z}_{(p)}$. Then, from $(\textbf{P}3)$, we obtain $\beta_n=\beta_s$. Thus, $1=\beta_s$. As $s\in\mathcal{P}_{\mathfrak{D}_p^{a-1}(\bm\beta),r}$ then, by definition of the set $\mathcal{P}_{\mathfrak{D}_p^{a-1}(\bm\beta),r}$, we have $(\mathfrak{D}_p^{a-1}(\beta_s))_{r}\in p\mathbb{Z}_{(p)}$. Thus, $(1)_{r}\in p\mathbb{Z}_{(p)}$. But, $(1)_{r}\notin p\mathbb{Z}_{(p)}$ because $r\in\{0,\ldots, p-1\}$.  Consequently, $t>0$. Thus, $t-1\in\{0,\ldots, p-1\}$. As $t-1\bmod p\equiv-\mathfrak{D}_p^a(\beta_{s})\bmod p$ and $\beta_{s,a,r}=\mathfrak{D}^a_p(\beta_s)+1$ then $t-1\in E_{\bm\alpha_{a,r},\bm\beta_{a,r},p}$. 

Now, we have the following equality
\begin{equation}\label{eq_8}
\mathcal{Q}_{\mathfrak{D}_p^a(\bm\alpha),\mathfrak{D}_p^a(\bm\beta)}(t)=\mathcal{Q}_{\bm\alpha_{a,r},\bm\beta_{a,r}}(t-1)\cdot\frac{\underset{w\in\mathcal{C}_{\mathfrak{D}_p^{a-1}(\bm\alpha),r}}\prod(\mathfrak{D}_p^{a}(\alpha_w)+t-1)}{\underset{w\in\mathcal{C}_{\mathfrak{D}_p^{a-1}(\bm\beta),r}}\prod(\mathfrak{D}_p^{a}(\beta_w)+t-1)}\cdot\frac{\underset{w\in\mathcal{P}_{\mathfrak{D}_p^{a-1}(\bm\alpha),r}}\prod\mathfrak{D}_p^{a}(\alpha_w)}{\underset{w\in\mathcal{P}_{\mathfrak{D}_p^{a-1}(\bm\beta),r}}\prod\mathfrak{D}_p^{a}(\beta_w)}.
\end{equation}
By $(\textbf{P}1)$, we know that $\mathfrak{D}_p^a(\alpha_{1}),\ldots,\mathfrak{D}_p^a(\alpha_{n}),$ $\mathfrak{D}_p^a(\beta_{1}),\ldots,\mathfrak{D}_p^a(\beta_{n})$ belong to $\mathbb{Z}_{(p)}^*$. Then,
$$\left(\underset{w\in\mathcal{P}_{\mathfrak{D}_p^{a-1}(\bm\alpha),r}}\prod\mathfrak{D}_p^{a}(\alpha_w)\right)\bigg/\left(\underset{w\in\mathcal{P}_{\mathfrak{D}_p^{a-1}(\bm\beta),r}}\prod\mathfrak{D}_p^{a}(\beta_w)\right)\in\mathbb{Z}_{(p)}^*.$$

Assume for contradiction that there is $\gamma\in\{\alpha_{1},\ldots,\alpha_{n}\}$ such that $\mathfrak{D}_p^a(\gamma)+t-1$ belongs to $p\mathbb{Z}_{(p)}$. Since $t-1\bmod p\equiv-\mathfrak{D}_p^a(\beta_{s})\bmod p$, it follows that $\mathfrak{D}_p^a(\gamma)-\mathfrak{D}_p^a(\beta_{s})\in p\mathbb{Z}_{(p)}$. That is a contradiction because, according to $(\textbf{P}2)$, $\mathfrak{D}_p^a(\gamma)-\mathfrak{D}_p^a(\beta_{s})\notin p\mathbb{Z}_{(p)}$. For this reason, the elements $\mathfrak{D}_p^a(\alpha_{1})+t-1,\ldots,\mathfrak{D}_p^a(\alpha_{n})+t-1$ belong to $\mathbb{Z}_{(p)}^*$.

Again, suppose, to derive a contradiction, that there is $w\in\mathcal{C}_{\mathfrak{D}_p^{a-1}(\bm\beta),r}$ such that $\mathfrak{D}_p^a(\beta_w)+t-1$ belongs to $p\mathbb{Z}_{(p)}$. Since $t-1\equiv-\mathfrak{D}_p^a(\beta_{s})\bmod p$, we obtain $\mathfrak{D}_p^a(\beta_w)-\mathfrak{D}_p^a(\beta_{s})\in p\mathbb{Z}_{(p)}$. Then, according to $(\textbf{P}3)$, we have $\beta_w=\beta_s$. On the one hand, we have $(\mathfrak{D}_p^{a-1}(\beta_s))_{r}\in p\mathbb{Z}_{(p)}$ because $s\in\mathcal{P}_{\mathfrak{D}_p^{a-1}(\bm\beta),r}$. On the other hand, $(\mathfrak{D}_p^{a-1}(\beta_s))_{r}\notin p\mathbb{Z}_{(p)}$  because $\beta_w=\beta_s$ and $w\in\mathcal{C}_{\mathfrak{D}_p^{a-1}(\bm\beta),r}$. So that, we have a contradiction. For this reason, for every $w\in\mathcal{C}_{\mathfrak{D}_p^{a-1}(\bm\beta),r}$, $\mathfrak{D}_p^a(\beta_{w})+t-1$ belongs to $\mathbb{Z}_{(p)}^*$.  Consequently, the element 
$$\left(\underset{w\in\mathcal{C}_{\mathfrak{D}_p^{a-1}(\bm\alpha),r}}\prod(\mathfrak{D}_p^{a}(\alpha_w)+t-1)\right)\bigg/\left(\underset{w\in\mathcal{C}_{\mathfrak{D}_p^{a-1}(\bm\beta),r}}\prod(\mathfrak{D}_p^{a}(\beta_w)+t-1)\right)\in\mathbb{Z}_{(p)}^*.$$ 
Thus, from Equation~\eqref{eq_8}, we have $v_p(\mathcal{Q}_{\bm\alpha_{a,r},\bm\beta_{a,r}}(t-1))=0$ because $v_p(\mathcal{Q}_{\mathfrak{D}_p^a(\bm\alpha),\mathfrak{D}_p^a(\bm\beta)}(t))=0$. So that, $t-1\in S_{\bm\alpha_{a,r},\bm\beta_{a,r},p}$.

Consequently, $\tau$ is well-defined and it is clear that $\tau$ is the inverse of $\sigma$. Therefore, $\sigma$ is a bijective map.

B) 
Let $(a,r)$ be in $\{1,\ldots, l\}\times\{0,\ldots,p-1\}$, and let $t$ be in $ S_{\bm\alpha_{a,r},\bm\beta_{a,r},p}$. We are going to see that $\mathcal{P}_{\bm\alpha_{a,r},t}=\mathcal{P}_{\mathfrak{D}_p^a(\bm\alpha),\sigma(t)}$. As $t\in S_{\bm\alpha_{a,r},\bm\beta_{a,r},p}$ then  $t\bmod p\equiv1-\beta_{s,a,r}$ for some $s\in\{1,\ldots,n\}$. 
 In particular, $$t\bmod p\equiv\left\{ \begin{array}{lcc}
             1-\mathfrak{D}_p^a(\beta_s)\bmod p &   if  &  s\in\mathcal{C}_{\mathfrak{D}_p^{a-1}(\bm\beta),r} \\
             \\  -\mathfrak{D}_p^a(\beta_s)\bmod p&  if & s\in\mathcal{P}_{\mathfrak{D}_p^{a-1}(\bm\beta),r}.  \\
             \end{array}
   \right.$$

\textbullet\quad Suppose that   $s\in\mathcal{C}_{\mathfrak{D}_p^{a-1}(\bm\beta),r}$. Then, $t\bmod p\equiv1-\mathfrak{D}_p^a(\beta_{s})\bmod p$ and $\sigma(t)=t$. We are going to see that $\mathcal{P}_{\bm\alpha_{a,r},t}\subset\mathcal{P}_{\mathfrak{D}_p^a(\bm\alpha),t}$. Let $w$ be in $\mathcal{P}_{\bm\alpha_{a,r},t}$. Then, $(\alpha_{w,a,r})_{t}\in p\mathbb{Z}_{(p)}$. If $w\in\mathcal{C}_{\mathfrak{D}_p^{a-1}(\bm\alpha),r}$ then $\alpha_{w,a,r}=\mathfrak{D}_p^a(\alpha_w)$. So, $(\mathfrak{D}_p^a(\alpha_w))_t\in p\mathbb{Z}_{(p)}$. Hence, $w\in\mathcal{P}_{\mathfrak{D}_p^a(\bm\alpha),t}$. Now, suppose that $w\in\mathcal{P}_{\mathfrak{D}_p^{a-1}(\bm\alpha),r}$. Then, $\alpha_{w,a,r}=1+\mathfrak{D}_p^a(\alpha_w)$. Thus, $(\mathfrak{D}_p^a(\alpha_{w})+1)_{t}\in p\mathbb{Z}_{(p)}$.  Suppose, towards a contradiction, that $(\mathfrak{D}_p^a(\alpha_{w}))_{t}\notin p\mathbb{Z}_{(p)}$. Since $(\mathfrak{D}_p^a(\alpha_{w})+1)_{t}\in p\mathbb{Z}_{(p)}$, we get $\mathfrak{D}_p^a(\alpha_w)+t\in p\mathbb{Z}_{(p)}$. As $0\leq t\leq p-1$ then $t=p\mathfrak{D}_p^{a+1}(\alpha_w)-\mathfrak{D}_p^a(\alpha_w)$. As $t\bmod p\equiv1-\mathfrak{D}_p^a(\beta_{s})\bmod p$ then $1-\mathfrak{D}_p^a(\beta_{s})+\mathfrak{D}_p^a(\alpha_w)\in p\mathbb{Z}_{(p)}$. This leads to a contradiction of $(\textbf{P}5)$. Whence, $(\mathfrak{D}_p^a(\alpha_{w}))_{t}\in p\mathbb{Z}_{(p)}$. For this reason, $w\in\mathcal{P}_{\mathfrak{D}_p^a(\bm\alpha),t}$. Consequently, we have $\mathcal{P}_{\bm\alpha_{a,r},t}\subset\mathcal{P}_{\mathfrak{D}_p^a(\bm\alpha),t}$.
 
Now, we show that $\mathcal{P}_{\mathfrak{D}_p^a(\bm\alpha),t}\subset\mathcal{P}_{\bm\alpha_{a,r},t}$. Let $w$ be in $\mathcal{P}_{\mathfrak{D}_p^a(\bm\alpha),t}$. Then, $(\mathfrak{D}_p^a(\alpha_w))_{t}\in p\mathbb{Z}_{(p)}$. If $w\in\mathcal{C}_{\mathfrak{D}_p^{a-1}(\bm\alpha),r}$ then $\alpha_{w,a,r}=\mathfrak{D}_p^a(\alpha_w)$. Thus, $w\in\mathcal{P}_{\bm\alpha_{a,r},t}$. Now, if $w\in\mathcal{P}_{\mathfrak{D}_p^{a-1}(\bm\alpha),r}$ then $\alpha_{w,a,r}=1+\mathfrak{D}_p^a(\alpha_w)$. We have $(\mathfrak{D}_p^a(\alpha_w))_{t}\in p\mathbb{Z}_{(p)}$ and from $(\textbf{P}1)$ we know that $\mathfrak{D}^a_p(\alpha_w)\in\mathbb{Z}_{(p)}^{*}$. Thus, $(\mathfrak{D}_p^a(\alpha_w)+1)_{t}\in p\mathbb{Z}_{(p)}$. Then, $w\in\mathcal{P}_{\bm\alpha_{a,r},t}$. Consequently, we have $\mathcal{P}_{\mathfrak{D}_p^a(\bm\alpha),t}\subset\mathcal{P}_{\bm\alpha_{a,r},t}$.  

Therefore, $\mathcal{P}_{\bm\alpha_{a,r},t}=\mathcal{P}_{\mathfrak{D}_p^a(\bm\alpha),t}$. But, remember that $\sigma(t)=t$. Whence, we obtain $\mathcal{P}_{\bm\alpha_{a,r},t}=\mathcal{P}_{\mathfrak{D}_p^a(\bm\alpha),\sigma(t)}$.

\textbullet\quad Suppose that  $s\in\mathcal{P}_{\mathfrak{D}_p^{a-1}(\bm\beta),r}$. Then $t\bmod p\equiv-\mathfrak{D}_p^a(\beta_{s})\bmod p$ and $\sigma(t)=t+1$. We are going to see that $\mathcal{P}_{\bm\alpha_{a,r},t}\subset\mathcal{P}_{\mathfrak{D}_p^a(\bm\alpha),t+1}$. Let $w$ be in $\mathcal{P}_{\bm\alpha_{a,r},t}$. Then $(\alpha_{w,a,r})_t\in p\mathbb{Z}_{(p)}$. Suppose that $w\in\mathcal{C}_{\mathfrak{D}_p^{a-1}(\bm\alpha),r}$. Then, $\alpha_{w,a,r}=\mathfrak{D}_p^a(\alpha_w)$. So that, $(\mathfrak{D}_p^a(\alpha_{w}))_{t}\in p\mathbb{Z}_{(p)}$. For this reason, $(\mathfrak{D}_p^a(\alpha_w))_{t+1}\in p\mathbb{Z}_{(p)}$. Therefore, $w\in\mathcal{P}_{\mathfrak{D}_p^a(\bm\alpha),t+1}$. Suppose now that $w\in\mathcal{P}_{\mathfrak{D}_p^{a-1}(\bm\alpha),r}$. Then, $\alpha_{w,a,r}=1+\mathfrak{D}_p^a(\alpha_w)$. So that, $(\mathfrak{D}_p^a(\alpha_{w})+1)_{t}\in p\mathbb{Z}_{(p)}$. Thus, $(\mathfrak{D}_p^a(\alpha_w))_{t+1}\in p\mathbb{Z}_{(p)}$. Hence, $w\in\mathcal{P}_{\mathfrak{D}_p^a(\bm\alpha),t+1}$. Consequently, $\mathcal{P}_{\bm\alpha_{a,r},t}\subset\mathcal{P}_{\mathfrak{D}_p^a(\bm\alpha),t+1}$.
  
Now, we show that  $\mathcal{P}_{\mathfrak{D}_p^a(\bm\alpha),t+1}\subset\mathcal{P}_{\bm\alpha_{a,r},t}$. Let $w$ be in $\mathcal{P}_{\mathfrak{D}_p^a(\bm\alpha),t+1}$. Then, $(\mathfrak{D}_p^a(\alpha_w))_{t+1}\in p\mathbb{Z}_{(p)}$. Suppose that $w\in\mathcal{C}_{\mathfrak{D}_p^{a-1}(\bm\alpha),r}$. Then, $\alpha_{w,a,r}=\mathfrak{D}_p^a(\alpha_w)$. Assume for contradiction that $(\mathfrak{D}_p^a(\alpha_w))_{t}\notin p\mathbb{Z}_{(p)}$. Thus, $\mathfrak{D}_p^a(\alpha_w)+t\in p\mathbb{Z}_{(p)}$. As $0\leq t\leq p-1$ then $t=p\mathfrak{D}_p^{a+1}(\alpha_w)-\mathfrak{D}_p^a(\alpha_w)$. Since $t\bmod p\equiv-\mathfrak{D}_p^a(\beta_{s})\bmod p$, it follows that $\mathfrak{D}_p^a(\alpha_w)-\mathfrak{D}_p^a(\beta_{s})\in p\mathbb{Z}_{(p)}$. This contradicts our condition $(\textbf{P}2)$. Then, we have $(\mathfrak{D}_p^a(\alpha_{w}))_{t}\in p\mathbb{Z}_{(p)}$. For this reason, $w\in\mathcal{P}_{\bm\alpha_{a,r},t}$. Now, suppose that $w\in\mathcal{P}_{\mathfrak{D}_p^{a-1}(\bm\alpha),r}$. Then, $\alpha_{w,a,r}=1+\mathfrak{D}_p^a(\alpha_w)$. By $(\textbf{P}1)$, we know that $\mathfrak{D}_p^a(\alpha_w)\in\mathbb{Z}_{(p)}^*$ and since $(\mathfrak{D}_p^a(\alpha_w))_{t+1}\in p\mathbb{Z}_{(p)}$, it follows that $(\mathfrak{D}_p^a(\alpha_w)+1)_{t}\in p\mathbb{Z}_{(p)}$. So that $w\in\mathcal{P}_{\bm\alpha_{a,r},t}$. Consequently, $\mathcal{P}_{\mathfrak{D}_p^a(\bm\alpha),t+1}\subset\mathcal{P}_{\bm\alpha_{a,r},t}$. 

Therefore,  $\mathcal{P}_{\bm\alpha_{a,r},t}=\mathcal{P}_{\mathfrak{D}_p^a(\bm\alpha), t+1}$. But, remember that $\sigma(t)=t+1$. Whence, we obtain $\mathcal{P}_{\bm\alpha_{a,r},t}=\mathcal{P}_{\mathfrak{D}_p^a(\bm\alpha),\sigma(t)}$.

Thus,  for every $t\in S_{\bm\alpha_{a,r},\bm\beta_{a,r},p}$, we have $\mathcal{P}_{\bm\alpha_{a,r},t}=\mathcal{P}_{\mathfrak{D}_p^a(\bm\alpha),\sigma(t)}$. One shows, in an exactly similar way  that, for every $t\in S_{\bm\alpha_{a,r},\bm\beta_{a,r},p}$, we have $\mathcal{P}_{\bm\beta_{a,r},t}=\mathcal{P}_{\mathfrak{D}_p^a(\bm\beta),\sigma(t)}$.

Finally, for every $t\in S_{\bm\alpha_{a,r},\bm\beta_{a,r},p}$, we also have $\mathcal{C}_{\bm\alpha_{a,r},t}=\mathcal{C}_{\mathfrak{D}_p^a(\bm\alpha),\sigma(t)}$ because $\mathcal{C}_{\bm\alpha_{a,r},t}$ is the complement of $\mathcal{P}_{\bm\alpha_{a,r},t}$ in $\{1,\ldots, n\}$, $\mathcal{C}_{\mathfrak{D}_p^a(\bm\alpha),\sigma(t)}$ is the complement of $\mathcal{P}_{\mathfrak{D}_p^a(\bm\alpha),\sigma(t)}$ in $\{1,\ldots,n\}$, and we have already seen that $\mathcal{P}_{\bm\alpha_{a,r},t}=\mathcal{P}_{\mathfrak{D}_p^a(\bm\alpha),\sigma(t)}$. In a similar way one has $\mathcal{C}_{\bm\beta_{a,r},t}=\mathcal{C}_{\mathfrak{D}_p^a(\bm\beta),\sigma(t)}$.

$\hfill\square$
 \section{Proof of Lemma~\ref{lemm_second_step}}\label{sec_proof_lemm_second_step}

Lemma~\ref{lemm_second_step} is obtained from the following lemma

\begin{lemm}\label{lemm_first_step}
Let $\bm\alpha=(\alpha_1,\ldots,\alpha_n)$, $\bm\beta=(\beta_1,\ldots,\beta_{n-1},1)$ be in  $(\mathbb{Q}\setminus\mathbb{Z}_{\leq0})^n$ and let $p$ be a prime number  such that $p>d_{\bm\alpha,\bm\beta}$  and $f(z):={}_nF_{n-1}(\bm\alpha,\bm\beta;z)$ belongs to $\mathbb{Z}_{(p)}[[z]]$. Suppose that $(\bm\alpha, \bm\beta)$ satisfies the $\textbf{P}_{p,l}$ property, where $l$ is the order of $p$ in $(\mathbb{Z}/d_{\bm\alpha,\bm\beta}\mathbb{Z})^*$. Then, for each $r\in S_{\bm\alpha,\bm\beta,p}$, $f_{1,r}\in1+z\mathbb{Z}_{(p)}[[z]]$ and $$f(z) \equiv\sum_{r\in S_{\bm\alpha,\bm\beta,p}} Q_{r}f_{1,r}^{p}\bmod p\quad\text{ with }\quad Q_r(z)=\sum_{s=r}^{r'-1}\mathcal{Q}_{\bm\alpha,\bm\beta}(s)z^s,$$ where $r'$ is defined as follows. If $r\neq\max  E_{\bm\alpha,\bm\beta,p}$ then $r'$ is the element in $ E_{\bm\alpha,\bm\beta,p}$ such that $r<r'$ and $(r,r')\cap  E_{\bm\alpha,\bm\beta,p}=\emptyset$ or otherwise, $r'=p$. 

\end{lemm}

\begin{proof}[Proof of Lemma~\ref{lemm_second_step}]
We proceed by induction on $a\in\{1,\ldots,l\}$. Suppose $a=1$. By assumption,  $(\bm\alpha,\bm\beta)$ satisfies the $\textbf{P}_{p,l}$ property and $f(z)\in\mathbb{Z}_{(p)}[[z]]$. Thus, the hypotheses of Lemma~\ref{lemm_first_step} are satisfied and we conclude that
$$f(z) \equiv\sum_{r\in S_{\bm\alpha,\bm\beta,p}} Q_{r}f_{r}^{p}\bmod p,$$
where, each $Q_{r}(z)$ belongs to $\mathbb{Z}_{(p)}$ and has degree less than $p$, and $f_{r}(z)\in1+z\mathbb{Z}_{(p)}[[z]]$. %$$f_{r}(z)=\sum_{m\geq0}\left(\frac{\underset{s\in\mathcal{C}_{\bm\alpha,r}}{\prod}\mathfrak{D}_p(\alpha_s)_m {\underset{s\in\mathcal{P}_{\bm\alpha,r}}{\prod}(\mathfrak{D}_p(\alpha_s)+1)_m}}{\underset{s\in\mathcal{C}_{\beta,r}}{\prod}\mathfrak{D}_p(\beta_s)_m \underset{s\in\mathcal{P}_{\bm\beta,r}}{\prod}(\mathfrak{D}_p(\beta_s)+1)_m}\right)z^m\in1+z\mathbb{Z}_{(p)}[[z]].$$
We now suppose that the conclusion of our lemma is true for some $a$ in $\{1,\ldots, l-1\}$. We are going to prove that it is also true for $a+1$. By induction hypothesis, we have 
\begin{equation}\label{eq_induc_hy}
f\equiv\sum_{r\in S_{\mathfrak{D}_p^{a-1}(\bm\alpha),\mathfrak{D}_p^{a-1}(\bm\beta),p}} Q_{a,r}(z)f_{a,r}^{p^a}\bmod p
\end{equation}
where, for every $r\in S_{\mathfrak{D}_p^{a-1}(\bm\alpha),\mathfrak{D}_p^{a-1}(\bm\beta),p}$, $Q_{a,r}(z)$ belongs to $\mathbb{Z}_{(p)}[z]$ and has degree less than $p^a$ and $f_{a,r}\in1+z\mathbb{Z}_{(p)}[[z]]$. %$$f_{a,r}=\sum_{m\geq0}\left(\frac{\underset{s\in\mathcal{C}_{\mathfrak{D}_p^{a-1}(\bm\alpha),r}}\prod(\mathfrak{D}_p^{a}(\alpha_s))_m\underset{s\in\mathcal{P}_{\mathfrak{D}_p^{a-1}(\bm\alpha),r}}{\prod}(\mathfrak{D}_p^a(\alpha_s)+1)_m}{\underset{s\in\mathcal{C}_{\mathfrak{D}_p^{a-1}(\bm\beta),r}}\prod(\mathfrak{D}_p^{a}(\beta_s))_m\underset{s\in\mathcal{P}_{\mathfrak{D}_p^{a-1}(\bm\beta),r}}{\prod}(\mathfrak{D}_p^a(\beta_s)+1)_m}\right)z^m\in1+z\mathbb{Z}_{(p)}[[z]].$$
  
We fix $r$ in $S_{\mathfrak{D}_p^{a-1}(\bm\alpha),\mathfrak{D}_p^{a-1}(\bm\beta),p}$. By definition, $f_{a,r}(z)={}_nF_{n-1}(\bm\alpha_{a,r},\bm\beta_{a,r}; z)$. We would like to apply Lemma~\ref{lemm_first_step} to $f_{a,r}(z)$. To this end, we are going to see that the hypotheses of Lemma \ref{lemm_first_step} are satisfied. By induction hypothesis, we know that $f_{a,r}(z)\in\mathbb{Z}_{(p)}[[z]]$ and thanks to $(\textbf{P}1)$, $\bm\alpha_{a,r},\bm\beta_{a,r}$ belong to $(\mathbb{Z}_{(p)}\setminus\mathbb{Z}_{\leq 0})^n$. According to Remark~\ref{rema_d_p_h},  $(\bm\alpha_{a,r},\bm\beta_{a,r})$ satisfies the $\textbf{P}_{p,l'}$ property, where $l'$ is the order of $p$ in $(\mathbb{Z}/d_{\bm\alpha_{a,r},\bm\beta_{a,r}}\mathbb{Z})^*$. We can then apply Lemma~\ref{lemm_first_step} to $f_{a,r}$ and we obtain
 \begin{equation}\label{eq_f_i_a}
 f_{a,r}(z)\equiv\sum_{\mu\in S_{\bm\alpha_{a,r},\bm\beta_{a,r},p}} P_{\mu}(z)g_{\mu}^p\bmod p,
 \end{equation}
where, each $P_{\mu}(z)\in\mathbb{Z}_{(p)}[z]$ has degree less than  $p$, and $$g_{\mu}(z)=\sum_{m\geq0}\left(\frac{\underset{s\in\mathcal{C}_{\bm\alpha_{a,r},\mu}}{\prod}\mathfrak{D}_p(\alpha_{s,a,r})_m {\underset{s\in\mathcal{P}_{\bm\alpha_{a,r},\mu}}{\prod}(\mathfrak{D}_p(\alpha_{s,a,r})+1)_m}}{\underset{s\in\mathcal{C}_{\bm\beta_{a,r},\mu}}{\prod}\mathfrak{D}_p(\beta_{s,a,r})_m \underset{s\in\mathcal{P}_{\bm\beta_{a,r},\mu}}{\prod}(\mathfrak{D}_p(\beta_{s,a,r})+1)_m}\right)z^m\in1+z\mathbb{Z}_{(p)}[[z]].$$
By (2) of Remark~\ref{rema_d_p_1}, we know that $\mathfrak{D}_p(\bm\alpha_{a,r})=\mathfrak{D}_p^{a+1}(\bm\alpha)$ and that $\mathfrak{D}_p(\bm\beta_{a,r})=\mathfrak{D}_p^{a+1}(\bm\beta)$. Hence, it follows that, for all $\mu\in S_{\bm\alpha_{a,r},\bm\beta_{a,r},p}$, $$g_{\mu}(z)=\sum_{m\geq0}\left(\frac{\underset{s\in\mathcal{C}_{\bm\alpha_{a,r},\mu}}{\prod}\mathfrak{D}_p^{a+1}(\alpha_{s})_m {\underset{s\in\mathcal{P}_{\bm\alpha_{a,r},\mu}}{\prod}(\mathfrak{D}_p^{a+1}(\alpha_{s})+1)_m}}{\underset{s\in\mathcal{C}_{\bm\beta_{a,r},\mu}}{\prod}\mathfrak{D}_p^{a+1}(\beta_{s})_m \underset{s\in\mathcal{P}_{\bm\beta_{a,r},\mu}}{\prod}(\mathfrak{D}_p^{a+1}(\beta_{s})+1)_m}\right)z^m.$$ 
We now want to see that $g_{\mu}(z)=f_{a+1,\sigma(\mu)}$ for all $\mu\in S_{\bm\alpha_{a,r},\bm\beta_{a,r},p}$, where  $\sigma:S_{\bm\alpha_{a,r},\bm\beta_{a,r},p}\rightarrow S_{\mathfrak{D}_p^{a}(\bm\alpha),\mathfrak{D}_p^{a}(\bm\beta),p}$ is the bijective map given by Lemma~\ref{lemm_aux}. By definition, we have $f_{a+1,\gamma}={}_nF_{n-1}(\bm\alpha_{a+1,\gamma},\bm\beta_{a+1,\gamma};z)$ for all $\gamma\in S_{\mathfrak{D}_p^{a}(\bm\alpha),\mathfrak{D}_p^{a}(\bm\beta),p}$. Thus, $$f_{a+1,\gamma}=\sum_{m\geq0}\left(\frac{\underset{s\in\mathcal{C}_{\mathfrak{D}_p^a(\bm\alpha),\gamma}}{\prod}\mathfrak{D}_p^{a+1}(\alpha_s)_m {\underset{s\in\mathcal{P}_{\mathfrak{D}_p^a(\bm\alpha),\gamma}}{\prod}(\mathfrak{D}_p^{a+1}(\alpha_s)+1)_m}}{\underset{s\in\mathcal{C}_{\mathfrak{D}_p^a(\bm\beta),\gamma}}{\prod}\mathfrak{D}_p^{a+1}(\beta_s)_m \underset{s\in\mathcal{P}_{\mathfrak{D}_p^a(\bm\beta),\gamma}}{\prod}(\mathfrak{D}_p^{a+1}(\beta_s)+1)_m}\right)z^m.$$

By invoking B) of Lemma~\ref{lemm_aux}, we obtain, for every $\mu\in S_{\bm\alpha_{a,r},\bm\beta_{a,r},p}$, the following equalities $\mathcal{P}_{\bm\alpha_{a,r},\mu}=\mathcal{P}_{\mathfrak{D}_p^a(\bm\alpha),\sigma(\mu)}$,  $\mathcal{C}_{\bm\alpha_{a,r},\mu}=\mathcal{C}_{\mathfrak{D}_p^a(\bm\alpha),\sigma(\mu)}$, $\mathcal{P}_{\bm\beta_{a,r},\mu}=\mathcal{P}_{\mathfrak{D}_p^a(\bm\beta),\sigma(\mu)}$, and  $\mathcal{C}_{\bm\beta_{a,r},\mu}=\mathcal{C}_{\mathfrak{D}_p^a(\bm\beta),\sigma(\mu)}$. For this reason, $g_{\mu}(z)=f_{a+1,\sigma(\mu)}$. By A) of Lemma~\ref{lemm_aux}, we know that $\sigma:S_{\bm\alpha_{a,r},\bm\beta_{a,r},p}\rightarrow S_{\mathfrak{D}_p^{a}(\bm\alpha),\mathfrak{D}_p^{a}(\bm\beta),p}$ is bijective. Then, from \eqref{eq_f_i_a}, we infer that
 \begin{equation}\label{eq_new_f_i_a}
 f_{a,r}(z)\equiv\sum_{\gamma\in S_{\mathfrak{D}_p^{a}(\bm\alpha),\mathfrak{D}_p^{a}(\bm\beta),p}} T_{r,\gamma}(z)f_{a+1,\gamma}^p\bmod p
 \end{equation}
 where  $T_{r,\gamma}(z)\in\mathbb{Z}_{(p)}[z]$ with degree less than  $p$. As $r$ in $S_{\mathfrak{D}_p^{a-1}(\bm\alpha),\mathfrak{D}_p^{a-1}(\bm\beta),p}$ is an arbitrary element, then, from Equalities \eqref{eq_induc_hy} and \eqref{eq_new_f_i_a}, we obtain $$f(z)\equiv\sum_{\gamma\in S_{\mathfrak{D}_p^{a}(\bm\alpha),\mathfrak{D}_p^{a}(\bm\beta),p}}Q_{a+1,\gamma}f_{a+1,\gamma}^{p^{a+1}}\bmod p,\text{ with } Q_{a+1,\gamma}=\sum_{r\in S_{\mathfrak{D}_p^{a-1}(\bm\alpha),\mathfrak{D}_p^{a-1}(\bm\beta),p}} Q_{a,r}(z)T_{r,\gamma}^{p^a}.$$
 For every $\gamma\in S_{\mathfrak{D}_p^{a}(\bm\alpha),\mathfrak{D}_p^{a}(\bm\beta),p}$, the polynomial $Q_{a+1,\gamma}$ has degree less than $p^{a+1}$ because, for every $r\in S_{\mathfrak{D}_p^{a-1}(\bm\alpha),\mathfrak{D}_p^{a-1}(\bm\beta),p}$, $Q_{a,r}\in\mathbb{Z}_{(p)}[z]$ has degree less than $p^a$ and $T_{r,\gamma}\in\mathbb{Z}_{(p)}[z]$ has degree less than $p$.  This completes the proof.
\end{proof}

The remainder of this section is devoted to proving Lemma~\ref{lemm_first_step}.

\subsection{Cartier operators}

The proof of the Lemma~\ref{lemm_first_step} depends essentially on Lemma~\ref{lemm_cartier_operator}. In preparation for stating Lemma  \ref{lemm_cartier_operator}, we recall the definition of Cartier operators over the ring $\mathbb{F}_p[[z]]$. For each $r\in\{0,\ldots, p-1\}$, we have the $\mathbb{F}_p$-linear operator $\Lambda_r:\mathbb{F}_p[[z]]\rightarrow\mathbb{F}_p[[z]]$ given by $$\Lambda_r\left(\sum_{j\geq0}a(n)z^j\right)=\sum_{j\geq0}a(jp+r)z^j.$$
The operators $\Lambda_0,\ldots,\Lambda_{p-1}$ are called the \emph{Cartiers Operators}\footnote{We refer the reader to \cite[Section 2]{AY}, where the authors explain why these operators are referred to as the Cartiers Operators.}.

\begin{lemm}\label{lemm_cartier_operator}
 %Let $\bm\alpha=(\alpha_1,\ldots,\alpha_n)$, $\bm\beta=(\beta_1,\ldots,\beta_{n-1},1)$ be in $(\mathbb{Q}\setminus\mathbb{Z}_{\leq0})^n$ and let $p$ be a primer number such that $\bm\alpha$, $\bm\beta$ belong to $(\mathbb{Z}_{(p)})^n$ and $f(z):={}_nF_{n-1}(\bm\alpha,\bm\beta;z)$ belongs to $\mathbb{Z}_{(p)}[[z]]$ and let $l$ be the order of $p$ in $(\mathbb{Z}/d_{\bm\alpha,\bm\beta}\mathbb{Z})^*$. If $(\bm\alpha, \bm\beta)$ satisfies the $\textbf{P}_{p,l}$ property 
 Let the assumptions be as in Lemma~\ref{lemm_first_step}. Then $$f(z)\equiv\sum_{r\in S_{\bm\alpha,\bm\beta,p}} P_r(z)\Lambda_r(f)^p\bmod p\quad\text{ with }\quad P_r(z)=\sum_{s=r}^{r'-1}\frac{\mathcal{Q}_{\bm\alpha,\bm\beta}(s)}{\mathcal{Q}_{\bm\alpha,\bm\beta}(r)}z^s,$$
where $r'$ is defined as follows. If $r\neq\max  E_{\bm\alpha,\bm\beta,p}$ then $r'$ is the element in $ E_{\bm\alpha,\bm\beta,p}$ such that $r<r'$ and $(r,r')\cap  E_{\bm\alpha,\bm\beta,p}=\emptyset$ or  otherwise, $r'=p$. 

\end{lemm}

\subsection{Auxiliary result I}\label{sec_aux_1}
In order to prove Lemma~\ref{lemm_cartier_operator}, we need the next auxiliary results which we state and prove. These auxiliary results deal with $p$-adic properties of the sequence $\{\mathcal{Q}_{\bm\alpha,\bm\beta}(i)\}_{i\geq0}$. The main result is Lemma~\ref{lemm_padic_valuation}.

\begin{lemm}\label{lemm_lambda_0}
Let $p$ be a prime number and let $\bm\alpha=(\alpha_1,\ldots,\alpha_n)$, $\bm\beta=(\beta_1,\ldots,\beta_{n-1},1)$ be in $(\mathbb{Z}_{(p)}\setminus\mathbb{Z}_{\leq0})^n$. Then, for all integers $j\geq0$,
$$\mathcal{Q}_{\bm\alpha,\bm\beta}(jp)=\mathcal{Q}_{\mathfrak{D}_p(\bm\alpha),\mathfrak{D}_p(\bm\beta)}(j)\omega,$$

where $\omega\in\mathbb{Z}_{(p)}^*$ and  $\omega\equiv1\bmod p$.

% If $\mathcal{Q}_{\bm\alpha,\bm\beta}(jp)\in\mathbb{Z}_{(p)}$ then $\mathcal{Q}_{\mathfrak{D}_p(\bm\alpha),\mathfrak{D}_p(\bm\beta)}(j)\in\mathbb{Z}_{(p)}$ and $$\mathcal{Q}_{\bm\alpha,\bm\beta}(jp)\bmod p=\mathcal{Q}_{\mathfrak{D}_p(\bm\alpha),\mathfrak{D}_p(\bm\beta)}(j)\bmod p.$$
\end{lemm}

\begin{proof}
If $j=0$ then there is nothing to prove. So, we suppose that $j>0$. Let $\gamma$ be in $\{\alpha_1,\ldots,\alpha_n,\beta_1,\ldots,\beta_n\}$ and let $r$ be the unique integer in $\{0,1,\ldots, p-1\}$ such that $p\mathfrak{D}_{p}(\gamma)-\gamma=r$. It is clear that $$(\gamma)_{jp}=\prod_{t=0}^{p-1}(\gamma+t)\prod_{t=0}^{p-1}(\gamma+p+t)\cdots\prod_{t=0}^{p-1}(\gamma+(j-1)p+t).$$
Note that, for all nonnegative integers $s$, $\mathfrak{D}_p(\gamma+sp)=\mathfrak{D}_p(\gamma)+s$ because $p(\mathfrak{D}_p(\gamma)+s)-(\gamma+sp)=r$. Then 
\begin{align*}
(\gamma)_{jp}=\prod_{s=0}^{j-1}\left(p(\mathfrak{D}_p(\gamma)+s)\underset{t\neq r}{\prod_{t=0}^{p-1}}(\gamma+sp+t)\right)&=p^j\prod_{s=0}^{j-1}(\mathfrak{D}_p({\gamma})+s)\prod_{s=0}^{j-1}\left(\underset{t\neq r}{\prod_{t=0}^{p-1}}(\gamma+sp+t)\right)\\
&=p^j(\mathfrak{D}_p(\gamma))_j\prod_{s=0}^{j-1}\left(\underset{t\neq r}{\prod_{t=0}^{p-1}}(\gamma+sp+t)\right).
\end{align*}
By Wilson's Theorem, it follows that, for all nonnegative integers $s$, $$\underset{t\neq r}{\prod_{t=0}^{p-1}}(\gamma+sp+t)\equiv(p-1)!\equiv-1\bmod p.$$
Therefore, 
\begin{equation}\label{eq_03} 
(\gamma)_{jp}=p^j(\mathfrak{D}_p(\gamma))_j\lambda,
\end{equation}
where $\lambda\in\mathbb{Z}_{(p)}^*$ and $\lambda\equiv(-1)^j\bmod p$.\\
Since $\gamma$ is an arbitrary element in $\{\alpha_1,\ldots,\alpha_n,\beta_1,\ldots,\beta_n\}$ and $$\mathcal{Q}_{\bm\alpha,\bm\beta}(jp)=\frac{(\alpha_1)_{jp}\cdots (\alpha_n)_{jp}}{(\beta_1)_{jp}\cdots(\beta_{n-1})_{jp}(1)_{jp}},$$ it follows from Equation~\eqref{eq_03} that 
\begin{equation}\label{eq_04}
\mathcal{Q}_{\bm\alpha,\bm\beta}(jp)=\mathcal{Q}_{\mathfrak{D}_p(\bm\alpha),\mathfrak{D}_p(\bm\beta)}(j)\omega,
\end{equation}
where $\omega\in\mathbb{Z}_{(p)}^*$ and $\omega\equiv1\bmod p$. \\

\end{proof}

\begin{lemma}\label{lemm_form}
Let $p$ be a prime number, let $\alpha$ be in $\mathbb{Z}_{(p)}$ and let $n\geq1$ be an integer. If $n=r_0+r_1p+\cdots+r_sp^s$ is the $p$-adic expansion of $n$ then $$(\alpha)_n \in p^{n_1}(\mathfrak{D}_p(\alpha))_{n_1}(\alpha+n_1p)_{r_0}\mathbb{Z}_{(p)}^*,$$
where $n_1=r_1+r_2p+\cdots+r_sp^{s-1}$.
\end{lemma}

\begin{proof}
It is not hard to see that $(\alpha)_n=\prod_{k=0}^{n_1-1}(\alpha+kp)_p\cdot(\alpha+n_1p)_{r_0}$. We know that there is a unique $r\in\{0,\ldots,p-1\}$ such that $\alpha+r=p\mathfrak{D}_p(\alpha)$. Then, for all integers $m\geq1$, $\alpha+mp+r=p(\mathfrak{D}_p(\alpha)+m)$. Hence, for all $m\geq1$, $$(\alpha+mp)_p=p(\mathfrak{D}_p(\alpha)+m)\underset{i\neq r}{\prod_{i=0}^{p-1}}(\alpha+mp+i).$$
But, for all $0\leq i<p$ such that $i\neq r$, $(\alpha+mp+i)\in\mathbb{Z}_{(p)}^*$ because $r$ is the unique element in  $\{0,\ldots,p-1\}$ such that $\alpha+r\in p\mathbb{Z}_{(p)}$. So, we conclude that  $$(\alpha)_n \in p^{n_1}(\mathfrak{D}_p(\alpha))_{n_1}(\alpha+n_1p)_{r_0}\mathbb{Z}_{(p)}^*.$$
\end{proof}

\begin{lemm}\label{lemm_padic_valuation}
 %Let $\bm\alpha=(\alpha_1,\ldots,\alpha_n)$, $\bm\beta=(\beta_1,\ldots,\beta_{n-1},1)$ be in $(\mathbb{Q}\cap\mathbb{Z}_{\leq0})^n$ and let $p$ be a primer number such that $\bm\alpha$, $\bm\beta$ belong to $(\mathbb{Z}_{(p)})^n$ and $f(z):={}_nF_{n-1}(\bm\alpha,\bm\beta;z)$ belongs to $\mathbb{Z}_{(p)}[[z]]$ and let $l$ be the order of $p$ in $(\mathbb{Z}/d_{\bm\alpha,\bm\beta}\mathbb{Z})^*$. Suppose that $(\bm\alpha, \bm\beta)$ satisfies the $\textbf{P}_{p,l}$ property. 
Let the assumptions be as in Lemma~\ref{lemm_first_step}. If $v_p(\mathcal{Q}_{\bm\alpha,\bm\beta}(r))>0$ then, for every  $j\in\mathbb{N}$, $v_p(\mathcal{Q}_{\bm\alpha,\bm\beta}(jp+r))>0.$
\end{lemm}

\begin{proof}
We split the proof into five steps.

\textbf{ Step I:}  We will prove that, for all integers $a\geq1$, $\mathfrak{D}_p^a(\alpha_i)-\mathfrak{D}_p^a(\beta_j)\in\mathbb{Z}_{(p)}^*$. Since $(\bm\alpha, \bm\beta)$ satisfies the $\textbf{P}_{p,l}$ propery, we know that, for all $1\leq m\leq l$, $\mathfrak{D}_p^m(\alpha_i)-\mathfrak{D}_p^m(\beta_j)\in\mathbb{Z}_{(p)}^*$ for all $1\leq i,j\leq n$. Therefore, it is sufficient to prove that, for all $a\geq1$,  $\mathfrak{D}^a_p(\bm\alpha)=\mathfrak{D}_p^q(\bm\alpha)$ and $\mathfrak{D}^a_p(\bm\beta)=\mathfrak{D}_p^q(\bm\beta)$ where $q=a\bmod l$ with $1\leq q< l$ if $a\neq0\bmod l$, and $q=l$ if $a=0\bmod l$. From the definition of $\mathfrak{D}_p$, it is not hard to see that, for all $a\geq1$, $d_{\mathfrak{D}_p^a(\bm\alpha),\mathfrak{D}_p^a(\bm\beta)}$ divides $d_{\bm\alpha,\bm\beta}$. Consequently,  $p^l\equiv 1\bmod d_{\mathfrak{D}_p^a(\bm\alpha),\mathfrak{D}_p^a(\bm\beta)}$ for all $a\geq1$. Further, for all $1\leq m\leq l$, $\mathfrak{D}^m_p(\bm\alpha)$, $\mathfrak{D}^m_p(\bm\beta)\in(\mathbb{Z}_{(p)}^*\cap(0,1])^n$ because, by assumption, $(\bm\alpha,\bm\beta)$ satisfies $(\textbf{P}1)$. So, by Lemma~\ref{cyclic}, we conclude that, for all $t\geq1$ and $1\leq m\leq l$, $\mathfrak{D}_p^{lt}(\mathfrak{D}_p^m(\bm\alpha))=\mathfrak{D}_p^m(\bm\alpha)$ and $\mathfrak{D}_p^{lt}(\mathfrak{D}_p^m(\bm\beta))=\mathfrak{D}_p^m(\bm\beta)$. Consequently, if $a=q+lt$ with $0\leq q<l$ and $q\neq0$ then $\mathfrak{D}_p^a(\bm\alpha)=\mathfrak{D}^q_p(\bm\alpha)$,  $\mathfrak{D}_p^a(\bm\beta)=\mathfrak{D}^q_p(\bm\beta)$ and if $q=0$, $\mathfrak{D}^a_p(\bm\alpha)=\mathfrak{D}_p^{(t-1)l}((\mathfrak{D}^l_p(\bm\alpha))=\mathfrak{D}^l_p(\bm\alpha)$, $\mathfrak{D}^a_p(\bm\beta)=\mathfrak{D}_p^{(t-1)l}((\mathfrak{D}^l_p(\bm\beta))=\mathfrak{D}^l_p(\bm\beta)$.

\textbf{Step II}  Let $(a,b)$ be in $\mathbb{Z}_{\geq0}\times\{0,\ldots, p\}$ and let $\gamma\in\{\alpha_1,\ldots,\alpha_n,\beta_1,\ldots,\beta_{n-1}\}$. If $v_p(\mathfrak{D}_p^a(\gamma)_{b})\geq1$ then $v_p(\mathfrak{D}_p^a(\gamma)_{b})=1$. In fact, we know that there is $c<b$ such that $\mathfrak{D}_p^a(\gamma)+c\in p\mathbb{Z}_{(p)}$ because $v_p(\mathfrak{D}_p^a(\gamma)_{b})\geq1$. So, $\mathfrak{D}_p^a(\gamma)+c=p\mathfrak{D}_p^{a+1}(\gamma)$ given that $c<b\leq p$. Therefore, $$\mathfrak{D}_p^a(\gamma)_{b}=p\mathfrak{D}_p^{a+1}(\gamma)\prod_{t=0,t\neq c}^{b-1}(\mathfrak{D}_p^a(\gamma)+t).$$ But the $p$-adic valuation of $\prod_{t=0,t\neq c}^{b-1}(\mathfrak{D}_p^a(\gamma)+t)$ is zero because $c$ is the unique element in $\{0,\ldots, p-1\}$ such that $\mathfrak{D}_p^a(\gamma)+c\in p\mathbb{Z}_{(p)}$. Hence, $v_p(\mathfrak{D}_p^a(\gamma)_{b})=1+v_p(\mathfrak{D}_p^{a+1}(\gamma))$. Now, we show that $v_p(\mathfrak{D}_p^{a+1}(\gamma))=0$. From the definition of $\mathfrak{D}_p$, it is clear that for all $t\geq1$, $d_{\mathfrak{D}_p^t(\bm\alpha),\mathfrak{D}_p^t(\bm\beta)}$ divides $d_{\bm\alpha,\bm\beta}$. Since, by assumption, $p>d_{\bm\alpha,\bm\beta}$, we get that, for  all $t\geq1$, $p>d_{\mathfrak{D}_p^t(\bm\alpha),\mathfrak{D}_p^t(\bm\beta)}$. Further, $\mathfrak{D}_p(\gamma)\in\mathbb{Z}_{(p)}^*\cap(0,1]$ because $(\bm\alpha, \bm\beta)$ satisfies the $\textbf{P}_{p,l}$ property. Thus, by (3) of Proposition~\ref{prop_d_p}, we deduce that, for all $t\geq1$, $\mathfrak{D}_p^t(\gamma)\in\mathbb{Z}_{(p)}^*$. In particular, $v_p(\mathfrak{D}_p^{a+1}(\gamma))=0$. Hence, $v_p(\mathfrak{D}_p^a(\gamma)_{b})=1$.

\textbf{Step III:} We now prove that, for all $(a,b)\in\mathbb{Z}_{\geq0}\times\{0,\ldots, p\}$, $$v_p(\mathcal{Q}_{\mathfrak{D}_p^a(\bm\alpha),\mathfrak{D}_p^a(\bm\beta)}(b))=\#\mathcal{P}_{\mathfrak{D}_p^a(\bm\alpha),b}-\#\mathcal{P}_{\mathfrak{D}_p^a(\bm\beta),b}.$$
To this end, it is sufficient to show that, for all $(a,b)\in\mathbb{Z}_{\geq0}\times\{0,\ldots, p\}$, $$v_p(\mathfrak{D}_p^a(\beta_1)_{b}\cdots\mathfrak{D}_p^a(\beta_n)_{b})=\#\mathcal{P}_{\mathfrak{D}_p^a(\bm\beta),b}\quad\text{ and }\quad v_p(\mathfrak{D}_p^a(\alpha_1)_{b}\cdots\mathfrak{D}_p^a(\alpha_n)_{b})=\#\mathcal{P}_{\mathfrak{D}_p^a(\bm\alpha),b}.$$ Let $i$ be in $\{0,\ldots,n\}$ such that $v_p(\mathfrak{D}_p^a(\beta_i)_{b})\geq1$. Then, according to Step II, $v_p(\mathfrak{D}_p^a(\beta_i)_{b})=1$.  Therefore, we get $v_p(\mathfrak{D}_p^a(\beta_1)_{r}\cdots\mathfrak{D}_p^a(\beta_n)_{b})=\#\mathcal{P}_{\mathfrak{D}_p^a(\bm\beta),b}$. Similarly, we also have the equality $v_p(\mathfrak{D}_p^a(\alpha_1)_{b}\cdots\mathfrak{D}_p^a(\alpha_n)_{b})=\#\mathcal{P}_{\mathfrak{D}_p^a(\bm\alpha),b}$. 

\textbf{Step IV:} In this step we prove that, for all $(a,b)\in\mathbb{Z}_{\geq0}\times\{0,\ldots, p\}$, $\#\mathcal{P}_{\mathfrak{D}_p^a(\bm\beta),b}\leq \#\mathcal{P}_{\mathfrak{D}_p^a(\bm\alpha),b}$. By assumption ${}_nF_{n-1}(\bm\alpha,\bm\beta)\in\mathbb{Z}_{(p)}[[z]]$. So, by Lemma~\ref{lemm_lambda_0}, we deduce that, for all $(a,b)\in\mathbb{Z}_{\geq0}\times\{0,\ldots, p\}$, $\mathcal{Q}_{\mathfrak{D}_p^a(\bm\alpha),\mathfrak{D}_p^a(\bm\beta)}(b)\in\mathbb{Z}_{(p)}$. Thus, from Step III, we conclude that $\#\mathcal{P}_{\mathfrak{D}_p^a(\bm\beta),b}\leq \#\mathcal{P}_{\mathfrak{D}_p^a(\bm\alpha),b}$. As a consequence, we get that, for any $(a,b)\in\mathbb{Z}_{\geq0}\times\{0,\ldots, p\}$, there is an injective map $\xi_{a,b}:\mathcal{P}_{\mathfrak{D}_p^a(\bm\beta),b}\rightarrow\mathcal{P}_{\mathfrak{D}_p^a(\bm\alpha),b}$. So, without losing any generality, we assume $\mathcal{P}_{\mathfrak{D}_p^a(\bm\beta),b}\subset\mathcal{P}_{\mathfrak{D}_p^a(\bm\alpha),b}$.

\textbf{Step V} Let $j\geq1$ be an integer. We now prove that  $v_p(\mathcal{Q}_{\bm\alpha,\bm\beta}(jp+r))>0$. The $p$-adic expansion of $jp+r$ is of the form $j_0+j_1p+\cdots +j_kp^k$, where $j_0=r$ and $j_s\in\{0,\ldots, p-1\}$ for all $1\leq s\leq k$. For all $0\leq s< k$, we set $\tau_s=j_{s+1}p+\cdots+j_{k}p^{k-s}$ and $\tau_k=0$. From Lemma~\ref{lemm_form}, we deduce that $$\mathcal{Q}_{\bm\alpha,\bm\beta}(jp+r)\in\Lambda\mathbb{Z}_{(p)}^{*}\quad\text{ with }\quad\Lambda=\prod_{s=0}^k\frac{(\mathfrak{D}_p^s(\alpha_1)+\tau_s))_{j_{s}}\cdots(\mathfrak{D}_p^s(\alpha_n)+\tau_s))_{j_{s}}}{(\mathfrak{D}_p^s(\beta_1)+\tau_s))_{j_{s}}\cdots(\mathfrak{D}_p^s(\beta_n)+\tau_s))_{j_{s}}}.$$
Thus, in order to show that $v_p(\mathcal{Q}_{\bm\alpha,\bm\beta}(jp+r))>0$, it is sufficient to see that $v_p(\Lambda)>0$. 

For every $s$ in $\{0,\ldots,k\}$, we put $\mathcal{J}_s=\{i\in\{0,\ldots,n\}:v_p((\mathfrak{D}_p^s(\beta_i)+\tau_s))_{j_{s}})\geq1\}$ and $\mathcal{I}_s=\{i\in\{0,\ldots,n\}:v_p((\mathfrak{D}_p^s(\alpha_i)+\tau_s))_{j_{s}})\geq1\}$. Actually, $\mathcal{J}_s=\mathcal{P}_{\mathfrak{D}_p^s(\bm\beta),j_s}$. In fact, if $i\in\mathcal{J}_s$ then there exists $k_s<j_s$ such that $\mathfrak{D}_p^s(\beta_i)+\tau_s+k_s\in p\mathbb{Z}_p$. Since $\tau_s\in p\mathbb{Z}_{(p)}$, we have $\mathfrak{D}_p^s(\beta_i)+k_s\in\mathbb{Z}_{(p)}$. Thus, $i\in\mathcal{P}_{\mathfrak{D}_p^s(\bm\beta),j_s}$. So $\mathcal{J}_s\subset\mathcal{P}_{\mathfrak{D}_p^s(\bm\beta),j_s}$. In a similar way, one obtains $\mathcal{P}_{\mathfrak{D}_p^s(\bm\beta),j_s}\subset\mathcal{J}_s$. Similarly, one shows that $\mathcal{I}_s=\mathcal{P}_{\mathfrak{D}_p^s(\bm\alpha),j_s}$. Thus, from Step IV, $\mathcal{J}_s\subset\mathcal{I}_s$. In addition, $\mathcal{J}_s=\mathcal{J}_{s,1}\cup\mathcal{J}_{s,>1}$, where $\mathcal{J}_{s,1}$ is the set of $i\in\mathcal{J}_s$ such that $v_p((\mathfrak{D}_p^s(\beta_i)+\tau_s))_{j_{s}})=1$ and $\mathcal{J}_{s,>1}$ is the complement of $\mathcal{J}_{s,1}$ in $\mathcal{J}_{s}$. %Let us show that $J_{k, >1}=\emptyset$. Suppose, to derive a contradiction, that $J_{k, >1}\neq\emptyset$. So, there is $i\in\{1,\ldots,n\}$ such that $v_p((\mathfrak{D}_p^k(\beta_i))_{j_k})>1$. Hence, $\mathfrak{D}_p^k(\beta_i)+q\in\ p\mathbb{Z}_{(p)}$  for some $q<j_k$. Whence, $\mathfrak{D}_p^k(\beta_i)+q=p\mathfrak{D}_p^{k+1}(\beta_i)$ because $q<j_k<p$. Further, it is clear that $$ (\mathfrak{D}_p^k(\beta_i))_{j_{k}}=p\mathfrak{D}_p^{k+1}(\beta_i)\prod_{t=0,t\neq q}^{j_s-1}(\mathfrak{D}_p^k(\beta_i)+t).$$
%But the $p$-adic valuation of $\prod_{t=0,t\neq q}^{j_k-1}(\mathfrak{D}_p^k(\beta_i)+t)$ is zero because $q<j_k<p$ and $q$ is the unique element in $\{0,\ldots, p-1\}$ such that $\mathfrak{D}_p^k(\beta_i)+q\in p\mathbb{Z}_{(p)}$. Thus, $v_p(p\mathfrak{D}_p^{k+1}(\beta_i))=v_p((\mathfrak{D}_p^k(\beta_i))_{j_{k}})$. Whence, $v_p(\mathfrak{D}_p^{k+1}(\beta_i))>1$. But, it follows from the proof of  Step I that $\mathfrak{D}_p^{k+1}(\beta_i)=\mathfrak{D}_p^{q'}(\beta_i)$ for some $q'\in\{1,\ldots, l\}$. Since $(\bm\alpha,\bm\beta)$ satisfies the $\textbf{P}_{p,l}$ property, we have $v_p(\mathfrak{D}_p^{q'}(\beta_i))=0$, which is a contradiction to $v_p(\mathfrak{D}_p^{k+1}(\beta_i))>1$.

It is easily checked that $\Lambda=\Psi\Theta,$ where 
\begin{align*}
\Psi=&\underset{\mathcal{J}_{s,>1}=\emptyset}{\prod\limits_{s=0}\limits^{k}}\frac{\underset{i\in\mathcal{J}_s}{\prod}\left(\mathfrak{D}_p^s(\alpha_i)+\tau_s\right)_{j_s}}{\underset{i\in\mathcal{J}_s}{\prod}\left(\mathfrak{D}_p^s(\beta_i)+\tau_s\right)_{j_s}}\cdot\frac{\underset{i\notin\mathcal{I}_s}{\prod}\left(\mathfrak{D}_p^s(\alpha_i)+\tau_s\right)_{j_s}}{\underset{i\notin\mathcal{J}_s}{\prod}\left(\mathfrak{D}_p^s(\beta_i)+\tau_s\right)_{j_s}}\\
&\underset{\mathcal{J}_{s,>1}\neq\emptyset}{\prod\limits_{s=0}\limits^{k}}\frac{\underset{i\in\mathcal{J}_{s}}{\prod}\left(\mathfrak{D}_p^s(\alpha_i)+\tau_s\right)_{j_s}}{\underset{i\in\mathcal{J}_{s,1}}{\prod}\left(\mathfrak{D}_p^s(\beta_i)+\tau_s\right)_{j_s}}\cdot\frac{\underset{i\notin\mathcal{I}_s}{\prod}\left(\mathfrak{D}_p^s(\alpha_i)+\tau_s\right)_{j_s}}{\underset{i\notin\mathcal{J}_s}{\prod}\left(\mathfrak{D}_p^s(\beta_i)+\tau_s\right)_{j_s}}
\end{align*}

and $$
\Theta=\underset{\mathcal{J}_{s,>1}=\emptyset}{\prod\limits_{s=0}\limits^{k}}\underset{i\in\mathcal{I}_s\setminus\mathcal{J}_s}{\prod}\left(\mathfrak{D}_p^s(\alpha_i)+\tau_s\right)_{j_s}\cdot\underset{\mathcal{J}_{s,>1}\neq\emptyset}{\prod\limits_{s=0}\limits^{k}}\frac{\underset{i\in\mathcal{I}_s\setminus\mathcal{J}_s}{\prod}\left(\mathfrak{D}_p^s(\alpha_i)+\tau_s\right)_{j_s}}{\underset{i\in\mathcal{J}_{s,>1}}{\prod}\left(\mathfrak{D}_p^s(\beta_i)+\tau_s\right)_{j_s}}.
$$
We also have $\Theta=\Theta_0\Theta_1$, where 
$$\Theta_0=\underset{i\in\mathcal{I}_0\setminus\mathcal{J}_0}{\prod}\left(\mathfrak{D}_p^s(\alpha_i)+\tau_0\right)_{j_0}\quad\text{ and }\quad\Theta_1=\frac{{\prod\limits_{s=1}\limits^{k}}\underset{i\in\mathcal{I}_s\setminus\mathcal{J}_s}{\prod}\left(\mathfrak{D}_p^s(\alpha_i)+\tau_s\right)_{j_s}}{\underset{\mathcal{J}_{s,>1}\neq\emptyset}{\prod\limits_{s=0}\limits^{k}}{\underset{i\in\mathcal{J}_{s,>1}}{\prod}\left(\mathfrak{D}_p^s(\beta_i)+\tau_s\right)_{j_s}}}.
$$
We now prove $v_p(\Psi\Theta_1)\geq 0$. From the definition of $\Psi$, it is clear that $$v_p(\Psi)\geq\underset{\mathcal{J}_{s,>1}\neq\emptyset}{\sum\limits_{s=0}\limits^{k}}(\#\mathcal{J}_s-\#\mathcal{J}_{s,1}).$$ So, in order to prove that $v_p(\Psi\Theta_1)\geq 0$, it is sufficient to prove that 
\begin{equation}\label{eq_theta_1}
v_p(\Theta_1)\geq\underset{\mathcal{J}_{s,>1}\neq\emptyset}{\sum\limits_{s=0}\limits^{k}}(\#\mathcal{J}_{s,1}-\#\mathcal{J}_s).
\end{equation}
Let $s$ be in $\{0,\ldots, k\}$ such that $J_{s,>1}\neq\emptyset$. Note that $k-s>0$ because $J_{k,>1}=\emptyset$\footnote{If $J_{k,>1}\neq\emptyset$ then there is $i\in\{1,\ldots,n\}$ such that $v_p(\mathfrak{D}_p^k(\beta_i)_{j_k})>1$. Since $j_k<p$, Step II implies  $v_p(\mathfrak{D}_p^k(\beta_i)_{j_k})=1$ which is a contradiction. Thus, $J_{k,>1}=\emptyset$.}.  Let $i\in J_{s,>1}$ and let $l=v_p( (\mathfrak{D}_p^s(\beta_i)+\tau_s))_{j_{s}})$.  Then $l\geq2$ and there exists $k_s<j_s$ such that $\mathfrak{D}_p^s(\beta_i)+\tau_s+k_s\in p\mathbb{Z}_{(p)}$. We now proceed to prove some properties which are crucial to prove Equation~\eqref{eq_theta_1}.

(\textbf{A}). We prove that $v_p(\mathfrak{D}_p^s(\beta_i)+\tau_s+k_s)=l$. In fact, we have  $$ (\mathfrak{D}_p^s(\beta_i)+\tau_s))_{j_{s}}=(\mathfrak{D}_p^s(\beta_i)+\tau_s+k_s)\prod_{t=0,t\neq k_s}^{j_s-1}(\mathfrak{D}_p^a(\beta_i)+\tau_s+t).$$
But the $p$-adic valuation of $\prod_{t=0,t\neq k_s}^{j_s-1}(\mathfrak{D}_p^a(\beta_i)+\tau_s+t)$ is zero because $k_s<j_s<p$ and $k_s$ is the unique element in $\{0,\ldots, p-1\}$ such that $\mathfrak{D}_p^s(\beta_i)+\tau_s+k_s\in p\mathbb{Z}_{(p)}$. Thus, $v_p(\mathfrak{D}_p^s(\beta_i)+\tau_s+k_s)=l$. In particular, we have $\mathfrak{D}_p^s(\beta_i)+\tau_s+k_s=p^l\mu$, with $\mu\in\mathbb{Z}_{(p)}^*$.

(\textbf{B}). We now show that, for all $1\leq m\leq\min\{k-s,l-1\}$, $\mathfrak{D}_p^{s+m}(\beta_i)+j_{s+m}\in p\mathbb{Z}_{(p)}$ and that $$\mathfrak{D}_p^{s+m}(\beta_i)+j_{s+m}+j_{s+m+1}p+\cdots+j_{k}p^{k-s-m}=p^{l-m}\mu.$$ We proceed by induction on $m\in\{1,\ldots, q-1\}$, where $q=\min\{k-s,l-1\}$. From (\textbf{A}), we have $\mathfrak{D}_p^s(\beta_i)+\tau_s+k_s=p^l\mu$, with $\mu\in\mathbb{Z}_{(p)}^*$.  As $\tau_s\in p\mathbb{Z}$ then $\mathfrak{D}_p^s(\beta_i)+k_s\in p\mathbb{Z}_{(p)}$. Therefore, $\mathfrak{D}_p^s(\beta_i)+k_s=p\mathfrak{D}_{p}^{s+1}(\beta_i)$ because $k_s<j_s<p$. Whence, $p\mathfrak{D}_{p}^{s+1}(\beta_i)+\tau_s=p^l\mu$. Remember that $\tau_s=j_{s+1}p+\cdots+k_{k}p^{k-s}$. Hence, $\mathfrak{D}_{p}^{s+1}(\beta_i)+j_{s+1}+j_{s+2}p+\cdots+j_{k}p^{k-s-1}=p^{l-1}\mu$ and $\mathfrak{D}_{p}^{s+1}(\beta_i)+j_{s+1}\in p\mathbb{Z}_{(p)}$. We now suppose that for some $m\in\{1,\ldots, q-2\}$, $\mathfrak{D}_p^{s+m}(\beta_i)+j_{s+m}\in p\mathbb{Z}_{(p)}$ and that $$\mathfrak{D}_p^{s+m}(\beta_i)+j_{s+m}+j_{s+m+1}p+\cdots+j_{k}p^{k-s-m}=p^{l-m}\mu.$$ We have $\mathfrak{D}_p^{s+m}(\beta_i)+j_{s+m}= p\mathfrak{D}_p^{s+m+1}(\beta_i)$ because $j_{s+m}<p$ and, by induction hypothesis, $\mathfrak{D}_p^{s+m}(\beta_i)+j_{s+m}\in p\mathbb{Z}_{(p)}$. So $$\mathfrak{D}_p^{s+m+1}(\beta_i)+j_{s+m+1}+\cdots+j_{k}p^{k-s-m}-1=p^{l-m-1}\mu.$$

(\textbf{C}). We now see that $l\leq k-s+1$. Suppose, towards a contradiction, that $l>k-s+1$. From (\textbf{B}), we know that, for all $m\in\{1,\ldots, k-s\}$, $$\mathfrak{D}_p^{s+m}(\beta_i)+j_{s+m}+j_{s+m+1}p+\cdots+j_{k}p^{k-s-m}=p^{l-m}\mu.$$ In particular, $\mathfrak{D}_p^{k}(\beta_i)+j_k=p^{l-k+s}\mu$ and $\mathfrak{D}_p^{k}(\beta_i)+j_k\in p\mathbb{Z}_{(p)}$. Hence, $\mathfrak{D}_p^{k}(\beta_i)+j_k=p\mathfrak{D}_p^{k+1}(\beta_i)$ because $j_k<p$. So $\mathfrak{D}_p^{k+1}(\beta_i)=p^{l-k+s-1}\mu$. But $l-k+s-1>0$ and hence, $\mathfrak{D}_p^{k+1}(\beta_i)\in p\mathbb{Z}_{(p)}$. From the definition of $\mathfrak{D}_p$, it is clear that, for all $t\geq1$, $d_{\mathfrak{D}_p^t(\bm\alpha),\mathfrak{D}_p^t(\bm\beta)}$ divides $d_{\bm\alpha,\bm\beta}$. Since by assumption, $p>d_{\bm\alpha,\bm\beta}$, we get that, for  all $t\geq1$, $p>d_{\mathfrak{D}_p^t(\bm\alpha),\mathfrak{D}_p^t(\bm\beta)}$. Further, $\mathfrak{D}_p(\beta_i)\in\mathbb{Z}_{(p)}^*\cap(0,1]$ because $(\bm\alpha, \bm\beta)$ satisfies the $\textbf{P}_{p,l}$ property. Thus, by (3) of Proposition~\ref{prop_d_p}, we deduce that, for all $t\geq1$, $\mathfrak{D}_p^t(\beta_i)\in\mathbb{Z}_{(p)}^*$. In particular, $\mathfrak{D}_p^{k+1}(\beta_i)\in \mathbb{Z}_{(p)}^*$, which is a contradiction to the fact that $\mathfrak{D}_p^{k+1}(\beta_i)\in p\mathbb{Z}_{(p)}$. Consequently, $l\leq k-s+1$.

(\textbf{D}). Now, we prove see that, for every $m\in\{1,\ldots,l-1\}$, $i\in\mathcal{I}_{s+m}\setminus\mathcal{J}_{s+m}$. From (\textbf{B}) and (\textbf{C}), we have $\mathfrak{D}_p^{s+m}(\beta_i)+j_{s+m}\in p\mathbb{Z}_{(p)}$ for all $m\in\{1,\ldots,l-1\}$. Then $i\in\mathcal{P}_{\mathfrak{D}_p^{s+m}(\bm\beta),j_{s+m}+1}$. By Step IV, we have $\mathcal{P}_{\mathfrak{D}_p^{s+m}(\bm\beta),j_{s+m}+1}\subset \mathcal{P}_{\mathfrak{D}_p^{s+m}(\bm\alpha),j_{s+m}+1}$. Hence, $i\in\mathcal{P}_{\mathfrak{D}_p^{s+m}(\bm\alpha),j_{s+m}+1}$. From Step I, we know that, for all $a\geq1$, $\mathfrak{D}_p^{a}(\alpha_i)-\mathfrak{D}_p^a(\beta_j)$ belongs to $\mathbb{Z}_{(p)}^*$ for all $1\leq i,j\leq n$. Since $\mathfrak{D}_p^{s+m}(\beta_i)+j_{s+m}\in p\mathbb{Z}_{(p)}$, we get that, for all $r\in\{1,\ldots,n\}$, $\mathfrak{D}_p^{s+m}(\alpha_r)+j_{s+m}\notin p\mathbb{Z}_{(p)}$. Hence, $i\in\mathcal{P}_{\mathfrak{D}_p^{s+m}(\bm\alpha),j_{s+m}}=\mathcal{I}_{s+m}.$ We now show that $i\notin\mathcal{J}_{s+m}$. Aiming for a contradiction, suppose that $i\in\mathcal{J}_{s+m}$. Thus, there is $d<j_{s+m}$ such that $\mathfrak{D}_p^{s+m}(\beta_i)+d +\tau_{j_{s+m}}\in p\mathbb{Z}_p$. Since $\tau_{j_{s+m}}\in p\mathbb{Z}$, $\mathfrak{D}_p^{s+m}(\beta_i)+d\in p\mathbb{Z}_p$. But we know that $\mathfrak{D}_p^{s+m}(\beta_i)+j_{s+m}\in p\mathbb{Z}_{(p)}$. Thus, $j_{s+m}-d\in p\mathbb{Z}$. That is a contradiction because $0\leq d<j_{s+m}<p$. Consequently, for all $m\in\{1,\ldots, l-1\}$, $i\in\mathcal{I}_{s+m}\setminus\mathcal{J}_{s+m}$. In particular, $v_p(\mathfrak{D}_p^s(\alpha_i)+\tau_{s+m})_{j_{s+m}})\geq1$. Whence, $$v_p\left(\frac{\prod_{m=1}^{l-1}(\mathfrak{D}_p^s(\alpha_i)+\tau_{s+m})_{j_{s+m}}}{(\mathfrak{D}^s_p(\beta_i)+\tau_{j_s})_{j_s}}\right)\geq -1.$$

From (\textbf{C}), we have $l\leq k-s+1$. Thus, for all $m\in\{1,\ldots, l-1\}$, $s+m\leq k$. So, from (\textbf{D}) it follows that the product $\prod_{m=1}^{l-1}(\mathfrak{D}_p^s(\alpha_i)+\tau_{s+m})_{j_{s+m}}$ is a factor of ${\prod\limits_{s=1}\limits^{k}}\underset{i\in\mathcal{I}_s\setminus\mathcal{J}_s}{\prod}\left(\mathfrak{D}_p^s(\alpha_i)+\tau_s\right)_{j_s}$. Consequenly, $$v_p(\Theta_1)\geq\underset{\mathcal{J}_{s,>1}\neq\emptyset}{\sum\limits_{s=0}\limits^{k}}(\#\mathcal{J}_{s,1}-\#\mathcal{J}_s)$$ because $\#\mathcal{J}_{s,>1}=\#\mathcal{J}_{s}-\#\mathcal{J}_{s,1}$.

Finally, from Step III, we have $$v_p(\mathcal{Q}_{\bm\alpha,\bm\beta}(r))=\#\mathcal{P}_{\bm\alpha,r}-\#\mathcal{P}_{\bm\beta,r}.$$ 
By assumption, $v_p(\mathcal{Q}_{\bm\alpha,\bm\beta}(r))>0$.
Since $r=j_0$, it is not hard to see that $\mathcal{J}_0=\mathcal{P}_{\bm\beta,r}$ and that $\mathcal{I}_0=\mathcal{P}_{\bm\alpha,r}$. Thus $\#\mathcal{J}_0<\#\mathcal{I}_0$. Whence, $v_p(\Theta_0)>0$.
This completes the proof.

\end{proof}

\begin{lemm}\label{lemm_exp}
Let $\bm\alpha=(\alpha_1,\ldots,\alpha_n)$ and $\bm\beta=(\beta_1,\ldots,\beta_{n-1},1)$ be in $(\mathbb{Q}\setminus\mathbb{Z}_{\leq0})^n$ and let $p$ be a prime number such that $\mathcal{H}(\bm\alpha,\bm\beta)\in\mathbb{Z}_{(p)}[z][\delta]$. Suppose that $A(z)=a_mz^{m}+\cdots+a_rz^{r}\in\mathbb{F}_p[z]$ is a solution of $\mathcal{H}(\bm\alpha,\bm\beta,p)$. If $a_m\neq0$ then $m\bmod p$ is an exponent at zero of  $\mathcal{H}(\bm\alpha,\bm\beta,p)$, that is, $m\bmod p$ belongs to $\{0,1-\beta_1\bmod p,\ldots, 1-\beta_{n-1}\bmod p\}$.
\end{lemm}

\begin{proof}
It is not hard to see that $$\mathcal{H}(\bm\alpha,\bm\beta,p)(A(z))=\prod_{j=1}^{n}(m+\beta_j-1\bmod p)a_mz^m+z^{m+1}B(z),$$
where $B(z)$ is a polynomial. As $\mathcal{H}(\bm\alpha,\bm\beta,p)(A(z))=0$ then $\prod_{j=1}^{n}(m+\beta_j-1\bmod p) a_m=0$. So, $\prod_{j=1}^{n}(m+\beta_j-1\bmod p)=0$ because, by hypothesis, $a_m\neq0$. Thus, we have $m\bmod p=1-\beta_j\bmod p$ for some $j\in\{1,\ldots,n\}$.

\end{proof}

\subsection{Proof of Lemma~\ref{lemm_cartier_operator}}

Let us write $E_{\bm\alpha,\bm\beta,p}=\{e_0, e_1,\ldots, e_k\}$, where $e_i<e_{i+1}$ for all $i\in\{0,\ldots,k\}$ and $e_0=0$. We set $e_{k+1}=p$. So,
for all $e_i\in\{e_0,e_1,\ldots, e_{k}\}$ and, for all nonnegative intergers $j$, we set $$P_{j,e_i}(z)=\sum_{s=jp+e_i}^{jp+e_{i+1}-1}(\mathcal{Q}_{\bm\alpha,\bm\beta}(s)\bmod p)z^{s}.$$

We split the proof into three steps:

\textbf{Step I}. For all $e_i\in\{e_0,\ldots, e_k\}$ and for all nonnegative integers $j$, the polynomial $P_{j,e_i}(z)$ is a solution of $\mathcal{H}(\bm\alpha,\bm\beta, p)$. 

\textbf{Step II}.  If $\mathcal{Q}_{\bm\alpha,\bm\beta}(e_i)\bmod p=0$ then, for every integer $j\geq0$, $P_{j,e_i}(z)$ is the zero polynomial. 

\textbf{Step III}. If $\mathcal{Q}_{\bm\alpha,\bm\beta}(e_i)\bmod p\neq0$  then, for every integer $j\geq0$, $$P_{j,e_i}(z)=\frac{\mathcal{Q}_{\bm\alpha,\bm\beta}(jp+e_i)}{\mathcal{Q}_{\bm\alpha,\bm\beta}(e_i)}\bmod p\cdot z^{jp}P_{0,e_i}(z).$$

\begin{proof}[Proof of Step I]
 It is not hard to see that 
\begin{equation}\label{eq_01}
\mathcal{H}(\bm\alpha,\bm\beta)=(1-z)\delta^{n}+[S_{n,1}(\bm\beta-\textbf{1})-zS_{n,1}(\bm\alpha)]\delta^{n-1}+\cdots+S_{n,n}(\bm\beta-\textbf{1})-zS_{n,n}(\bm\alpha),
\end{equation}
 where $\textbf{1}=(1,\ldots,1)\in\mathbb{N}^n$ and $S_{n,r}=\sum_{1\leq i_{1}<\cdots<i_{r}\leq n}X_{i_1}\cdots X_{i_{r}}.$

Now, we set $I(z)=\prod_{i=1}^n(z+\beta_i-1)$ and $T(z)=\prod_{i=1}^n(z+\alpha_i)$. Then, it follows from Equality~\eqref{eq_01} that 
\begin{align*}
\mathcal{H}(\bm\alpha,\bm\beta,p)(P_{j,e_i}(z))&=(I(jp+e_i)\mathcal{Q}_{\bm\alpha,\bm\beta}(jp+e_i)\bmod p)z^{jp+e_i}\\
&+\sum_{k=jp+e_i+1}^{jp+e_{i+1}-1}((I(k)\mathcal{Q}_{\bm\alpha,\bm\beta}(k)-T(k-1)\mathcal{Q}_{\bm\alpha,\bm\beta}(k-1))\bmod p) z^k\\
&-(T(jp+e_{i+1}-1)\mathcal{Q}_{\bm\alpha,\bm\beta}(jp+e_{i+1}-1)\bmod p)z^{jp+e_{i+1}}.
\end{align*}

We now prove that $\mathcal{H}(\bm\alpha,\bm\beta,p)(P_{j,e_i}(z))=0$. Recall that, by definition, $e_i\bmod p$ belongs to $\{0,1-\beta_1,\ldots, 1-\beta_{n-1}\}\bmod p$. Then, we have  $I(e_i)\bmod p\equiv0$. But, $I(jp+e_i)\equiv I(e_i)\bmod p$. Thus, $I(jp+e_i)\bmod p\equiv0$.  So that, $I(jp+e_i)\mathcal{Q}_{\bm\alpha,\bm\beta}(jp+e_i)\bmod p\equiv0$. Furthermore, it is clear that, for every positive integer $t$ we have 
$$I(t)\mathcal{Q}_{\bm\alpha,\bm\beta}(t)-\mathcal{Q}_{\bm\alpha,\bm\beta}(t-1)T(t-1)=0.$$ By hypotheses,  $\bm\alpha$ and $\bm\beta$ belong to $\mathbb{Z}_{(p)}^n$. Thus, for every integer $t$, $I(t)$ and $T(t-1)$ belong to $\mathbb{Z}_{(p)}$. Again, by hypotheses, we know that ${}_nF_{n-1}(\bm\alpha,\bm\beta;z)$ belongs to $\mathbb{Z}_{(p)}[[z]]$. Consequently, for every integer $t\geq1$, $\mathcal{Q}_{\bm\alpha,\bm\beta}(t)$ and $\mathcal{Q}_{\bm\alpha,\bm\beta}(t-1)$ belong to $\mathbb{Z}_{(p)}$. Therefore, from the previous equality we conclude that, for every positive integer $t$,  
\begin{equation}\label{eq_02}
(I(t)\mathcal{Q}_{\bm\alpha,\bm\beta}(t)-\mathcal{Q}_{\bm\alpha,\bm\beta}(t-1)T(t-1))\bmod p=0.
\end{equation}
In particular, for every $k\in\{jp+e_i+1,\ldots, jp+e_{i+1}-1\}$, we have $$(I(k)\mathcal{Q}_{\bm\alpha,\bm\beta}(k)-T(k-1)\mathcal{Q}_{\bm\alpha,\bm\beta}(k-1))\bmod p=0$$ and $$(I(jp+e_{i+1})\mathcal{Q}_{\bm\alpha,\bm\beta}(jp+e_{i+1})-\mathcal{Q}_{\bm\alpha,\bm\beta}(jp+e_{i+1}-1)T(jp+e_{i+1}-1))\bmod p=0.$$
 Since $e_{i+1}\bmod p$ is an exponent at zero of $\mathcal{H}(\bm\alpha,\bm\beta,p)$, $I(e_{i+1})\bmod p\equiv0$. But, it is clear that $I(jp+e_{i+1})\equiv I(e_{i+1})\bmod p$. Thus, $I(jp+e_{i+1})\bmod p\equiv0$. Therefore, we have $$\mathcal{Q}_{\bm\alpha,\bm\beta}(jp+e_{i+1}-1)T(jp+e_{i+1}-1))\equiv0\bmod p.$$
So that, $\mathcal{H}(\bm\alpha,\bm\beta,p)(P_{j,e_i}(z))=0$. 

\end{proof}

\begin{proof}[Proof of Step II] Suppose that $\mathcal{Q}_{\bm\alpha,\bm\beta}(e_i)\bmod p=0$. Let $j$ be a nonnegative integer. We want to show that $P_{j,e_i}(z)$ is the zero polynomial. For this purpose, we show by induction on $s\in\{jp+e_i,jp+e_i+1,\ldots, jp+e_{i+1}-1\}$ that $\mathcal{Q}_{\bm\alpha,\bm\beta}(s)\equiv0\bmod p$. Since $v_p(\mathcal{Q}_{\bm\alpha,\bm\beta}(e_i))>0$, by Lemma~\ref{lemm_padic_valuation}, we have $v_p(\mathcal{Q}_{\bm\alpha,\bm\beta}(jp+e_i))>0$. So that,$\mathcal{Q}_{\bm\alpha,\bm\beta}(jp+e_i))\bmod p=0$.  Now, suppose that $\mathcal{Q}_{\bm\alpha,\bm\beta}(s)\equiv0\bmod p$ for some $s$ in the set $\{jp+e_i,jp+e_i+1,\ldots, jp+e_{i+1}-2\}$. From Equation $\eqref{eq_02}$ we know that $(I(s+1)\mathcal{Q}_{\bm\alpha,\bm\beta}(s+1)-\mathcal{Q}_{\bm\alpha,\bm\beta}(s)T(s))\bmod p=0$. By applying our induction hypothesis, we obtain $I(s+1)\mathcal{Q}_{\bm\alpha,\bm\beta}(s+1)\bmod p=0$. Suppose, towards a contradiction, that $I(s+1)\equiv0\bmod p$. Then, $s+1\bmod p$ is an exponent at  zero of $\mathcal{H}(\bm\alpha,\bm\beta,p)$ and since $jp+e_i<s+1<jp+e_{i+1}$, we have $e_i<s+1-jp<e_{i+1}$. So, $0\leq s+1-jp<p$ and  therefore, $s+1-jp\in E_{\bm\alpha,\bm\beta,p}$. Hence, there is $m\in\{0,\ldots,k\}$ such that $e_m=s+1-jp$. Then, $e_i<e_m<e_{i+1}$. Now, we know that if $i<m$ then $e_{i+1}\leq e_m$ and if $m<i$ then $e_m<e_i$. This is a clear contradiction of the fact that $e_i<e_m<e_{i+1}$. Thus, $I(s+1)\bmod p\neq0$. Then, it follows that $\mathcal{Q}_{\bm\alpha,\bm\beta}(s+1)\bmod p=0$. Therefore, $P_{j,e_i}(z)$ is the zero polynomial.

\end{proof}

\begin{proof}[Proof of Step III]
  From Step I, we know that $P_{j,e_i}$ and $P_{0,e_i}$ are solutions of  $\mathcal{H}(\bm\alpha,\bm\beta,p)$. Thus, the polynomial $P_{j,e_i}(z)-\frac{\mathcal{Q}_{\bm\alpha,\bm\beta}(jp+e_i)}{\mathcal{Q}_{\bm\alpha,\bm\beta}(e_i)}\bmod p\cdot z^{jp}P_{0,e_i}(z)$ is a solution of $\mathcal{H}(\bm\alpha,\bm\beta,p)$. Suppose, to derive a contradiction, that this polynomial is not zero. Then, from Lemma~\ref{lemm_exp} it follows that the differential operator $\mathcal{H}(\bm\alpha,\bm\beta,p)$ has an exponent at zero in the set $\{(e_i+1)\bmod p,\ldots, (e_{i+1}-1)\bmod p\}$ because 
\begin{align*}
& P_{j,e_i}(z)-\frac{\mathcal{Q}_{\bm\alpha,\bm\beta}(jp+e_i)}{\mathcal{Q}_{\bm\alpha,\bm\beta}(e_i)}\bmod p\cdot z^{jp}P_{0,e_i}(z)\\
&=\sum_{s=e_i+1}^{e_{i+1}-1}\left(\mathcal{Q}_{\bm\alpha,\bm\beta}(jp+s)-\frac{\mathcal{Q}_{\bm\alpha,\bm\beta}(jp+e_i)\mathcal{Q}_{\bm\alpha,\bm\beta}(s)}{\mathcal{Q}_{\bm\alpha,\bm\beta}(e_i)}\bmod p\right)z^{jp+s}.
\end{align*}
 Therefore, there exits $e_m\in E_{\bm\alpha,\beta,p}\cap\{e_i+1,\ldots, e_{i+1}-1\}$.  If $i<m$ then $e_{i+1}\leq e_m$ and if $m<i$ then $e_m<e_i$. This leads to the contradiction that $e_m\in\{e_i+1,\ldots, e_{i+1}-1\}$. 
 
 Consequently, the polynomial $P_{j,e_i}(z)-\frac{\mathcal{Q}_{\bm\alpha,\bm\beta}(jp+e_i)}{\mathcal{Q}_{\bm\alpha,\bm\beta}(e_i)}\bmod p\cdot z^{jp}P_{0,e_i}(z)$ is the zero polynomial. 
 
 \end{proof}
 
Now, we are in a position to finish the proof of the lemma. Let us write $S_{\bm\alpha,\bm\beta,p}=\{0,r_1,\ldots, r_t\}$. It is clear that $$f(z)\bmod p=\sum_{j\geq0}\left(\sum_{i=0}^kP_{j,e_i}(z)\right).$$ From Step II we deduce that if $e_i\in E_{\bm\alpha,\bm\beta,p}\setminus S_{\bm\alpha,\bm\beta,p}$ then, for all integers $j\geq0$, $P_{j,e_i}(z)$ is the zero polynomial. Then, we have 
 
 $$f(z)\bmod p=\sum_{j\geq0}P_{j,0}(z)+\sum_{j\geq0}P_{j,r_1}(z)+\cdots+\sum_{j\geq0}P_{j,r_t}(z).$$ From Step III we conclude that if $r_i\in S_{\bm\alpha,\bm\beta,p}$ then, for all integers $j\geq0$, \[P_{j,r_i}(z)=\mathcal{Q}_{\bm\alpha,\bm\beta}(jp+r_i)\bmod p\cdot z^{jp}P_{i}(z),\] where $P_i(z)=\sum_{s=r_i}^{e_j-1}\frac{\mathcal{Q}_{\bm\alpha,\bm\beta}(s)}{\mathcal{Q}_{\bm\alpha,\bm\beta}(r_i)}z^s$ with $e_{j-1}=r_i$.
  
  Therefore, 
  \begin{align*}
  \sum_{j\geq0}P_{j,r_i}(z)&=\sum_{j\geq0}(\mathcal{Q}_{\bm\alpha,\bm\beta}(jp+r_i)\bmod p\cdot z^{jp}P_i(z))\\
  &=P_i(z)\sum_{j\geq0}(\mathcal{Q}_{\bm\alpha,\bm\beta}(jp+r_i)\bmod p\cdot z^{jp})\\
  &=P_i(z)\Lambda_{r_i}(f)^p.
  \end{align*}
Consequently, we have   $$f(z)\equiv\sum_{r\in S_{\bm\alpha,\bm\beta,p}} P_r(z)\Lambda_r(f)^p\bmod p\quad\text{ with }\quad P_r(z)=\sum_{s=r}^{r'-1}\frac{\mathcal{Q}_{\bm\alpha,\bm\beta}(s)}{\mathcal{Q}_{\bm\alpha,\bm\beta}(r)}z^s.$$

 $\hfill\square$

\subsection{Auxiliary result II}\label{sec_aux_2}

With the aim of carrying out the proof of Lemma~\ref{lemm_first_step}, we need one more auxiliary result.

\begin{lemm}\label{lemm_ultra_tec}
Let $p$ be a primer number, let $r$ be in $\{0,1,\ldots,p-1\}$ and let $\bm\alpha=(\alpha_1,\ldots,\alpha_n)$ and $\bm\beta=(\beta_1,\ldots,\beta_{n-1},1)$ be in $(\mathbb{Z}_{(p)}\setminus\mathbb{Z}_{\leq0})^n$. Consider the following elements:  $$\theta=\frac{\underset{s\in\mathcal{C}_{\bm\alpha,r}}{\prod}(\alpha_s)_{r}}{\underset{s\in\mathcal{C}_{\bm\beta,r}}{\prod}(\beta_s)_{r}}\text{, }\tau=\frac{\underset{s\in\mathcal{P}_{\bm\alpha,r}}{\prod}\left(\underset{t\neq p\mathfrak{D}_p(\alpha_s)-\alpha_s}{\prod\limits_{t=0}\limits^{r-1}(\alpha_s+t)}\right)}{\underset{s\in\mathcal{P}_{\bm\beta,r}}{\prod}\left(\underset{t\neq p\mathfrak{D}_p(\beta_s)-\beta_s}{\prod\limits_{t=0}\limits^{r-1}(\beta_s+t)}\right)},\text{ and }\lambda_j=\frac{\underset{s\in\mathcal{P}_{\bm\alpha,r}}{\prod}(\mathfrak{D}_p(\alpha_s)+j)}{\underset{s\in\mathcal{P}_{\bm\beta,r}}{\prod}(\mathfrak{D}_p(\beta_s)+j)},$$
with $j\in\mathbb{Z}_{\geq0}$.
 Suppose that $\mathfrak{D}_p(\bm\alpha)$, $\mathfrak{D}_p(\bm\beta)$ belong to $(\mathbb{Z}_{(p)}^*)^n$ and that $v_p(\mathcal{Q}_{\bm\alpha,\bm\beta}(r))=0$. Then:
 \begin{enumerate}
 \item  $v_p(\lambda_0)=v_p(\theta)=v_p(\tau)=0,$
 \item  $\#\mathcal{P}_{\bm\alpha,r}=\#\mathcal{P}_{\bm\beta,r}$ and $\mathcal{Q}_{\bm\alpha,\bm\beta}(r)=\lambda_0\tau\theta$,
  \item  for every integer $j\geq0$, $\mathcal{Q}_{\bm\alpha,\bm\beta}(jp+r)=\mathcal{Q}_{\mathfrak{D}_p(\bm\alpha),\mathfrak{D}_p(\bm\beta),p}(j)\lambda_j\nu$, where $\nu\in\mathbb{Z}_{(p)}^*$ and $$\nu\bmod p=(\tau\bmod p)(\theta\bmod p),$$
 \item if for every integer $j\geq0$, $v_p(\mathcal{Q}_{\bm\alpha,\bm\beta}(j))\geq0$, then, for every $j\geq0$, we have $v_p(\mathcal{Q}_{\mathfrak{D}_p(\bm\alpha),\mathfrak{D}_p(\bm\beta)}(j)\lambda_j)\geq0$ and $$v_p\left(\frac{\underset{s\in\mathcal{C}_{\bm\alpha,r}}{\prod}\mathfrak{D}_p(\alpha_s)_j {\underset{s\in\mathcal{P}_{\bm\alpha,r}}{\prod}(\mathfrak{D}_p(\alpha_s)+1)_j}}{\underset{s\in\mathcal{C}_{\bm\beta,r}}{\prod}\mathfrak{D}_p(\beta_s)_j \underset{s\in\mathcal{P}_{\bm\beta,r}}{\prod}(\mathfrak{D}_p(\beta_s)+1)_j}\right)\geq0,$$
\item if for every integer $j\geq0$, $v_p(\mathcal{Q}_{\bm\alpha,\bm\beta}(j))\geq0$, then, for every integer $j\geq0$,  
%$$\mathcal{Q}_{\bm\alpha,\bm\beta}(jp+r)\bmod p\equiv(\mathcal{Q}_{\mathfrak{D}_p(\bm\alpha),\mathfrak{D}_p(\bm\beta)}(j)\lambda_j\bmod p)(\theta\bmod p)(\tau\bmod p)$$
$$\mathcal{Q}_{\bm\alpha,\bm\beta}(jp+r)\bmod p=\left(\frac{\underset{s\in\mathcal{C}_{\bm\alpha,r}}{\prod}\mathfrak{D}_p(\alpha_s)_j {\underset{s\in\mathcal{P}_{\bm\alpha,r}}{\prod}(\mathfrak{D}_p(\alpha_s)+1)_j}}{\underset{s\in\mathcal{C}_{\bm\beta,r}}{\prod}\mathfrak{D}_p(\beta_s)_j \underset{s\in\mathcal{P}_{\bm\beta,r}}{\prod}(\mathfrak{D}_p(\beta_s)+1)_j}\bmod p\right)(\mathcal{Q}_{\bm\alpha,\bm\beta}(r)\bmod p).$$

 \end{enumerate}
 
\end{lemm}

\begin{proof} We first suppose that $r=0$. Then $\mathcal{P}_{\bm\alpha,r}=\emptyset=\mathcal{P}_{\bm\beta,r}$. So, $\tau=1$ and, for all $j\in\mathbb{Z}_{\geq0}$, $\lambda_j=1$. Also, it is clear that $\theta=1$. Therefore, (1) and (2) are satisfied and (3), (4), and (5) follows immediately from Lemma~\ref{lemm_lambda_0}. We now suppose that $r>0$.

(1)

\textbullet\quad We prove that $v_p(\tau)=0$. If $s\in\mathcal{P}_{\bm\alpha,r}$ then the $p$-adic valuation of $$\underset{t\neq p\mathfrak{D}_p(\alpha_s)-\alpha_s}{\prod\limits_{t=0}\limits^{r-1}(\alpha_s+t)}$$ is zero because $k=p\mathfrak{D}_p(\alpha_s)-\alpha_s$ is the unique element in $\{0,1,\ldots, p-1\}$ such that $\alpha_s+k\in p\mathbb{Z}_{(p)}$ and, by assumption, $0<r<p$. Similarly, if $s\in\mathcal{P}_{\bm\beta,r}$ then the $p$-adic valuation of $$\underset{t\neq p\mathfrak{D}_p(\beta_s)-\beta_s}{\prod\limits_{t=0}\limits^{r-1}(\beta_s+t)}$$ is zero. Therefore,  $v_p(\tau)=0$.

\textbullet\quad We prove that $v_p(\theta)=0$. It is clear that if $s\in\mathcal{C}_{\bm\alpha,r}$ then  the $p$-adic valuation of $(\alpha_s)_{r}$ is zero. Likewise, if $s\in\mathcal{C}_{\bm\beta,r}$ then the $p$-adic valuation of $(\beta_s)_{r}$ is zero. Thus, the $p$-adic valuation of $\theta$ is zero.
 
\textbullet\quad Finally, $v_p(\lambda_0)=0$ because by assumption, $\mathfrak{D}_p(\bm\alpha)$, $\mathfrak{D}_p(\bm\beta)$ belong to $(\mathbb{Z}_{(p)}^*)^n$.

(2).  It is clear that $\mathcal{Q}_{\bm\alpha,\bm\beta}(r)=p^{\#\mathcal{P}_{\bm\alpha,r}-\#\mathcal{P}_{\bm\beta,r}}\lambda_0\tau\theta$. From (1), we know that $v_p(\tau)=v_p(\lambda_0)=v_p(\theta)=0$ and by asumption, $v_p(\mathcal{Q}_{\bm\alpha,\bm\beta}(r))=0$. Thus, $v_p(p^{\#\mathcal{P}_{\bm\alpha,r}-\#\mathcal{P}_{\bm\beta,r}})=0$. Whence, $\#\mathcal{P}_{\bm\alpha,r}=\#\mathcal{P}_{\bm\beta,r}$. So, $\mathcal{Q}_{\bm\alpha,\bm\beta}(r)=\lambda_0\tau\theta.$

(3). Let $j$ be a nonnegative integer. The following equality is straightforward 
\begin{equation*}
\mathcal{Q}_{\bm\alpha,\bm\beta}(jp+r)=\mathcal{Q}_{\bm\alpha,\bm\beta}(jp)\mathcal{Q}_{\bm\alpha+\textbf{jp},\bm\beta+\textbf{jp}}(r),
\end{equation*}
where $$\mathcal{Q}_{\bm\alpha+\textbf{jp},\bm\beta+\textbf{jp}}(r)=\frac{(\alpha_1+jp)_{r}\cdots(\alpha_n+jp)_{r}}{(\beta_1+jp)_{r}\cdots(\beta{}_n+jp)_{r}}.$$
Clearly, we also have
\begin{equation*}
\mathcal{Q}_{\bm\alpha+\textbf{jp},\bm\beta+\textbf{jp}}(r)=\frac{\prod_{s\in\mathcal{P}_{\bm\alpha,r}}(\alpha_s+jp)_{r}\prod_{s\in\mathcal{C}_{\bm\alpha,r}}(\alpha_s+jp)_{r}}{\prod_{s\in\mathcal{P}_{\bm\beta,r}}(\beta_s+jp)_{r}\prod_{s\in\mathcal{C}_{\bm\beta,r}}(\beta_s+jp)_{r}}.
\end{equation*}
By (2), we have $\#\mathcal{P}_{\bm\alpha,r}=\#\mathcal{P}_{\bm\beta,r}$ and thus, $$\mathcal{Q}_{\bm\alpha+\textbf{jp},\bm\beta+\textbf{jp}}(r)=\lambda_j\cdot\xi,\text{ where }\quad\xi=\frac{\underset{s\in\mathcal{P}_{\bm\alpha,r}}{\prod}\left(\underset{t\neq p\mathfrak{D}_p(\alpha_s)-\alpha_s}{\prod\limits_{t=0}\limits^{r-1}(\alpha_s+jp+t)}\right)\underset{s\in\mathcal{C}_{\bm\alpha,r}}{\prod}(\alpha_s+jp)_{r}}{\underset{s\in\mathcal{P}_{\bm\beta,r}}{\prod}\left(\underset{t\neq p\mathfrak{D}_p(\beta_s)-\beta_s}{\prod\limits_{t=0}\limits^{r-1}(\beta_s+jp+t)}\right)\underset{s\in\mathcal{C}_{\bm\beta,r}}{\prod}(\beta_s+jp)_{r}}.$$

By Lemma~\ref{lemm_lambda_0}, we know that $\mathcal{Q}_{\bm\alpha,\bm\beta}(jp)=\mathcal{Q}_{\mathfrak{D}_p(\bm\alpha),\mathfrak{D}_p(\bm\beta)}(j)\omega$, where $\omega\in\mathbb{Z}_{(p)}^*$ and $\omega\equiv1\bmod p$. Whence,
\begin{equation*}\label{eq_jj}
\mathcal{Q}_{\bm\alpha,\bm\beta}(jp+r)=\mathcal{Q}_{\mathfrak{D}_p(\bm\alpha),\mathfrak{D}_p(\bm\beta)}(j)\cdot\omega\cdot\lambda_j\cdot\xi.
\end{equation*}
We put $\nu=\omega\xi$. We now prove that $\nu\in\mathbb{Z}_{(p)}^*$ and that $\nu\bmod p=(\tau\bmod p)(\theta\bmod p)$.
Since $r<p$, it is clear that
\begin{equation}\label{eq_congruences}
\underset{t\neq p\mathfrak{D}_p(\alpha_s)-\alpha_s}{\prod\limits_{t=0}\limits^{r-1}}(\alpha_s+jp+t)\bmod p\equiv\underset{t\neq p\mathfrak{D}_p(\alpha_s)-\alpha_s}{\prod\limits_{t=0}\limits^{r-1}}(\alpha_s+t)\bmod p\neq0
\end{equation}
and that
\begin{equation}\label{eq_congruences1}
\underset{t\neq p\mathfrak{D}_p(\beta_s)-\beta_s}{\prod\limits_{t=0}\limits^{r-1}}(\beta_s+jp+t)\bmod p\equiv\underset{t\neq p\mathfrak{D}_p(\beta_s)-\beta_s}{\prod\limits_{t=0}\limits^{r-1}}(\beta_s+t)\bmod p\neq0.
\end{equation}
So, it follows from Equations \eqref{eq_congruences} and \eqref{eq_congruences1} that 
\begin{equation*}\label{eq_t}
\underset{s\in\mathcal{P}_{\bm\alpha,r}}{\prod}\left(\underset{t\neq p\mathfrak{D}_p(\alpha_s)-\alpha_s}{\prod\limits_{t=0}\limits^{r-1}}(\alpha_s+jp+t)\right)\bigg/\underset{s\in\mathcal{P}_{\bm\beta,r}}{\prod}\left(\underset{t\neq p\mathfrak{D}_p(\beta_s)-\beta_s}{\prod\limits_{t=0}\limits^{r-1}}(\beta_s+jp+t)\right)\equiv\tau\bmod p.
\end{equation*}
Furthermore, it is not hard to see that, $s\in\mathcal{P}_{\alpha,r}$ if and only if $(\alpha_s+jp)_{r}\in p\mathbb{Z}_{(p)}$ and that, $s\in\mathcal{P}_{\bm\beta,r}$ if and only if $(\beta_s+jp)_{r}\in p\mathbb{Z}_{(p)}$. For this reason, the $p$-adic valuation of $\underset{s\in\mathcal{C}_{\alpha,r}}{\prod}(\alpha_s+jp)_{r}$ and $\underset{s\in\mathcal{C}_{\bm\beta,r}}{\prod}(\beta_s+jp)_{r}$ is zero. So

\begin{equation*}\label{eq_jp}
\underset{s\in\mathcal{C}_{\bm\alpha,r}}{\prod}(\alpha_s+jp)_{r}\bigg/\underset{s\in\mathcal{C}_{\bm\beta,r}}{\prod}(\beta_s+jp)_{r}\equiv\theta\bmod p.
\end{equation*}

Consequently, $\xi\in\mathbb{Z}_{(p)}^*$ and $\xi\bmod p=(\tau\bmod p)(\theta\bmod p)$. Finally, we know that $\omega\in\mathbb{Z}_{(p)}^{*}$ and that $\omega\bmod p=1$. So,  $\nu\in\mathbb{Z}_{(p)}^*$ and $\nu\bmod p=(\tau\bmod p)(\theta\bmod p)$. 

(4). Let $j\geq1$ be an integer. From (3), we have  $\mathcal{Q}_{\bm\alpha,\bm\beta}(jp+r)=\mathcal{Q}_{\mathfrak{D}_p(\bm\alpha),\mathfrak{D}_p(\bm\beta),p}(j)\lambda_j\nu$, where $\nu\in\mathbb{Z}_{(p)}^*$. So $v_p(\mathcal{Q}_{\bm\alpha,\bm\beta}(jp+r))=v_p(\mathcal{Q}_{\mathfrak{D}_p(\bm\alpha),\mathfrak{D}_p(\bm\beta),p}(j)\lambda_j)$.
But, by assumption, we know that $v_p(\mathcal{Q}_{\bm\alpha,\bm\beta}(jp+r))\geq0$. Whence, $v_p(\mathcal{Q}_{\mathfrak{D}_p(\bm\alpha),\mathfrak{D}_p(\bm\beta),p}(j)\lambda_j)\geq0.$

Now, it is clear that \begin{equation}\label{eq_jjj}
\lambda_0\left(\frac{\underset{s\in\mathcal{C}_{\bm\alpha,r}}{\prod}\mathfrak{D}_p(\alpha_s)_j {\underset{s\in\mathcal{P}_{\bm\alpha,r}}{\prod}(\mathfrak{D}_p(\alpha_s)+1)_j}}{\underset{s\in\mathcal{C}_{\bm\beta,r}}{\prod}\mathfrak{D}_p(\beta_s)_j \underset{s\in\mathcal{P}_{\bm\beta,r}}{\prod}(\mathfrak{D}_p(\beta_s)+1)_j}\right)=\mathcal{Q}_{\mathfrak{D}_p(\bm\alpha),\mathfrak{D}_p(\bm\beta)}(j)\lambda_j.
\end{equation}
By (1), we know that $v_p(\lambda_0)=0$. So, $$v_p\left(\frac{\underset{s\in\mathcal{C}_{\bm\alpha,r}}{\prod}\mathfrak{D}_p(\alpha_s)_j {\underset{s\in\mathcal{P}_{\bm\alpha,r}}{\prod}(\mathfrak{D}_p(\alpha_s)+1)_j}}{\underset{s\in\mathcal{C}_{\bm\beta,r}}{\prod}\mathfrak{D}_p(\beta_s)_j \underset{s\in\mathcal{P}_{\bm\beta,r}}{\prod}(\mathfrak{D}_p(\beta_s)+1)_j}\right)=v_p(\mathcal{Q}_{\mathfrak{D}_p(\bm\alpha),\mathfrak{D}_p(\bm\beta),p}(j)\lambda_j)\geq0.$$

(5). Let $j\geq1$ be an integer. From (3) and (4) we get $$\mathcal{Q}_{\bm\alpha,\bm\beta}(jp+r)\bmod p=(\mathcal{Q}_{\mathfrak{D}_p(\bm\alpha),\mathfrak{D}_p(\bm\beta),p}(j)\lambda_j \bmod p)(\tau\bmod p)(\theta\bmod p).$$
So, from Equation~\eqref{eq_jjj}, we get 
\begin{align*}
&\mathcal{Q}_{\bm\alpha,\bm\beta}(jp+r)\bmod p=\\
&\left(\frac{\underset{s\in\mathcal{C}_{\bm\alpha,r}}{\prod}\mathfrak{D}_p(\alpha_s)_j {\underset{s\in\mathcal{P}_{\bm\alpha,r}}{\prod}(\mathfrak{D}_p(\alpha_s)+1)_j}}{\underset{s\in\mathcal{C}_{\bm\beta,r}}{\prod}\mathfrak{D}_p(\beta_s)_j \underset{s\in\mathcal{P}_{\bm\beta,r}}{\prod}(\mathfrak{D}_p(\beta_s)+1)_j}\bmod p\right)(\lambda_0\bmod p)(\tau\bmod p)(\theta\bmod p).
\end{align*}
From (2), we conclude that $\mathcal{Q}_{\bm\alpha,\bm\beta}(r)\bmod p= (\lambda_0\bmod p)(\tau\bmod p)(\theta\bmod p)$. Consequently, $$\mathcal{Q}_{\bm\alpha,\bm\beta}(jp+r)\bmod p=\left(\frac{\underset{s\in\mathcal{C}_{\bm\alpha,r}}{\prod}\mathfrak{D}_p(\alpha_s)_j {\underset{s\in\mathcal{P}_{\bm\alpha,r}}{\prod}(\mathfrak{D}_p(\alpha_s)+1)_j}}{\underset{s\in\mathcal{C}_{\bm\beta,r}}{\prod}\mathfrak{D}_p(\beta_s)_j \underset{s\in\mathcal{P}_{\bm\beta,r}}{\prod}(\mathfrak{D}_p(\beta_s)+1)_j}\bmod p\right)(\mathcal{Q}_{\bm\alpha,\bm\beta}(r)\bmod p).$$

\end{proof}

\subsection{Proof of Lemma~\ref{lemm_first_step}}
 Let us write $E_{\bm\alpha,\bm\beta,p}=\{e_0, e_1,\ldots, e_k\}$, where $e_i<e_{i+1}$ for all $i\in\{0,\ldots,k\}$ and $e_0=0$. We set $e_{k+1}=p$ and we also write $S_{\bm\alpha,\bm\beta,p}=\{0,r_1,\ldots, r_t\}$. Recall that $f(z)$ is the power series $\sum_{j\geq0}\mathcal{Q}_{\bm\alpha,\bm\beta}(j)z^j\in\mathbb{Z}_{(p)}[[z]]$. 
By hypotheses, we know that $\bm\alpha$ and $\bm\beta$ belong to $\mathbb{Z}_{(p)}^n$. So, by Lemma~\ref{lemm_cartier_operator}, we have $$f(z)=P_0(z)\Lambda_0(f)^p+P_{r_1}(z)\Lambda_{r_1}(f)^p+\cdots+P_{r_t}(z)\Lambda_{r_t}(f)^p\bmod p,$$
where, for all $r_i\in S_{\bm\alpha,\bm\beta,p}$, $P_{r_i}(z)=\sum_{s=r_i}^{e_j-1}\left(\frac{\mathcal{Q}_{\bm\alpha,\bm\beta}(s)}{\mathcal{Q}_{\bm\alpha,\bm\beta}(r_i)}\bmod p\right)z^s$ with $e_{j-1}=r_i$. By definition, $$\Lambda_{r_i}(f)=\sum_{j\geq0}\mathcal{Q}_{\bm\alpha,\bm\beta}(jp+r_i)z^j.$$ 

Now, by assumption, we know that $\mathfrak{D}_p(\bm\alpha)$, $\mathfrak{D}_p(\bm\beta)$ belong to $(\mathbb{Z}_{(p)}^*)^n$. Further, for all integers $j\geq0$, $v_p(\mathcal{Q}_{\bm\alpha,\bm\beta}(j))\geq0$ because $f(z)\in\mathbb{Z}_{(p)}[[z]]$ and, by definition, $v_p(\mathcal{Q}_{\bm\alpha,\bm\beta}(r_i))=0$ for all $r_i\in S_{\bm\alpha,\bm\beta,p}$. Therefore, by (4) of Lemma~\ref{lemm_ultra_tec}, we conclude that, for all $r_i\in S_{\bm\alpha,\bm\beta,p}$,  $$f_{1,r_i}(z)={}_nF_{n-1}(\bm\alpha_{1,r_i},\bm\beta_{1,r_i};z)=\sum_{m\geq0}\left(\frac{\underset{s\in\mathcal{C}_{\bm\alpha,r_i}}{\prod}\mathfrak{D}_p(\alpha_s)_m {\underset{s\in\mathcal{P}_{\bm\alpha,r_i}}{\prod}(\mathfrak{D}_p(\alpha_s)+1)_m}}{\underset{s\in\mathcal{C}_{\bm\beta,r_i}}{\prod}\mathfrak{D}_p(\beta_s)_m \underset{s\in\mathcal{P}_{\bm\beta,r_i}}{\prod}(\mathfrak{D}_p(\beta_s)+1)_m}\right)z^m$$
belongs to $1+z\mathbb{Z}_{(p)}[[z]].$ Furthermore, we deduce from (5) of Lemma~\ref{lemm_ultra_tec} that, for all $i\in\{0,\ldots,t\}$, 
$$\Lambda_{r_i}(f)\bmod p=\mathcal{Q}_{\bm\alpha,\bm\beta}(r_i)\bmod p\cdot f_{1,r_i}(z)\bmod p.$$
Therefore, $$f(z)\equiv Q_0(z)f_{1,0}^p+Q_{r_1}f_{1,r_1}(z)^p+\cdots+Q_{r_t}(z)f_{1,r_t}(z)^p\bmod p,$$
where $Q_{r_i}(z)=\mathcal{Q}_{\bm\alpha,\bm\beta}(r_i)P_{r_i}(z).$

$\hfill\square$

\section{Constructing the polynomial $P_p(Y)$}\label{sec_construction}
In this section we show how to obtain the polynomial $P_p(Y)$.  Let $\bm\alpha=(\alpha_1,\ldots,\alpha_n)$, $\bm\beta=(\beta_1,\ldots,\beta_{n-1},1)$ be in $(\mathbb{Q}\cap(0,1])^n$ and let $p$ be a prime number such that $p>2d_{\bm\alpha,\bm\beta}$ and $f(z):={}_nF_{n-1}(\bm\alpha,\bm\beta;z)$ belongs to $\mathbb{Z}_{(p)}[[z]]$ and let $l$ be the order of $p$ in $(\mathbb{Z}/d_{\bm\alpha,\bm\beta}\mathbb{Z})^*$. %For every $r\in S_{\mathfrak{D}_p^{l-1}(\bm\alpha),\mathfrak{D}_p^{l-1}(\bm\beta),p}$, we set 
%\begin{equation}\label{eq_hyp_f_i}
%f_r=\sum_{m\geq0}\left(\frac{\underset{s\in\mathcal{C}_{\mathfrak{D}_p^{l-1}(\bm\alpha),r}}{\prod}(\alpha_s)_m {\underset{s\in\mathcal{P}_{\mathfrak{D}_p^{l-1}(\bm\alpha),r}}{\prod}(\alpha_s+1)_m}}{\underset{s\in\mathcal{C}_{\mathfrak{D}_p^{l-1}(\bm\beta),r}}{\prod}(\beta_s)_m \underset{s\in\mathcal{P}_{\mathfrak{D}_p^{l-1}(\bm\beta),r}}{\prod}(\beta_s+1)_m}\right)z^m.
%\end{equation}
As $p>2d_{\bm\alpha,\bm\beta}$ then $\bm\alpha$, $\bm\beta$ belong to $(\mathbb{Z}_{(p)}^*)^n$ and, by Remark~\ref{rema_p_l}, $(\bm\alpha,\bm\beta)$ satisfies the $\textbf{P}_{p,l}$ property.  Then, by Proposition~\ref{prop_third_step}, for every $r\in S_{\mathfrak{D}_p^{l-1}(\bm\alpha),\mathfrak{D}_p^{l-1}(\bm\beta),p}$, $f_{l,r}\in1+z\mathbb{Z}_{(p)}[[z]]$ and 
\begin{equation}\label{eq_sys_1}
f_{l,r}\equiv\sum_{j\in S_{\mathfrak{D}_p^{l-1}(\bm\alpha),\mathfrak{D}_p^{l-1}(\bm\beta),p}}Q_{r,j}(z)f_{l,j}^{p^l}\bmod p,
\end{equation}
where, for every $j\in S_{\mathfrak{D}_p^{l-1}(\bm\alpha),\mathfrak{D}_p^{l-1}(\bm\beta),p}$, $Q_{i,j}(z)$ belongs to $\mathbb{Z}_{(p)}[z]$ and has degree less than $p^l$.  By following the proof of Theorem~\ref{theo_main}, the polynomial $P_p(Y)$ results from applying Proposition~\ref{lemm_system} to the system~\eqref{eq_sys_1}. Thus $P_p(Y)$ is obtained by subsequent elimination of the series $f_{l,j}^{p^l}$ for $j\in S_{\mathfrak{D}_p^{l-1}(\bm\alpha),\mathfrak{D}_p^{l-1}(\bm\beta),p}\setminus\{0\}$\footnote{Remember that, for every integer $a\geq1$, $0\in S_{\mathfrak{D}_p^{a-1}(\bm\alpha),\mathfrak{D}_p^{a-1}(\bm\beta),p}$. Note that $f_{l,0}=f.$}. It follows from the proof of Proposition~\ref{lemm_system} that this subsequent elimination is explicit once the polynomials $Q_{r,j}$ are known. Lemma~\ref{lemm_explicit_poly} gives a formula for each polynomial $Q_{r,j}$. This formula is given recursively and is constructed from the polynomials given by the conclusion of Lemma~\ref{lemm_first_step}. In order to state the lemma, we introduce the following polynomials. Let $r$ be in $S_{\mathfrak{D}_p^{l-1}(\bm\alpha),\mathfrak{D}_p^{l-1}(\bm\beta),p}$ and let us consider the vectors $\bm\alpha_{l,r}=\bm\omega=(\omega_1,\ldots,\omega_n)$ and $\bm\beta_{l,r}=\bm\eta=(\eta_1,\ldots,\eta_n).$\footnote{For every integer $a\geq1$ and $r\in\{0,\ldots, p-1\}$, the definition of the vectors $\bm\alpha_{a,r}$, $\bm\beta_{a,r}$ was given at the beginning of Section~\ref{sec_proof_theo_main}.}.  For every $j\in S_{\bm\omega,\bm\eta,p}$, we set $$T_{r,j}(z)=\sum_{s=j}^{j'-1}(\mathcal{Q}_{\bm\omega,\bm\eta}(s)\bmod p)z^s,$$
where $j'$ is defined as follows.  If $j\neq\max E_{\bm\omega,\bm\eta,p}$ then $j'$ is the element in $E_{\bm\omega,\bm\eta,p}$ such that $(j,j')\cap E_{\bm\omega,\bm\eta,p}=\emptyset$ or otherwise, $j'=p$. 

Let $k$ be in $\{2,\ldots, l\}$.  For every $j\in S_{\mathfrak{D}_p^{k-1}(\bm\omega),{\mathfrak{D}_p^{k-1}(\bm\eta),p}}$ and $b\in S_{\mathfrak{D}_p^{k-2}(\bm\omega),{\mathfrak{D}_p^{k-2}(\bm\eta),p}}$ , we set $$T^{(k-1,b)}_{r,j}=\sum_{s=\tau(j)}^{j'-1}(\mathcal{Q}_{\bm\omega_{k-1,b},\bm\eta_{k-1,b}}(s)\bmod p)z^s,$$
where $\tau:S_{\mathfrak{D}_p^{k-1}(\bm\omega),{\mathfrak{D}_p^{k-1}(\bm\eta),p}}\rightarrow S_{\bm\omega_{k-1,b},\bm\eta_{k-1,b},p}$ is the function given by Lemma~\ref{lemm_aux} and $j'$ is defined as follows. If $\tau(j)\neq\max E_{\bm\omega_{k-1,b},\bm\eta_{k-1,b},p}$ then $j'$ is the element in $E_{\bm\omega_{k-1,b},\bm\eta_{k-1,b},p}$ such that $(\tau(j),j')\cap E_{\bm\omega_{k-1,b},\bm\eta_{k-1,b},p}=\emptyset$ or otherwise $j'=p$.

We are now ready to state Lemma~\ref{lemm_explicit_poly}.

\begin{lemma}\label{lemm_explicit_poly}
Let the assumptions be as in Proposition~\ref{prop_third_step}. If $l\geq2$ then, for every $r, j\in S_{\mathfrak{D}_p^{l-1}(\bm\alpha),\mathfrak{D}_p^{l-1}(\bm\beta),p}$, 
\begin{align*}
&Q_{r,j}=\\
&\left(\sum_{j_{l-1}\in S_{\mathfrak{D}^{l-2}_p(\bm\omega),\mathfrak{D}^{l-2}_p(\bm\eta),p}}\cdots\sum_{j_1\in S_{\bm\omega,\bm\eta,p}}T_{r,j_1}(T^{(1,j_1)}_{r,j_2})^p\cdots (T^{(l-2,j_{l-2})}_{r,j_{l-1}})^{p^{l-2}}\right) (T^{(l-1,j_{l-1})}_{r,j})^{p^{l-1}}.
\end{align*}
If $l=1$ then, for every $r, j\in S_{\bm\alpha,\bm\beta,p}$, $$Q_{r,j}=T_{r,\tau(j)},$$ where $\tau: S_{\bm\alpha,\bm\beta}\rightarrow S_{\bm\omega,\bm\eta,p}$ is the bijective map given by A) of Lemma~\ref{lemm_aux}.

\end{lemma}

\begin{proof}
Let $r$ be in $S_{\mathfrak{D}_p^{l-1}(\bm\alpha),\mathfrak{D}_p^{l-1}(\bm\beta),p}$. Let $F$ be the hypergeometric series with parameters $\bm\alpha_{l,r}=\bm\omega=(\omega_1,\ldots,\omega_n)$ and $\bm\beta_{l,r}=\bm\eta=(\eta_1,\ldots,\eta_n).$ Then $F=f_{l,r}$. For every $0\leq a< l$ and $j\in S_{\mathfrak{D}^a_p(\bm\omega),\mathfrak{D}^a_p(\bm\eta),p}$, we put $F_{a+1,j}={}_nF_{n-1}(\bm\omega_{a+1,j},\bm\eta_{a+1,j};z)$. That is, $$F_{a+1,j}=\sum_{m\geq0}\left(\frac{\underset{s\in\mathcal{C}_{\mathfrak{D}_p^a(\bm\omega),j}}{\prod}\mathfrak{D}_p^{a+1}(\omega_s)_m {\underset{s\in\mathcal{P}_{\mathfrak{D}_p^a(\bm\omega),j}}{\prod}(\mathfrak{D}_p^{a+1}(\omega_s)+1)_m}}{\underset{s\in\mathcal{C}_{\mathfrak{D}_p^a(\bm\eta),j}}{\prod}\mathfrak{D}_p^{a+1}(\eta_s)_m \underset{s\in\mathcal{P}_{\mathfrak{D}_p^a(\bm\eta),j}}{\prod}(\mathfrak{D}_p^{a+1}(\eta_s)+1)_m}\right)z^m.$$
As $(\bm\alpha,\bm\beta)$ satisfies the $\textbf{P}_{p,l}$ property then, by (2) of Remark~\ref{rema_p_l_i}, $(\bm\omega,\bm\eta)$ satisfies the $\textbf{P}_{p,l}$ property and $l$ is the order of $p$ in $(\mathbb{Z}/d_{\bm\omega,\bm\eta}\mathbb{Z})^*$. Thus, by Lemma~\ref{lemm_second_step}, $F_{a+1,j}\in 1+z\mathbb{Z}_{(p)}[[z]]$ for all $0\leq a< l$ and $j\in S_{\mathfrak{D}^a_p(\bm\omega),\mathfrak{D}^a_p(\bm\eta),p}$.
 %By (1) of Remark~\ref{rema_p_l_i}, we get that, for all $a $\mathfrak{D}_p^{l-1}(\bm\omega)=\mathfrak{D}_p^{l-1}(\bm\alpha)$ and $\mathfrak{D}_p^{l-1}(\bm\eta)=\mathfrak{D}_p^{l-1}(\bm\beta)$

 We first prove by induction on $k\in\{1,\ldots, l\}$ that 

 \begin{equation}\label{eq_hyp_F}
F=\sum_{j\in S_{\mathfrak{D}^{k-1}_p(\bm\omega),\mathfrak{D}^{k-1}_p(\bm\eta),p} } Q_{r,j}^{(k-1)}F_{k,j}^{p^k},
\end{equation}
where $$Q_{r,j}^{(k-1)}=\sum_{j_{k-1}\in S_{\mathfrak{D}^{k-2}_p(\bm\omega),\mathfrak{D}^{k-2}_p(\bm\eta),p}}\cdots\sum_{j_1\in S_{\bm\omega,\bm\eta,p}}T_{r,j_1}(T^{(1,j_1)}_{r,j_2})^p\cdots(T^{(k-1,j_{k-1})}_{r,j})^{p^{k-1}}$$

For $k=1$, according to Lemma~\ref{lemm_first_step}, we have $$F=\sum_{j\in S_{\bm\omega,\bm\eta,p}}T_{r,j}(z)F_{1,j}^p.$$
We now suppose that for some $k\in\{1,\ldots l-1\}$ Equality~\eqref{eq_hyp_F} holds. We are going to see that Equation~\ref{eq_hyp_F} also holds for $k+1$. Let $j$ be in $S_{\mathfrak{D}^{k-1}_p(\bm\omega),\mathfrak{D}^{k-1}_p(\bm\eta),p}$. By definition, $F_{k,j}$ is the hypergeometric series ${}_nF_{n-1}(\bm\omega_{k,j},\bm\eta_{k,j};z)$. Further, we know that $(\bm\omega,\bm\eta)$ satisfies the $\textbf{P}_{p,l}$ property and thus, by Remark~\ref{rema_d_p_h}, $(\bm\omega_{k,j},\bm\eta_{k,j})$ satisfies the $\textbf{P}_{p,l'}$ property, where $l'$ is the order $p$ in $(\mathbb{Z}/d_{\bm\omega_{k,j},\bm\eta_{k,j}}\mathbb{Z})^*$. So, by applying Lemma~\ref{lemm_first_step} to $F_{k,j}$, we get

$$F_{k,j}=\sum_{\gamma\in S_{\bm\omega_{k,j},\bm\eta_{k,j}}}Q_{\gamma}F_{\gamma}^p,$$
where $$F_{\gamma}=\sum_{m\geq0}\left(\frac{\underset{s\in\mathcal{C}_{\bm\omega_{k,j},\gamma}}{\prod}\mathfrak{D}_p(\omega_{s,k,j})_m {\underset{s\in\mathcal{P}_{\bm\omega_{k,j},\gamma}}{\prod}(\mathfrak{D}_p(\omega_{s,k,j})+1)_m}}{\underset{s\in\mathcal{C}_{\bm\eta_{k,j},\gamma}}{\prod}\mathfrak{D}_p(\eta_{s,k,j})_m \underset{s\in\mathcal{P}_{\bm\eta_{k,j},\gamma}}{\prod}(\mathfrak{D}_p(\eta_{s,k,j})+1)_m}\right)z^m\text{, } Q_{\gamma}=\sum_{s=\gamma}^{\gamma'-1}\mathcal{Q}_{\bm\omega_{k,j},\bm\eta_{k,j}}(s)z^s$$
and $\gamma'$ is defined as follows. If $\gamma\neq\max E_{\bm\omega_{k,j},\bm\eta_{k,j}}$ then $\gamma'$ is the element in $E_{\bm\omega_{k,j},\bm\eta_{k,j}}$ such that $(\gamma,\gamma')\cap E_{\bm\omega_{k,j},\bm\eta_{k,j}}=\emptyset$ or otherwise, $\gamma'=p$. 

We now prove that $F_{\gamma}= F_{k+1,\sigma(\gamma)}$, where $\sigma:S_{\bm\omega_{k,j},\bm\eta_{k,j},p}\rightarrow S_{\mathfrak{D}_p^k(\bm\omega),\mathfrak{D}_p^k(\bm\eta),p}$ is the function given by Lemma~\ref{lemm_aux}\footnote{Note that we can apply  Lemma~\ref{lemm_aux} to $(\bm\omega,\bm\eta)$ because $(\bm\omega,\bm\eta)$ satisfies the $\textbf{P}_{p,l}$ property.}. For this purpose, we first prove that, for all $s\in\{1,\ldots, n\}$, $\mathfrak{D}_p(\omega_{s,k,j})=\mathfrak{D}_p^{k+1}(\omega_s)$. By definition $\omega_{s,k,j}=\mathfrak{D}^k_p(\omega_s)$ if $s\in\mathcal{C}_{\mathfrak{D}_p^{k-1}(\bm\omega),j}$ or $\omega_{s,k,j}=\mathfrak{D}^k_p(\omega_s)+1$ if $s\in\mathcal{P}_{\mathfrak{D}_p^{k-1}(\bm\omega),j}$. It is clear that in the first case $\mathfrak{D}_p(\omega_{s,k,j})=\mathfrak{D}_p^{k+1}(\omega_s)$. Suppose now that  $\omega_{s,k,j}=\mathfrak{D}^k_p(\omega_s)+1$. Again, by definition $\omega_s=\alpha_{s,l,i}$ and thus $\omega_s=\alpha_s$ if  $s\in\mathcal{C}_{\mathfrak{D}_p^{l-1}(\bm\alpha),i}$ or $\omega_s=\alpha_s+1$ if $s\in\mathcal{P}_{\mathfrak{D}_p^{l-1}(\bm\alpha),i}$. By assumption, $\alpha_s\in\mathbb{Z}_{(p)}^*$ and thus, by (1) of Remark~\ref{rema_d_p_1}, $\mathfrak{D}_p(\omega_s)=\mathfrak{D}_p(\alpha_s)$. In addition, for all integers $1\leq r\leq l$, $\mathfrak{D}^r_p(\omega_s)\in\mathbb{Z}_{(p)}^{*}$ given that  $\mathfrak{D}_p(\omega_s)=\mathfrak{D}_p(\alpha_s)$ and $(\bm\alpha,\bm\beta)$ satisfies the $\textbf{P}_{p,l}$ property. Hence, according to (1) of Remark~\ref{rema_d_p_1} again, $\mathfrak{D}_p(\omega_{s,k,j})=\mathfrak{D}_p(\mathfrak{D}^k_p(\omega_s)+1)=\mathfrak{D}_p^{k+1}(\omega_s).$ In a similar way we show that,  for all $s\in\{1,\ldots, n\}$, $\mathfrak{D}_p(\eta_{s,k,j})=\mathfrak{D}_p^{k+1}(\eta_s)$. Now, by (ii) of Lemma~\ref{lemm_aux}, we have $\mathcal{C}_{\bm\omega_{k,j},\gamma}=\mathcal{C}_{\mathfrak{D}_p^k(\bm\omega),\sigma(\gamma)}$ and $\mathcal{P}_{\bm\omega_{k,j},\gamma}=\mathcal{P}_{\mathfrak{D}_p^k(\bm\omega),\sigma(\gamma)}$. Again,  by (ii) of Lemma~\ref{lemm_aux}, we have $\mathcal{C}_{\bm\eta_{k,j},\gamma}=\mathcal{C}_{\mathfrak{D}_p^k(\bm\eta),\sigma(\gamma)}$ and $\mathcal{P}_{\bm\eta_{k,j},\gamma}=\mathcal{P}_{\mathfrak{D}_p^k(\bm\eta),\sigma(\gamma)}$. Consequently, $F_{\gamma}= F_{k+1,\sigma(\gamma)}$. Finally, it is clear that $Q_{\gamma}=T_{r,\sigma(\gamma)}^{(k,j)}$ because $\tau(\sigma(\gamma)=\gamma$. Therefore,
$$F_{k,j}=\sum_{i\in S_{\mathfrak{D}^k_p(\bm\omega),\mathfrak{D}^k_p(\bm\eta),p}}T_{r,i}^{(k,j)}F_{k+1,i}^p.$$
 Thus, from induction hypothesis and the previous equality, we obtain
  \begin{equation*}
 \begin{split}
&F=\\
&\sum_{j\in S_{\mathfrak{D}^{k}_p(\bm\omega),\mathfrak{D}^{k}_p(\bm\eta),p} }\left(\sum_{j_{k}\in S_{\mathfrak{D}^{k-1}_p(\bm\omega),\mathfrak{D}^{k-1}_p(\bm\eta),p}}\cdots\sum_{j_1\in S_{\bm\omega,\bm\eta,p}}T_{r,j_1}(T^{(1,j_1)}_{r,j_2})^p\cdots(T^{(k,j_{k})}_{r,j})^{p^{k}}\right)F_{k+1,j}^{p^{k+1}},
\end{split}
\end{equation*}
which shows that Equation~\ref{eq_hyp_F} is true for $k+1$.  

So, by induction we conclude that Equation~\eqref{eq_hyp_F} holds for all $1\leq k\leq l$. 

Suppose that $l\geq2$. We will show that, for every $j\in S_{\mathfrak{D}^{l-1}_p(\bm\omega),\mathfrak{D}^{l-1}_p(\bm\eta),p}$, $F_{l,j}=f_{l,j}$. By (1) of Remark~\ref{rema_p_l_i}, we have $\mathfrak{D}_p^{l-1}(\bm\omega)=\mathfrak{D}_p^{l-1}(\bm\alpha)$ and $\mathfrak{D}_p^{l-1}(\bm\eta)=\mathfrak{D}_p^{l-1}(\bm\beta)$. Thus, $\mathcal{C}_{\mathfrak{D}_p^{l-1}(\bm\omega),j}=\mathcal{C}_{\mathfrak{D}_p^{l-1}(\bm\alpha),j}$, $\mathcal{P}_{\mathfrak{D}_p^{l-1}(\bm\omega),j}=\mathcal{P}_{\mathfrak{D}_p^{l-1}(\bm\alpha),j}$, $\mathcal{C}_{\mathfrak{D}_p^{l-1}(\bm\eta),j}=\mathcal{C}_{\mathfrak{D}_p^{l-1}(\bm\beta),j}$, and $\mathcal{P}_{\mathfrak{D}_p^{l-1}(\bm\eta),j}=\mathcal{P}_{\mathfrak{D}_p^{l-1}(\bm\beta),j}$. Furthermore, we also have $S_{\mathfrak{D}^{l-1}_p(\bm\omega),\mathfrak{D}^{l-1}_p(\bm\eta),p}=S_{\mathfrak{D}^{l-1}_p(\bm\alpha),\mathfrak{D}^{l-1}_p(\bm\beta),p}$.  By (1) of Remark~\ref{rema_p_l_i} again, we have $\mathfrak{D}_p^{l}(\bm\omega)=\mathfrak{D}_p^{l}(\bm\alpha)$ and $\mathfrak{D}_p^{l}(\bm\eta)=\mathfrak{D}_p^{l}(\bm\beta)$. As $p^{l}\equiv 1\bmod d_{\bm\alpha,\bm\beta}$ and, by assumption, $\bm\alpha$, $\bm\beta\in(\mathbb{Z}_{(p)}^*)^n$ then Lemma~\ref{cyclic} implies that $\mathfrak{D}_p^{l}(\bm\alpha)=\bm\alpha$ and $\mathfrak{D}_p^{l}(\bm\beta)=\bm\beta$. So,  $\mathfrak{D}_p^{l}(\bm\omega)=\bm\alpha$ and $\mathfrak{D}_p^{l}(\bm\eta)=\bm\beta$.  Therefore, for every $j\in S_{\mathfrak{D}^{l-1}_p(\bm\omega),\mathfrak{D}^{l-1}_p(\bm\eta),p}$, $F_{l,j}=f_{l,j}$. Consequently, it follows from Equation~\eqref{eq_hyp_F} that $$F=\sum_{j\in S_{\mathfrak{D}_p^{l-1}(\bm\alpha),\mathfrak{D}_p^{l-1}(\bm\beta),p}}Q_{r,j}(z)f_{l,j}^{p^l}\bmod p,$$
where 
\begin{align*}
&Q_{r,j}=\\
&\left(\sum_{j_{l-1}\in S_{\mathfrak{D}^{l-2}_p(\bm\omega),\mathfrak{D}^{l-2}_p(\bm\eta),p}}\cdots\sum_{j_1\in S_{\bm\omega,\bm\eta,p}}T_{r,j_1}(T^{(1,j_1)}_{r,j_2})^p\cdots (T^{(l-2,j_{l-2})}_{r,j_{l-1}})^{p^{l-2}}\right) (T^{(l-1,j_{l-1})}_{r,j})^{p^{l-1}}.
\end{align*}
This completes the case $l\geq2$ because $F=f_{l,r}$.

Suppose now that $l=1$. As $p\equiv 1\bmod d_{\bm\alpha,\bm\beta}$ and, by assumption, $\bm\alpha$, $\bm\beta\in(\mathbb{Z}_{(p)}^*)^n$ then Lemma~\ref{cyclic} implies that,  $\mathfrak{D}_p(\bm\alpha)=\bm\alpha$ and $\mathfrak{D}_p(\bm\beta)=\bm\beta$. By assumption again, $(\bm\alpha,\bm\beta)$ satisfies the $\textbf{P}_{p,1}$ property. Then, by Lemma~\ref{lemm_aux}, we have $\sigma:S_{\bm\alpha_{1,r},\bm\beta_{1,l},p}\rightarrow S_{\bm\alpha,\bm\beta,p}$. But definition, $\bm\alpha_{1,r}=\bm\omega$ and $\bm\beta_{1,r}=\bm\eta.$ So, $\sigma:S_{\bm\omega,\bm\eta,p}\rightarrow S_{\bm\alpha,\bm\beta,p}$. We are going to see that, for all $j\in S_{\bm\omega,\bm\eta,p}$, $F_{1,j}=f_{1,\sigma(j)}$. By B) of Lemma~\ref{lemm_aux}, we get $\mathcal{C}_{\bm\omega,j}=\mathcal{C}_{\bm\alpha,\sigma(j)}$,  $\mathcal{P}_{\bm\omega,j}=\mathcal{P}_{\bm\alpha,\sigma(j)}$,  $\mathcal{C}_{\bm\eta,j}=\mathcal{C}_{\bm\beta,\sigma(j)}$, and $\mathcal{P}_{\bm\eta,j}=\mathcal{C}_{\bm\alpha,\sigma(j)}$. By (1) of Remark~\ref{rema_p_l_i}, we have $\mathfrak{D}_p(\bm\omega)=\mathfrak{D}_p(\bm\alpha)$ and $\mathfrak{D}_p(\bm\eta)=\mathfrak{D}_p(\bm\beta)$. But we know that $\mathfrak{D}_p(\bm\alpha)=\bm\alpha$ and $\mathfrak{D}_p(\bm\beta)=\bm\beta$. So $\mathfrak{D}_p(\bm\omega)=\bm\alpha$ and $\mathfrak{D}_p(\bm\eta)=\bm\beta$. Consequently, for all $j\in S_{\bm\omega,\bm\eta,p}$, $F_{1,j}=f_{1,\sigma(j)}$. Since $\sigma$ is a bijective map, we deduce from Equation~\eqref{eq_hyp_F} that $$F=\sum_{j\in S_{\bm\alpha,\bm\beta,p}}Q_{r,j}(z)f_{1,j}^p,$$
where $Q_{r,j}=T_{r,\tau(j)}$ with $\tau$ the inverse of $\sigma$. This completes the case $l=1$ because $F=f_{1,r}$.
 \end{proof}

 As an application of Lemma~\ref{lemm_explicit_poly}, we will give a formula for each rational function appearing in Equations~\eqref{eq_ex_1} and \eqref{eq_ex_3}. 
 \begin{theo}\label{theo_cons}
Let $\bm\alpha=(\alpha_1,\ldots,\alpha_n)$, $\bm\beta=(\beta_1,\ldots,\beta_{n-1},1)$ be in $(\mathbb{Q}\cap(0,1])^n$ and let $p$ be a prime number such that $p>2d_{\bm\alpha,\bm\beta}$ and $f(z):={}_nF_{n-1}(\bm\alpha,\bm\beta;z)$ belongs to $\mathbb{Z}_{(p)}[[z]]$. Suppose that $\# S_{\bm\alpha,\bm\beta,p}=2$.  We write $S_{\bm\alpha,\bm\beta,p}=\{0,r\}$, $E_{\bm\alpha,\bm\beta,p}=\{e_0,e_1,\ldots,e_k\}$ with $e_0=0$ and $e_i<e_{i+1}$ for all $i\in\{0,\ldots, k\}$, $E_{\bm\alpha_{1,r},\bm\beta_{1,r},p}=\{e'_0,e'_1,\ldots,e'_m\}$ with $e'_0=0$ and $e'_{i}<e'_{i+1}$ for all $i\in\{0,\ldots, m\}$. Let $e_{k+1}$ and $e'_{m+1}$ be the prime number $p$. We put $r=e_{s-1}$ and $\tau(r)=e'_{h-1}$. If $p\equiv1\bmod d_{\bm\alpha,\bm\beta}$ then $$f(z)\equiv Q_1(z)f(z)^{p}+Q_2(z)f^{p^{2}}\bmod p,$$
where $$Q_1(z)=P_0+\frac{T_1^p}{P_1^{p-1}}\text{ and } Q_2(z)=P_1T_0^p-\frac{T_1^pP_0^p}{P_1^{p-1}},$$
with $$P_0(z)=\sum_{j=0}^{e_1-1}\mathcal{Q}_{\bm\alpha,\bm\beta}(j)z^j\text{, }\quad P_1(z)=\sum_{j=r_1}^{e_s-1}\mathcal{Q}_{\bm\alpha,\bm\beta}(j)z^j,$$
and  
$$T_0(z)=\sum_{j=0}^{e'_1-1}\mathcal{Q}_{\bm\alpha_{1,r},\bm\beta_{1,r}}(j)z^j\text{, }\quad T_1(z)=\sum_{j=\tau(r_1)}^{e'_h-1}\mathcal{Q}_{\bm\alpha_{1,r},\bm\beta_{1,r}}(j)z^jz^j.$$
\end{theo}

\begin{proof}
 It is clear that 1 is the order of $p$ in $(\mathbb{Z}/d_{\bm\alpha,\bm\beta}\mathbb{Z})^*$. It follows from Remark~\ref{rema_p_l}, that $(\bm\alpha,\bm\beta)$ satisfies de $\textbf{P}_{p,1}$ property because $p>2d_{\bm\alpha,\bm\beta}$. Further, $\bm\alpha$, $\bm\beta$ belong to $(\mathbb{Z}_{(p)}^*)^n$ because $p>2d_{\bm\alpha,\bm\beta}$ and $\bm\alpha$, $\bm\beta$ belong to $(0,1]^n$.  Note that $f$ is the hypergeometric series $f_{1,0}$ because Lemma~\ref{cyclic} implies $\mathfrak{D}_p(\bm\alpha)=\bm\alpha$ and $\mathfrak{D}_p(\bm\beta)=\bm\beta$. Then, by Lemma~\ref{lemm_explicit_poly}, we get
\begin{equation}\label{eq_8_1}
f(z)\equiv P_0(z)f(z)^p+P_1(z)f_{1,1}(z)^p\bmod p,
\end{equation} 
\begin{equation}\label{eq_8_2}
f_{1,1}(z)\equiv T_0(z)f(z)^p+T_1(z)f_{1,1}(z)^p\bmod p,
\end{equation}

where
$$f_{1,1}(z)=\sum_{m\geq0}\left(\frac{\underset{s\in\mathcal{C}_{\bm\alpha,r}}{\prod}\mathfrak{D}_p(\alpha_s)_m {\underset{s\in\mathcal{P}_{\bm\alpha,r}}{\prod}(\mathfrak{D}_p(\alpha_s)+1)_m}}{\underset{s\in\mathcal{C}_{\bm\beta,r}}{\prod}\mathfrak{D}_p(\beta_s)_m \underset{s\in\mathcal{P}_{\bm\beta,r}}{\prod}(\mathfrak{D}_p(\beta_s)+1)_m}\right)z^m\in1+z\mathbb{Z}_{(p)}[[z]].$$

From Equations \eqref{eq_8_1} and \eqref{eq_8_2}, we get
$$
f=\left(P_0+\frac{T_1^p}{P_1^{p-1}}\right)f^p+\left(P_1T_0^p-\frac{T_1^pP_0^p}{P_1^{p-1}}\right)f^{p^2}\bmod p.
$$
%we get, $$P_1f_1-T_1f\equiv(P_1T_0-T_1P_0)f^p\bmod p.$$
%As $r_1\in S_{\bm\alpha,\bm\beta,p}$ then $v_p(\mathcal{Q}_{\bm\alpha,\bm\beta}(r_1))=0$. Whence, $P_1\bmod p$ is not the polynomial zero. Thus, we get $$f_1\equiv\left(T_0-\frac{T_1P_0}{P_1}\right)f^p+\frac{T_1}{P_1}f.$$
%Substituting this last equality into \eqref{eq_8_1} we have 
%\begin{align*}
%f\equiv&P_0f^p+P_1\left(\left(T_0-\frac{T_1P_0}{P_1}\right)f^p+\frac{T_1}{P_1}f\right)^p\bmod p\\
%\equiv&P_0f^p+P_1\left(\left(T_0^p-\left(\frac{T_1P_0}{P_1}\right)^p\right)f^{p^2}+\left(\frac{T_1}{P_1}\right)^pf^p\right)\bmod p\\
%\equiv&\left(P_0+\frac{T_1^p}{P_1^{p-1}}\right)f^p+\left(P_1T_0^p-\frac{T_1^pP_0^p}{P_1^{p-1}}\right)f^{p^2}\bmod p.
%\end{align*}
\end{proof}

As a corollary of Theorem~\ref{theo_cons} we have 

\begin{coro}\label{coro_explcit}

Let $\mathfrak{f}(z):={}_2F_1(\bm\alpha,\bm\beta;z)$ with $\bm\alpha=(\frac{1}{3},\frac{1}{2})$ and $\bm\beta=(\frac{5}{12},1)$, and let $p$ be a prime number such that $p=1+12k$ and $p>24$. Then, $\mathfrak{f}(z)\in\mathbb{Z}_{(p)}[[z]]$ and  $$\mathfrak{f}\equiv\left(P_0+\frac{T_1^p}{P_1^{p-1}}\right)\mathfrak{f}^p+\left(P_1T_0^p-\frac{T_1^pP_0^p}{P_1^{p-1}}\right)\mathfrak{f}^{p^2}\bmod p,$$
where $$P_{0}(z)=\sum_{j=0}^{5k}\frac{(1/3)_j(1/2)_j}{(5/12)_j(1)_j}z^j\text{ , } P_{1}=\sum_{j=1+5k}^{p-1}\frac{(1/3)_j(1/2)_j}{(5/12)_j(1)_j}z^j$$
and $$T_{0}(z)=\sum_{j=0}^{5k-1}\frac{(1/3+1)_j(1/2)_j}{(5/12+1)_j(1)_j}z^j\text{, }T_{1}(z)=\sum_{j=5k}^{p-1}\frac{(1/3+1)_j(1/2)_j}{(5/12+1)_j(1)_j}z^j.$$
\end{coro}

\begin{proof}
Note that $E_{\bm\alpha,\bm\beta,p}=\{0,1+5k\}$ and we have proved in Example \ref{exam_1} that $S_{\bm\alpha,\bm\beta,p}=\{0,1+5k\}$. From the calculations made in Example~\ref{exam_1} it follows that $\bm\alpha_{1,1+5k}=(1/3+1,1)$ and $\bm\beta_{1,1+5k}=(5/12+1,1)$. Thus $E_{\bm\alpha_{1,1+5k},\bm\beta_{1,1+5k}}=\{0,5k\}$ and $\tau(1+5k)=5k$, where $\tau:S_{\mathfrak{D}_p(\bm\alpha),\mathfrak{D}_p(\bm\beta),p}\rightarrow S_{\bm\alpha_{1,1+5k},\bm\beta_{1,1+5k},p}$ is the function given by Lemma~\ref{lemm_aux}. Since $\mathfrak{D}_p(\bm\alpha)=\bm\alpha$ and $\mathfrak{D}_p(\bm\beta)=\bm\beta$, $S_{\mathfrak{D}_p(\bm\alpha),\mathfrak{D}_p(\bm\beta),p}=\{0,1+5k\}$. Thus, from Theorem~\ref{theo_cons}, we get $$\mathfrak{f}\equiv\left(P_0+\frac{T_1^p}{P_1^{p-1}}\right)\mathfrak{f}^p+\left(P_1T_0^p-\frac{T_1^pP_0^p}{P_1^{p-1}}\right)\mathfrak{f}^{p^2}\bmod p.$$
\end{proof}
In the next theorem we give an explicit formula for each rational function appearing in Equation~\eqref{eq_ex_3}. 

\begin{theo}\label{theo_explicit}
Let  $\mathfrak{g}(z)={}_3F_2(\bm\alpha,\bm\beta;z)$ with $\bm\alpha=(\frac{1}{9},\frac{4}{9},\frac{5}{9})$ and $\bm\beta=(\frac{1}{3},1,1)$, and let $p$ be a prime number such that $p=8+9k_p$ and $p>18$. Then, $\mathfrak{g}(z)\in\mathbb{Z}_{(p)}[[z]]$ and $$\mathfrak{g}\equiv\left(P_{0,0}P_{1,0}^p+P_{0,0}P_{1,1}^p\left(\frac{R_{1,0}}{P_{0,0}}\right)^{p^2}\right)\mathfrak{g}^{p^2}\bmod p,$$
where $$P_{0,0}(z)=\sum_{s=0}^{5+6k_p}\frac{(1/9)_s(4/9)_s(5/9)_s}{(1/3)_s(1)_s^2}z^s\text{, } P_{1,0}=\sum_{s=0}^{2+3k_p}\frac{(8/9)_s(5/9)_s(4/9)_s}{(2/3)_s(1)_s^2}z^s$$ 
and $$P_{1,1}=\sum_{s=3+3k_p}^{p-1}\frac{(8/9)_s(5/9)_s(4/9)_s}{(2/3)_s(1)_s^2}z^s\text{, }R_{1,0}(z)=\sum_{s=0}^{4+6k_p}\frac{(1/9+1)_s(4/9)_s(5/9)_s}{(1/3+1)_s(1)_s^2}z^s.$$
\end{theo}

\begin{proof}
%We first prove that  $\mathcal{S}_{\bm\alpha,\bm\beta,p}=\{0\}$.  We have $E_{\bm\alpha,\bm\beta,p}=\{0,6+6k_p\}$ because $1-(1/3)=2/3\equiv 6+6k_p\bmod p$.  Since $\mathfrak{D}_p(\bm\alpha)=(\frac{8}{9},\frac{5}{9},\frac{4}{9})$ and $\mathfrak{D}_p(\bm\beta)=(\frac{2}{3},1,1)$, it follows from Lemma~\ref{lemm_val_p} that $v_p((1/9)_{6+6k_p})=0$, $v_p((4/9)_{6+6k_p})=1$, $v_p((5/9)_{6+6k_p})=1$, and $v_p((1/3)_{6+6k_p})=1$. It is clear that $v_p((1)_{6+6k_p})=0$. Therefore, $$v_p\left(\frac{(1/9)_{6+6k_p}(4/9)_{6+6k_p}(5/9)_{6+6k_p}}{(1/3)_{6+6k_p}(1)_{6+6k_p}^2}\right)=1.$$
%Whence, $6+6k_p\notin S_{\bm\alpha,\bm\beta,p}$. Thus, $S_{\bm\alpha,\bm\beta,p}=\{0\}.$ 
From Example~\ref{exam_2} we know that $E_{\mathfrak{D}_p(\bm\alpha),\mathfrak{D}_p(\bm\beta)}=S_{\mathfrak{D}_p(\bm\alpha),\mathfrak{D}_p(\bm\beta),p}=\{0,3+3k\}.$ Furthermore, it is clear that 2 is the order of $p$ in $(\mathbb{Z}/9\mathbb{Z})^*$ and, according to Remark~\ref{rema_p_l}, $(\bm\alpha,\bm\beta)$ satisfies the $\textbf{P}_{p,2}$ property because $p>18$. Hence, by Lemma~\ref{lemm_explicit_poly}, we get
\begin{equation}\label{eq_sec_9_1}
\mathfrak{g}(z)\equiv P_{0,0}P_{1,0}^p\mathfrak{g}^{p^2}+P_{0,0}P_{1,1}^p\mathfrak{g}_{1,1}^{p^2}\bmod p,
\end{equation}
 %We now prove that $S_{\bm\omega,\bm\eta,p}=\{0\}$. From Lemma~\ref{lemm_aux}, we know that there is a bijective map $\sigma:S_{\bm\omega,\bm\eta,p}\rightarrow S_{\mathfrak{D}_p^2(\bm\alpha),\mathfrak{D}_p^2(\bm\beta),p}$. But, $\mathfrak{D}_p^2(\bm\alpha)=\bm\alpha$ and $\mathfrak{D}_p^2(\bm\beta)=\bm\beta$. Thus, $S_{\bm\omega,\bm\eta,p}=\{0\}$ because we have already seen that $S_{\bm\alpha,\bm\beta,p}=\{0\}$. Finally, $S_{\mathfrak{D}_p(\bm\omega),\mathfrak{D}_p(\bm\eta),p}=S_{\mathfrak{D}_p(\bm\alpha),\mathfrak{D}_p(\bm\beta)}$ because $\mathfrak{D}_p(\bm\omega)=\mathfrak{D}_p(\bm\alpha)$ and $\mathfrak{D}_p(\bm\eta)=\mathfrak{D}_p(\bm\beta)$. Then, it follows that $S_{\mathfrak{D}_p(\bm\omega),\mathfrak{D}_p(\bm\eta),p}=\{0,3+3k\}$. Further, we also have $S_{\bm\omega_{1,0},\bm\eta_{1,0},p}=\{0,3+3k\}=E_{\bm\omega_{1,0},\bm\eta_{1,0},p}$ because $\bm\omega_{1,0}=\mathfrak{D}_p(\bm\alpha)$ and $\bm\eta_{1,0}=\mathfrak{D}_p(\bm\beta)$. 
\begin{equation}\label{eq_sec_9_2}
\mathfrak{g}_{1,1}(z)\equiv R_{1,0}P_{1,0}^p\mathfrak{g}^{p^2}+R_{1,0}P_{1,1}^p\mathfrak{g}_{1,1}^{p^2}\bmod p,
\end{equation}
where $\mathfrak{g}_{1,1}$ is the hypergeometric series ${}_3F_2(\bm\omega,\bm\eta;z)$ with $\bm\omega=((1/9)+1,4/9,5/9)$ and $\bm\eta=((1/3)+1,1,1)$. Multiplying Equation~\eqref{eq_sec_9_1} by $R_{1,0}$ and Equation~\eqref{eq_sec_9_2} by $P_{0,0}$ and subtracting the equations obtained we deduce that $$R_{1,0}\mathfrak{g}\equiv P_{0,0}\mathfrak{g}_{1,1}\bmod p.$$
So $\mathfrak{g}_{1,1}\equiv\frac{R_{1,0}}{P_{0,0}}\mathfrak{g}\bmod p$. By replacing this last equality into  \eqref{eq_sec_9_1} we obtain $$\mathfrak{g}\equiv\left(P_{0,0}P_{1,0}^p+P_{0,0}P_{1,1}^p\left(\frac{R_{1,0}}{P_{0,0}}\right)^{p^2}\right)\mathfrak{g}^{p^2}\bmod p.$$

\end{proof}

\end{document}